\documentclass[12pt]{amsart}

\usepackage{amsmath,amsthm,amsfonts,amscd,amssymb,amsopn,enumerate}
\usepackage{eucal}
\usepackage{mathrsfs,color,hyperref,yhmath}
\usepackage{wasysym}
\usepackage[all]{xy}
\usepackage{fullpage}
\usepackage{tikz-cd}

\newcommand{\mbb}[1]{\mathbb #1}
\newcommand{\mf}[1]{\mathfrak #1}

\newcommand{\mc}[1]{\mathcal #1}
\newcommand{\ms}[1]{\mathscr #1}

\newcommand{\wh}{\widehat}
\newcommand{\Spec}{\operatorname{Spec}}

\newcommand{\im}{\operatorname{im}}

\newcommand{\red}{{\operatorname{red}}}
\newcommand{\Pic}{{\operatorname{Pic}}}

\newcommand{\Frac}{{\operatorname{frac}}}

\newcommand{\ann}{\operatorname{ann}}

\newcommand{\til}[1]{\widetilde{#1}}

\newcommand{\customdiagram}[4]{
\fbox{\xymatrix @C=.33cm @R=-0.3cm {
 & {#2} \ar@/^0.4pc/[dr] \\
{#1} \ar@/^0.4pc/[ur] \ar@/_0.4pc/[dr] & & {#4} \\
 & {#3} \ar@/_0.4pc/[ur]
}}
}

\newcommand{\vardiagram}[2]{
\fbox{\xymatrix @C=.33cm @R=-0.3cm {
 & {#2}_1 \ar@/^0.4pc/[dr] \\
{#2}_0 \ar@/^0.4pc/[ur] \ar@/_0.4pc/[dr] & & {#1} \\
 & {#2}_2 \ar@/_0.4pc/[ur]
}}
}

\newcommand{\mvardiagram}[2]{
\mbox{\xymatrix @C=.33cm @R=-0.3cm {
 & {#2}_1 \ar@/^0.4pc/[dr] \\
{#2}_0 \ar@/^0.4pc/[ur] \ar@/_0.4pc/[dr] & & {#1} \\
 & {#2}_2 \ar@/_0.4pc/[ur]
}}
}

\newcommand{\varPdiagram}[2]{
\fbox{\xymatrix @C=.33cm @R=-0.3cm {
 & {#2}_\UU \ar@/^0.4pc/[dr] \\
{#2}_\BB \ar@/^0.4pc/[ur] \ar@/_0.4pc/[dr] & & {#1} \\
 & {#2}_\PP \ar@/_0.4pc/[ur]
}}
}

\newcommand{\antidiagram}[1]{
\fbox{\xymatrix @C=.33cm @R=-0.3cm {
 & {#1}_1 \ar@/_0.4pc/[dl] \\
{#1}_0 & & {#1} \ar@/_0.4pc/[ul] \ar@/^0.4pc/[dl] \\
 & {#1}_2 \ar@/^0.4pc/[ul]
}}
}

\theoremstyle{plain}
\newtheorem{thm}{Theorem}[section]

\newtheorem{lem}[thm]{Lemma}
\newtheorem{lemma}[thm]{Lemma}
\newtheorem{cor}[thm]{Corollary}
\newtheorem{prop}[thm]{Proposition}

\newtheorem*{thm*}{Theorem}
\newtheorem*{rem*}{Remark}
\newtheorem*{lem*}{Lemma}

\newtheorem*{cor*}{Corollary}
\newtheorem*{prop*}{Proposition}

\theoremstyle{definition}

\newtheorem{ex}[thm]{Example}

\theoremstyle{remark}
\newtheorem{rem}[thm]{Remark}

\newcommand{\sheaf}[1]{\mathscr{#1}}

\newcommand{\PP}{\sheaf{P}}
\newcommand{\UU}{\sheaf{U}}
\newcommand{\BB}{\sheaf{B}}

\newcommand{\fX}{\mathfrak X}
\newcommand{\fU}{\mathfrak U}

\newcommand{\iso}{\ {\buildrel \sim \over \longrightarrow} \ }

\newcommand{\oper}[1]{\operatorname{#1}}

\newcommand{\cha}{\oper{char}}

\newcommand{\Spf}{\oper{Spf}}
\newcommand{\codim}{\oper{codim}}

\usepackage[OT2,T1]{fontenc}

\def\<{\left<}
\def\>{\right>}

\DeclareSymbolFont{cyrletters}{OT2}{wncyr}{m}{n}
\DeclareMathSymbol{\Sha}{\mathalpha}{cyrletters}{"58}

\numberwithin{equation}{section}

\usepackage{fullpage}

\title{Finiteness of formal pushforwards}
\author{David Harbater}
\author{Julia Hartmann}
\author{Daniel Krashen}

\date{May 20, 2026}

\thanks{
\textit{\!\!\!\! Mathematics Subject Classification} (2020): 14D15, 13F25, 13B35 (primary); 13F40 (secondary). 
}

\begin{document}

\maketitle

\begin{abstract}
Under mild hypotheses, given a scheme $U$ and an open subset $V$ whose complement has codimension at least two, the pushforward of a torsion-free coherent sheaf on $V$ is coherent on $U$, and in particular is finite.  We prove an analog of this finiteness assertion in the context of formal schemes over a complete discrete valuation ring, but show that coherence does not always hold.  We then relate this to the problem of gluing 
formal functions, where the patches do not cover the entire scheme.
\end{abstract}

\section{Introduction} \label{intro}

If $j:V \to U$ is an inclusion of an open subscheme of a scheme $U$, then the map $j_*$, which carries sheaves of modules on $V$ to sheaves of modules on $U$, preserves quasi-coherence but not necessarily coherence.  For example, if $U$ is the affine $x$-line over a field $k$, and $V$ is the complement of the origin, then $j_*(\mc O_V)$ is not coherent because its global sections are $k[x,x^{-1}]$, which is not finite over $\mc O(U) = k[x]$.

But for a normal connected quasi-projective 
variety $U$, if the sheaf if torsion-free and the complement of $V$ in $U$ has codimension at least two, then torsion freeness and coherence are preserved under pushforward (see Theorem~\ref{scheme pushforward}, where the hypotheses on $U$ are weaker).  In particular, the pushforward is finite, in the sense that it assigns to each affine open set a finite torsion-free module over the ring of functions.

In this paper, we prove the following finiteness result in the context of formal schemes over a complete discrete valuation ring $T$.

\begin{thm*} [see Theorem~\ref{genl fin ext}]
Let $\ms X$ be a normal connected quasi-projective $T$-scheme,
and let $f:V\hookrightarrow U$ be an inclusion of non-empty open subsets of the reduced closed fiber of $\ms X$ such that the complement of $V$ in $U$ has codimension at least two in $U$.  Write $\frak U, \frak V$ for the formal completions of $\ms X$ along $U,V$. If $\ms F$ is a torsion-free coherent sheaf on $\frak V$, then $\wh f_*(\ms F)$ is a finite torsion-free sheaf on $\mf U$. 
\end{thm*}

Unlike the situation for schemes, the pushforward of formal sheaves does not in general preserve quasi-coherence, and finite pushforwards need not be coherent, as we show in Section~\ref{fin vs coh}.  In that section we also give sufficient conditions for a formal pushforward to be coherent; see
Theorem~\ref{coherent conditions} and its corollaries.

A motivation for studying this question comes from patching problems for modules.
Such problems arise, for example, in the context of an affine open cover of an affine scheme or formal scheme, where one gives compatible finite modules over the rings of functions on these open subsets, and asks for a finite module over the ring of global functions that induces the data compatibly.
Patching problems have been useful in obtaining results in Galois theory and local-global principles; e.g., see \cite{Ha94}, \cite{HH:FP}, \cite{HHK}.  Those papers considered projective curves over complete discrete valuation rings and their function fields.  In that situation, the closed fiber (which is the underlying topological space of the associated formal scheme) can be covered by just two affine open subsets.  As a result, in patching formal modules on open subsets to obtain a global formal module, one can avoid the difficulty of having to satisfy cocycle conditions arising from triple overlaps.  On the other hand, in higher dimensional cases, a quasi-projective variety need not have an open covering by just two affine open subsets.  But on any quasi-projective variety, one can find two affine open subsets such that the complement of their union has codimension two (see Lemma~\ref{exist 2 conn opens}).  If the pushforward from this union is coherent, then the patching problem can be treated without 
having to consider the cocycle conditions.

\smallskip 

{\bf Structure of the manuscript:}  We provide background and context in Section~\ref{background}, followed by two commutative algebra results in Section~\ref{2 lemmas}, and general results on formal schemes and formal patches in Section~\ref{formal}.  Using that material, in Section~\ref{patches} we obtain a key result (Proposition~\ref{patch 2}) that asserts that the intersection of two finitely generated torsion-free formal modules is also finitely generated under a codimension two hypothesis on the complement of the union.  
In Section~\ref{pushforwards} we build on this to prove the finiteness result Theorem~\ref{genl fin ext}, mentioned above.
Pushforward is somewhat more nicely behaved in the case of modules that are reflexive, rather than just being torsion-free.  This is discussed in Section~\ref{reflexive case}; e.g., see Corollary~\ref{reflexive pp cor}.  
Section~\ref{fin vs coh} begins by showing that a formal sheaf is coherent if and only if it is finite and quasi-coherent (Proposition~\ref{coh qcoh fin}), and then showing that 
for formal schemes, unlike for schemes, pushforward need not preserve quasi-coherence (Proposition~\ref{quasicoh}).  The remainder of Section~\ref{fin vs coh} is devoted to conditions and counterexamples concerning coherence of pushforwards, referred to above.
Finally, Section~\ref{p problems} applies these results to patching problems.

\smallskip

{\bf Acknowledgements:} We thank Craig Huneke for helping us with the commutative algebra Lemma~\ref{fglemma}, and Johan de Jong for pointing us to a result in the Stacks Project that yields Theorem~\ref{scheme pushforward}.  We thank V.~Srinivas and H\'el\`ene Esnault for their comments about the distinction between finite and coherent pushforwards, and in particular we thank  V.~Srinivas for bringing Example~\ref{Srinivas ex} to our attention.

\section{Background and context} \label{background}

We begin by fixing some terminology.
Following \cite[Partie 2, Proposition~5.1.2]{EGA4} 
and \cite[Partie 1, Chapter 0, D\'efinition~14.2.1]{EGA4}, 
if $X$ is a scheme then
the {\it codimension} $\codim_X(Y)$ of a closed subscheme $Y \subseteq X$ is the infimum 
of the Krull dimensions of the local rings $\mc O_{X,y}$ over $y \in Y$; this is also the infimum of the codimensions of the irreducible components of $Y$.  Under this definition, the codimension of the empty set is infinite.  For closed subschemes $Z \subseteq Y \subseteq X$, we have $\codim_X(Z) \ge \codim_Y(Z) + \codim_X(Y)$. 

Given a commutative ring $R$ (not necessarily a domain), recall that an $R$-module $M$ is {\it torsion-free} if no regular element of $R$ annihilates any non-zero element of $M$; or equivalently, if $M \to M \otimes_R K$ is injective, where $K$ is the total ring of fractions of $R$.  E.g., see \cite[Section~1]{vasc}.  
As in \cite[Partie 4, 20.1.5]{EGA4}, a sheaf of modules $\ms F$ on a scheme $X$ is {\it torsion-free} if the natural homomorphism $\ms F \to \ms F \otimes_{\mc O_X} \mc M_X$ is injective; here $\mc M_X$ is the sheaf of meromorphic functions on $X$.  This is equivalent to the condition that $\ms F(U)$ is a torsion-free $\mc O_X(U)$-module for every affine open subset $U$ of $X$; thus torsion-freeness is local.  By \cite[Partie 4, Proposition~20.1.6]{EGA4}, being torsion-free is also equivalent to the condition that every associated point of $\ms F$ is an associated point of $\mc O_X$.  (Recall from \cite[Partie~2, D\'efinition~3.1.1]{EGA4} that a point $x$ of $X$ is an {\it associated point of} $\ms F$ if the maximal ideal $\frak m_x \subset \mc O_{X,x}$ is an associated prime of the $\mc O_{X,x}$-module $\ms F_x$; i.e., is the annihilator of an element of $\ms F_x$.)  The {\it associated points of} $X$ are by definition the associated points of $\mc O_X$.  Thus the associated points of an open subset $U \subseteq X$ are the associated points of $X$ that lie in $U$.

Our results about formal schemes will rely on assertions in earlier sections that have several mild but technical hypotheses concerning Noetherian rings, related to the notion of excellence.  For convenience, we list those conditions here, along with references for details.  First, 
a {\it G-ring} is a Noetherian ring $R$ such that the map $R_{\mf p} \to \wh{R_{\mf p}}$ is regular for every prime ideal $\mf p$ of $R$, where $\wh{R_{\mf p}}$ is the completion of the local ring $R_{\mf p}$.  A ring is {\it quasi-excellent} if it is a G-ring that also satisfies what is known as the J-2 property (see \cite[32.B]{Mats} for the definition).
By~\cite[Theorem~78]{Mats}), quasi-excellent rings satisfy the condition of being Nagata rings (a class of Noetherian rings defined at \cite[31.A]{Mats}).  
A ring is {\it excellent} if it is quasi-excellent and in addition has the property of being 
universally catenary (defined at \cite[14.B]{Mats}); see also~\cite[Definition~07QT]{stacks}.  
Noetherian complete local rings are excellent, and excellence is preserved under localizing and under passage to a finitely generated algebra (see \cite[Section~34]{Mats}).
A scheme is {\it excellent} if it can be covered by affine open subsets $U_i$ such that each of the rings $\mc O_X(U_i)$ is excellent (see \cite[Definition~8.2.35]{Liu}); these schemes are automatically locally Noetherian.  One similarly defines schemes that are Nagata, are universally catenary, etc.

If $f:V \to U$ is a quasi-compact and quasi-separated morphism of schemes (e.g., an inclusion of Noetherian schemes), and if $\ms F$ is quasi-coherent on $V$, then $f_*(\ms F)$ is quasi-coherent on $U$ (see \cite[Lemma~01LC]{stacks}). 
For coherent modules, there is the following result, which is known to the experts, and which is essentially a special case of \cite[Partie 2, Corollaire~5.11.4(ii)]{EGA4} 
and \cite[Lemma~0AWA]{stacks} 
(as Johan de Jong pointed out to us).  Note that this theorem holds in particular in the case mentioned in the introduction, viz., of a normal connected quasi-projective scheme $U$, since normal (and integral) schemes are reduced, and since quasi-projective varieties are excellent (by \cite[Section~34]{Mats}).

\begin{thm} \label{scheme pushforward}
Let $U$ be a scheme that is Nagata and universally catenary, and such that each associated point has codimension zero (e.g., $U$ is an excellent reduced scheme).
Let $j:V \hookrightarrow U$ be the inclusion of 
an open subset such that the complement of $V$ in $U$ has codimension at least two in $U$.  Then for any torsion-free coherent sheaf $\ms F$ on $V$, the pushforward $j_*(\ms F)$ is a torsion-free coherent sheaf on $U$.
\end{thm}

\begin{proof}
We first observe that any excellent reduced scheme $U$ satisfies the hypotheses of the theorem.  Namely, such a $U$ is locally Noetherian since it is excellent; and since $U$ is also reduced it then follows from \cite[Lemmas~0EMA, 05AR]{stacks} 
that each associated point of~$U$ has codimension zero.  Also, as noted above, excellence implies the conditions of being Nagata and universally catenary.  So an excellent reduced $U$ satisfies the above hypotheses.

The hypotheses on $U$ in the theorem are inherited by $V$, since $V$ is an open subset of~$U$.  And as  
noted above, for $\ms F$ a torsion-free coherent sheaf on $V$, the associated points of $\ms F$ are also associated points of $\mc O_V$, or equivalently of $V$.  Thus the associated points of $\ms F$ have codimension zero in $V$.

By hypothesis, the complement $Z$ of $V$ in $U$ has codimension at least two in $U$.  Also, by the previous paragraph, for 
every associated point $x$ of $\ms F$, the closure $\overline{\{x\}}$ of $\{x\}$ in $U$ is an irreducible component of $U$.  It then follows that for every associated point $x$ of $\ms F$, the codimension of $Z \cap \overline{\{x\}}$ in $\overline{\{x\}}$ is at least $2$; or equivalently, $\dim(\mc O_{\overline{\{x\}},z})\ge 2$ for every $z \in Z \cap \overline{\{x\}}$.

As a consequence, since $U$ is Nagata and universally catenary, we obtain that $j_*(\ms F)$ is coherent on $U$, by applying \cite[Lemma~0AWA]{stacks} 
(or alternatively \cite[Partie 2, Corollaire~5.11.4(ii)]{EGA4}; see also \cite[Proposition~0334]{stacks}). 

It remains to show that $j_*(\ms F)$ is torsion-free on $U$; i.e., $j_*(\ms F)(O)$ is a torsion-free $\mc O_U(O)$-module for every affine open subset $O$ of $U$.
After replacing $U$ by an arbitrary such $O$ and replacing $V$ by $O \cap V$, 
we may assume that $U$ is an affine scheme $\Spec(R)$,  
and prove that $j_*(\ms F)(U)$ is a torsion-free $R$-module.  
By \cite[Lemmas~05AR, 05C3]{stacks},   
the set of zero-divisors in $R$ is the union of the 
associated primes of $R$; and by hypothesis this is  
the set of elements of $R$ that vanish at the generic point of some irreducible component of $U$ (i.e., at one of the points of codimension zero).
For the same reason, the corresponding statement holds for the zero divisors in $R'$, where $\Spec(R')$ is any affine open subset of $U$.

Now let $m$ be a non-zero element of $M:= j_*(\ms F)(U) = \ms F(V)$ and let $r$ be a regular element of $R$.  We wish to show that $rm \ne 0$.  Since $m$ is non-zero in $\ms F(V)$, and since $\ms F$ is coherent, there is a non-empty affine open subset $V' = \Spec(R') \subseteq V \subseteq U$ such that the restriction $m'$ of $m$ from $V$ to $V'$ is non-zero in $\ms F(V')$.  
Since $r$ is regular in $R$ (i.e., not a zero-divisor), it does not vanish at the generic point of any irreducible component of $U$, by the previous paragraph.  Thus the image $r'$ of $r$ in $R'$ does not vanish at the generic point of any irreducible component of $V' = \Spec(R')$ (since the latter set of generic points is contained in the former set).

So $r'$ is a regular element of $R'$, by the previous paragraph applied to $R'$.  But $\ms F(V')$ is a torsion-free module over $R'= \mc O_V(V')$, since $\ms F$ is a torsion-free sheaf on $V$.  Hence $r'm' \ne 0$ in $\ms F(V')$.  Since $r'm'$ is the image of $rm$ under the restriction map $M = \ms F(V) \to \ms F(V')$, it follows that $rm \ne 0$ in $M$, as needed.
\end{proof}

\begin{ex} \label{torsion pushforward ex}
To illustrate the role of the torsion-free condition on coherent sheaves here (or more generally, the condition on associated points), let $V$ be the complement of the origin in the affine $x,y$-plane $U$ over a field $k$, with inclusion $j:V \hookrightarrow U$, and let $\ms F = j^*(\mc O/\mc I)$, where~$\mc I$ is the sheaf of ideals on $U$ induced by the ideal $(x) \subset k[x,y] = \mc O(U)$.  The pushforward $j_*\ms F$ is not coherent on $U$, since its global sections are $k[y,y^{-1}]$.  Here the complement $Z$ of $V$ in $U$ has codimension two, but $\ms F$ is not torsion-free, since it is $x$-torsion, with $(x)$ an associated point.  Moreover $Z$ is of codimension one (not two) in the closure of the associated point $(x)$. 
\end{ex}

\smallskip

In Section~\ref{pushforwards}, we study the analogous situation for torsion-free coherent sheaves  on formal schemes over a complete discrete valuation ring $T$. There, we consider formal schemes $\mf V \subseteq \mf U$ having underlying sets $V \subseteq U$, with $V$ open in $U$ such that the complement of $V$ in $U$ has codimension at least two.  In Theorem~\ref{genl fin ext}, we prove in this situation that the pushforward from $\mf V$ to $\mf U$ is torsion-free and finite.  Here the torsion-free hypothesis means that the sections over each affine open set $V$ of the underlying space $U$ form a torsion-free module over $\mc O_{\mf U}(V)$.  Without the torsion-free assumption, one can construct counterexamples to finiteness similar to the one above, by taking the $t$-adic completion of the base change of the above example from $k$ to $k[[t]]$.

The proof of finiteness for pushforwards of formal coherent sheaves is more involved than the proof over schemes.  Namely, suppose we are given
a torsion-free coherent sheaf $\ms F$ on a formal scheme $\mf U$ as above, with $U_n$ being the $n$-th thickening of the reduced closed fiber. 
It is tempting to try to apply the scheme-theoretic result \cite[Lemma~0AWA]{stacks} 
(or \cite[Partie 2, Corollaire~5.11.4(ii)]{EGA4}) to the pullback $\ms F_n$ of $\ms F$ to each $U_n$, and to use that a coherent sheaf on $\mf U$ corresponds to 
an inverse system of compatible coherent sheaves on the schemes~$U_n$ (see \cite[Th\'eor\`eme~10.11.3]{EGA1}).
The difficulty is that $\ms F_n$ need not be torsion-free, and may have new associated points of positive codimension in $U_n$; and this would prevent the use of the above results.  (See also Remark~\ref{rems on fin pushforward}(\ref{rem reductions not tor free}).)  Instead, in Section~\ref{patches}, we follow a strategy that relies on the commutative algebra lemmas proven in Section~\ref{2 lemmas}; and we build on that in proving Theorem~\ref{genl fin ext}. 

\section{Two general lemmas} \label{2 lemmas}

Before turning to formal schemes, we prove two general results.  
The proof of the first lemma was outlined for us by Craig Huneke in the case that $I$ is prime.

\begin{lem} \label{fglemma} 
Let $R$ be a $G$-ring that is a normal domain, let $I$ be a proper ideal in $R$,
and let $M$ be a finitely generated torsion-free $R$-module.  Let $P_1,\dots,P_s$ be the minimal primes over $I$.
\begin{enumerate} [(a)]
\item For every $i \ge 0$ there is an $n \ge 0$ such that
$M \cap P_1^nM_{P_1} \cap \cdots \cap P_s^nM_{P_s} \subseteq I^iM$.
\item In particular, for every integer $c \ge 0$ there is some $n \ge 0$
such that if $r\in R$ and $m\in M$ satisfy $rm \in I^nM$ then either $r \in P_j$ for some $j$ or 
$m \in I^cM$.
\end{enumerate}
\end{lem}

\begin{proof}
The radical $\sqrt I$ of $I$ is the ideal $P_1 \cap \cdots \cap P_s$, and 
by \cite[Proposition~7.14]{AM} there is an integer $\alpha$ such that $\sqrt I^{\,\alpha} \subseteq I$.  Thus for part~(a), it suffices to prove the assertion with~$I$ replaced by~$\sqrt I$.  So we will assume that $I$ is the intersection of the prime ideals $P_i$,  
and will proceed by induction on $s$.

If $s=1$, then $I$ is a prime ideal $P$.  
First consider
the special case that $M=R$.   
In this situation, for each positive integer $n$, $M \cap P^nM_P$ is just the $n$-th symbolic power $P^{(n)} := R \cap P^nR_P$ of $P$.
Since $R$ is a normal G-ring, the completion $\wh R_Q$ of $R$ at each prime ideal $Q \subset R$ is also normal, by \cite[33.I]{Mats}.  
Since $\wh R_Q$ is normal and local, it is a domain, and its only associated prime is $(0)$.  Since this holds for all $Q$, \cite[Theorem~1]{Schen} asserts that the $P$-adic topology on $R$ defined by the ideals $P^n$ is equivalent to the $P$-symbolic topology defined by the ideals $P^{(n)}$.  
(Namely, the condition in part (ii) of that theorem holds because the annihilator ideals $Q$ considered there properly contain $P$, and the only associated prime of the complete local ring at $Q$ is $(0)$.)
Hence part~(a) follows in this special case.

Next, still with $s=1$ and $I=P$, consider a more general finitely generated torsion-free $R$-module $M$.  
By \cite[Lemma~0AUU]{stacks}, 
$M$ is contained in a finitely generated free $R$-module~$E$.  By the Artin-Rees lemma (e.g., \cite[Lemma~00IN]{stacks}), 
there is a positive integer $d$ such that for every $e \ge d$, $M \cap P^eE = P^{e-d}(M \cap P^dE) \subseteq P^{e-d}M$. Take $i\ge 0$.
By the previous paragraph, there exists $n\ge 0$ such that $R \cap P^nR_P \subseteq P^{i+d}$.  Thus the free module $E$ satisfies
$E \cap P^nE_P \subseteq P^{i+d}E$.
Here $M \subseteq E$ and so $M_P \subseteq E_P$.  Hence
\[M \cap P^nM_P = M \cap E \cap P^nM_P\subseteq M \cap E \cap P^nE_P \subseteq M \cap P^{i+d}E \subseteq P^iM,\]  at the last step using Artin-Rees with $e=i+d$.  This proves the case $s=1$.

For the inductive step, take $I = P_1 \cap \cdots \cap P_s$, and assume that the assertion holds for
$J := P_1 \cap \cdots \cap P_{s-1}$.  Here $I = P_s \cap J$.
We will prove that for every $i$ there is an $n$ such that 
$M \cap P_1^nM_{P_1} \cap \cdots \cap P_s^nM_{P_s} \subseteq I^iM$. 
So take some $i \ge 0$.
By the inductive hypothesis, there is an $n'\ge 0$ such that 
$M \cap P_1^{n'}M_{P_1} \cap \cdots \cap P_{s-1}^{n'}M_{P_{s-1}} \subseteq J^iM$. 
By the above case of $s=1$ applied to the finitely generated torsion free module $J^iM$ and the ideal $P_s$, 
there is some $m \ge 0$ such that $J^iM \cap P_s^mJ^iM_{P_s} \subseteq P_s^iJ^iM$.
Since $P_1,\dots,P_s$ are the (distinct) minimal primes over $I$, no $P_j$ is contained in $P_s$ for $j<s$.  Thus $J = P_1 \cap \cdots \cap P_{s-1}$ is also not contained in $P_s$, by \cite[Proposition~1.11(ii)]{AM}.  Hence
$JR_{P_s}$ is the unit ideal of $R_{P_s}$, and
$J^iM_{P_s} = M_{P_s}$; 
so $J^iM \cap P_s^mM_{P_s} 
= J^iM \cap P_s^mJ^iM_{P_s} 
\subseteq P_s^iJ^iM 
= (P_sJ)^iM 
\subseteq (P_s \cap J)^iM 
= I^iM$.
Thus \[M \cap P_1^{n'}M_{P_1} \cap \cdots \cap P_{s-1}^{n'}M_{P_{s-1}} \cap P_s^mM_{P_s} 
\subseteq J^iM \cap P_s^mM_{P_s} \subseteq I^iM.\]
Let $n = \max(n',m)$.  Then 
\[M \cap P_1^nM_{P_1} \cap \cdots \cap P_s^nM_{P_s} \subseteq M \cap P_1^{n'}M_{P_1} \cap \cdots \cap P_{s-1}^{n'}M_{P_{s-1}} \cap P_s^mM_{P_s} 
\subseteq I^iM,\] and this concludes the inductive proof of part~(a).

For part~(b), let $n \ge 0$ be associated to the value $i=c$ as in part~(a).   
Suppose that $r\in R$ and $m\in M$ satisfy $rm \in I^nM$.  
Thus $rm \in P_j^nM_{P_j}$ for all $j$.
If $r$ does not lie in any $P_j$, then 
$r$ is a unit in each $R_{P_j}$ and so $m \in P_j^nM_{P_j}$ for all $j$.
Hence 
$m \in M \cap P_1^nM_{P_1} \cap \cdots \cap P_s^nM_{P_s} \subseteq 
I^cM$ by part~(a).
\end{proof}

\begin{lem} \label{mnqlemma}  
Let $R$ be a normal G-ring
that is complete with respect to a non-zero principal ideal $I = (t)$, and let $M$ be a non-zero finitely generated torsion-free $R$-module.  Let $P_1,\dots,P_s$ be the minimal primes over $I$, and for each $j=1,\dots,s$ and $i \ge 1$ write $P_jR_i$ for the image of $P_j$ in $R_i := R/I^i$.  For each $i \ge 1$ also write $M_i = M/I^iM$, and let $Q_i$ be the set of elements $q$ of the $R_i$-module $M_i$ such that $\ann(q)$ is not contained in any of the ideals $P_1R_i, \cdots, P_sR_i$.
Then the following hold.
\begin{enumerate}[(a)]
\item $Q_i$ is an $R_i$-submodule of $M_i$, and each $q \in Q_i$ satisfies $\ann(q) \not\subseteq P_1R_i \cup \cdots \cup P_sR_i$.
\item Every associated prime of the $R_i$-module $N_i := M_i/Q_i$ is of the form $P_jR_i$ with $1 \le j \le s$.
\item The inverse system $\{M_i\}$ induces inverse systems $\{Q_i\}$ and $\{N_i\}$ by restriction and quotient.
\item If $m_i \in M_i$ and $t^dm_i \in Q_i$ for some $d<i$, then the image of $m_i$ in $M_{i-d}$ lies in $Q_{i-d}$.
\item There is a positive integer $n$ such that for every $i$, $Q_{i-1+n} \to Q_i$ is the zero map.
\item $\displaystyle \lim_\leftarrow M_i = M$, $\displaystyle \lim_\leftarrow Q_i = 0$, and $\displaystyle \lim_\leftarrow N_i = M$.
\end{enumerate}
\end{lem}

\begin{proof}
Recall that every Noetherian normal ring is a finite product of Noetherian normal domains; see \cite[Lemma~030C]{stacks}. 
Hence we may write $R \cong R^{(1)} \times \cdots \times R^{(s)}$, where each factor is a Noetherian normal domain; and correspondingly, we have
$M_i \cong M_i^{(1)} \times \cdots \times M_i^{(s)}$, 
$Q_i \cong Q_i^{(1)} \times \cdots \times Q_i^{(s)}$, and $N_i \cong N_i^{(1)} \times \cdots \times N_i^{(s)}$.  Here 
the set of associated primes of~$N_i$ is the union of the sets of associated primes of $N_i^{(1)}, \dots, N_i^{(s)}$.  Thus 
in order to prove the lemma in general, it suffices to prove it in the special case in which $R$ is a domain, by 
applying the special case to each factor.  Here, for the proof of part~(e), we can take $n$ to be the maximum of the values $n^{(1)},\dots,n^{(s)}$ corresponding to the factors.

So for the remainder of the proof we assume that $R$ is a domain.

Since each ideal $P_j \subset R$ is prime and contains $I$, it follows that $P_jR_i \subset R_i = R/I^i$ is also prime.  Thus if $q \in Q_i$ then $\ann(q)$ is not contained in $\Pi_i := P_1R_i \cup \cdots \cup P_sR_i$, by prime avoidance.  Now take $q_1,q_2 \in Q_i$.  
Since $\ann(q_i) \not\subseteq \Pi_i$, there
exist elements $r_1,r_2 \in R_i \smallsetminus \Pi_i$ that annihilate $q_1,q_2$ respectively.  So $r_1r_2 \in R_i$ is not in $\Pi_i$ and it annihilates $q_1+q_2$.  Also, for any $r \in R_i$, the above element $r_1 \in R_i \smallsetminus \Pi_i$ annihilates $rq_1$.  Hence $q_1+q_2$ and $rq_1$ lie in $Q_i$.  So $Q_i$ is an $R_i$-submodule of $M_i$, proving~(a).

Next, let $\bar m = m + Q_i \in N_i$, with $m \in M_i$.  If there exists $r \in R_i$ that is not in $\Pi_i$ and such that $rm \in Q_i$, then by part~(a) there exists $s \in R_i$ such that $srm = 0 \in M_i$ and $s$ is not in $\Pi_i$.  But then $sr \in R_i$ is also not in $\Pi_i$; hence $m \in Q_i$ and thus $\bar m=0$.  This shows that the annihilator of each non-zero element of $N_i$ is contained in $\Pi_i$, and hence is contained in one of the ideals $P_1R_i, \cdots, P_sR_i$ by prime avoidance.  That is, every associated prime of $N_i$ is contained in some $P_jR_i$, $j=1,\dots,s$.  But each $P_j$ is a minimal prime over $I$ and hence over $I^i$; thus $P_jR_i$ is a minimal prime of $R_i$.  
Therefore every associated prime of the $R_i$-module $N_i$ is among $P_1R_i, \cdots, P_sR_i$, proving (b).

Since $(t) \subseteq P_i$, the surjection $M_i \to M_{i-1}$ restricts to a map $Q_i \to Q_{i-1}$, and so the 
modules $Q_i$ form an inverse system.
It follows that
the maps $Q_i \to Q_{i-1}$ yield well-defined surjections $N_i \to N_{i-1}$, so that the modules $N_i$ also form an inverse system.
This proves (c).

By \cite[Lemma~00MA, (3)]{stacks}, 
we have $\displaystyle M = M \otimes_R R = M \otimes_R \lim_\leftarrow R_i = \lim_\leftarrow M_i$.  This proves the first part of (f).

For part (d), by induction we are reduced to the case that $d=1$, with $i \ge 2$.  So 
suppose that $m_i \in M_i$ and $tm_i \in Q_i$.  Let $m \in M$ be an element such that $m_i$ is the image of $m$ in $M_i$.  By definition of $Q_i$, there exists $r_i \in R$ such that $r_i tm \in t^iM$ and $\bar r_i \not\in P_1R_i,\dots,P_sR_i$, where $\bar r_i \in R_i$ is the image of $r_i$.  Write 
$r_i tm = t^im'$ for some $m' \in M$.  Since $M$ is torsion-free, and since the non-zero element $t$ is regular (because $R$ is a domain), it follows that 
$r_im = t^{i-1}m'$.
Note that $r_i \not\in P_1,\dots,P_s$; and since each $P_j$ contains $t$, the image of $r_i$ in $R_{i-1}$ is not in any $P_jR_{i-1}$.
Hence the image of $m_i$ in $M_{i-1}$ lies in $Q_{i-1}$.  This proves (d) in the case $d=1$, and hence in the general case.  

Next, we show that (e) holds for the integer $n$ obtained by setting $c=1$ in Lemma~\ref{fglemma}(b).  We first treat the case where $i=1$; i.e., we show that the image of $Q_n \to Q_1$ is trivial.  Namely, given a non-zero element $m_n \in Q_n \subseteq M_n$, we may choose $m \in M$ lying over $m_n$; and 
then there exists $r \in R$ such that $rm \in t^n M$ and $r \not\in P_1,\dots,P_s$ (as in the previous paragraph).    
By the defining property of $n$, it follows that $m \in t^cM = tM$.  Hence the image of $m$ in $M_1$ is trivial.  But this element is the same as the image of $m_n$ in $Q_1$; and so this proves (e) in the case $i=1$.  

For a more general value of $i$ in the assertion of (e), 
suppose for the sake of contradiction that $m_{n+i-1} \in Q_{n+i-1} \subseteq M_{n+i-1}$ is an element whose image $m_i \in Q_i \subseteq M_i$ is non-zero.  Pick a representative $m \in M$ of $m_{n+i-1}$.  Then 
$m \not\in t^iM$.  So there is a maximum integer $d \ge 0$ such that $m \in t^d M$, and $d < i$.
Thus we may write $m=t^d m'$ for some $m' \in M$ such that $m' \not\in tM$.  
Let $m'_{n+i-1}$ be the image of $m'$ in $M_{n+i-1}$.  
Thus $t^d m'_{n+i-1} = m_{n+i-1} \in Q_{n+i-1} \subseteq M_{n+i-1}$; and so the image $m'_{n+i-1-d}$ of $m'_{n+i-1}$ in $M_{n+i-1-d}$ lies in $Q_{n+i-1-d}$, by part (d).  Let $m_n' \in Q_n \subseteq M_n$ and $m_1' \in Q_1 \subseteq M_1$ be the images of $m'_{n+i-1-d}$.  (Note that $n+i-1-d \ge n \ge 1$.)  Thus $m_1'$ is the image of $m_n'$; and $m_1' \ne 0 \in M_1 = M/tM$ because $m' \not\in tM$.  But by the previous paragraph, the image of $Q_n \to Q_1$ is trivial.  This contradiction proves (e). 

Part (e) implies that $\displaystyle \lim_\leftarrow Q_i = 0$, which is the second part of (f).  For the third part of (f), note that part (e) implies that 
the inverse system $\{Q_n\}$ satisfies the Mittag-Leffler condition 
(see \cite[Section~0594] {stacks}).  
Since $0 \to Q_i \to M_i \to N_i \to 0$ is exact, it then follows from 
\cite[Lemma~0598]{stacks} 
that $\displaystyle 0 \to \lim_\leftarrow Q_i \to M \to \lim_\leftarrow N_i \to 0$ is exact.  Since $\displaystyle \lim_\leftarrow Q_i=0$, 
the map $\displaystyle M \to \lim_\leftarrow N_i$ is an isomorphism,
as asserted.  This completes the proof in the case that~$R$ is a domain, and thus also in the general case.
\end{proof}

\section{Formal schemes and patches} \label{formal}

Let $T$ be a complete discrete valuation ring with uniformizer $t$, and let $\ms X$ be an integral normal $T$-scheme of finite type having function field $F$.  Let $X := \ms X_s^{\red}$ be the reduced closed fiber of $\ms X$, where $\ms X_s$ is the fiber of $\ms X$ over the closed point $s$ of $\Spec(T)$.
Given an open subset $\ms U \subset \ms X$, we may consider the $t$-adic completion $\wh{\mc O_{\ms X}(\ms U)}$ of the ring $\mc O_{\ms X}(\ms U)$.  Also, given any subset $U$ of $X$, we may take the subring $\mc O_{\ms X,U} = \bigcap_{P \in U} \mc O_{\ms X,P}$ of $F$ consisting of the rational functions on $\ms X$ that are regular at every point of $U$; this is normal since each local ring $\mc O_{\ms X,P}$ is.  We write $\wh{\mc O}_{\ms X,U}$ for its $t$-adic completion.

Consider the formal scheme $\mf X = \ms X_{/X}$ obtained by completing $\ms X$ along $X$, as in \cite[Section 10.8]{EGA1}.  
The underlying topological space of the ringed space $\mf X$ is $X$; and 
the structure sheaf $\mc O_{\mf X}$ is the 
inverse limit of the $\mc O_{\ms X}$-modules $\mc O_{X_n}$, where $X_n$ is the fiber of $\ms X$ over $\Spec(T/(t^n))$.  This inverse limit is defined because the morphisms $X \to X_n \to X_{n+1}$ are homeomorphisms by \cite[5.1.2, 5.1.3]{EGA1}, and so the underlying spaces may be identified.  Similarly, we may identify each open subset $U \subseteq X$ with an open subset $\mf U \subseteq \mf X$ as topological spaces (though not as ringed spaces).  With $U$ corresponding to $\mf U$, we will often write $\mc O_{\mf X}(U)$ for the $t$-adically complete ring $\mc O_{\mf X}(\mf U)$.  Similarly, if $\ms F$ is a sheaf of $\mc O_{\mf X}$-modules, we may write $\ms F(U)$ for $\ms F(\mf U)$.  Here the structure sheaf of $\mf U$ is the restriction of that of $\mf X$; and so for an open subset $V \subseteq U$, we have $\mc O_{\mf U}(V) = \mc O_{\mf X}(V)$.

If $U$ is an affine open subset of $X$, then the corresponding open subset $U_n \subseteq X_n$ is also affine, by \cite[Proposition~5.1.9]{EGA1}.  Since the sheaf and presheaf inverse limits of sheaves coincide, $\displaystyle \mc O_{\mf X}(U) = \lim_{\leftarrow} \mc O_{X_n}(U_n)$; here $\mc O_{X_n}(U_n) = \mc O_{\mf X}(U)/(t^n)$.  
Similarly, $\mc O_X(U) = \mc O_{\mf X}(U)/I$, where $I$ is the radical of the ideal $t\mc O_{\mf X}(U)$.  
In this situation, we will often write $\wh R_U$ for the ring $\mc O_{\mf X}(U)$.  By part~(c) of Proposition~\ref{RU versions} below, 
this generalizes the notation used in \cite{HH:FP}, \cite{HHK}, and later papers, where $\wh R_U$ was used for the ring $\wh{\mc O}_{\ms X,U}$ in the case of a projective normal $T$-curve $\ms X$.

\begin{prop} \label{RU versions}
Let $T$ be a complete discrete valuation ring with uniformizer $t$, let $\ms X$ be a normal integral $T$-scheme of finite type, and let
$\mf X$ be the formal completion of $\ms X$ along its reduced closed fiber $X$.  Let $U$ be a non-empty affine open subset of $X$.
\begin{enumerate} [(a)]
\item 
The natural map $\wh{\mc O}_{\ms X,U} \to \mc O_{\mf X}(U)$ is injective.
\item 
Suppose that $\ms U$ is an affine open subset of $\ms X$ such that $\ms U \cap X = U$.  Then the natural maps $\wh{\mc O_{\ms X}(\ms U)} \to \wh{\mc O}_{\ms X,U} \to \mc O_{\mf X}(U)$ are isomorphisms.
\item
If $\ms X$ is a normal 
projective $T$-curve, then such a $\ms U$ exists, and so 
the natural map $\wh{\mc O}_{\ms X,U} \to \mc O_{\mf X}(U)$ is an isomorphism.
\end{enumerate}
\end{prop}

\begin{proof}
The assertion is trivial if $\ms X$ consists just of the fiber over the closed point of $\Spec(T)$, and so we may assume that $t$ is a non-zero element of the function field $F$ of $\ms X$.

As above, let $X_n$ be the fiber of $\ms X$ over $\Spec(T/t^n)$, and let $U_n$ be the affine open subset of $X_n$ corresponding to $U$ under the homeomorphism $X \to X_n$.
Since taking inverse limits is left exact, in order to prove part~(a), it suffices to show injectivity modulo $t^n$ for all $n$.  So take $f \in \wh{\mc O}_{\ms X,U}/(t^n) = \mc O_{\ms X,U}/(t^n)$ that lies in the kernel of the map to $\mc O_{\mf X}(U)/(t^n)$.  Let $\til f \in \mc O_{\ms X,U} \subseteq F$ be an element that maps to $f$.  Thus the restriction of $\til f$ to $U_n$ is zero.  
Hence for every generic point $\eta$ of $U_n$, the image of $\til f$ in 
$\mc O_{\ms X,\eta}$ lies in the ideal $(t^n)$,  and so the element $g:=\til f/t^n \in F$ lies in $\mc O_{\ms X,\eta} \subseteq F$.  
Now for every point $P \in U$, if $\mf p$ is a height one prime of $\mc O_{\ms X,P}$, then the localization $(\mc O_{\ms X,P})_{\mf p}$ is either of the form $\mc O_{\ms X,\eta}$ for some generic point $\eta$ of $U_n$ as above (if $t \in \mf p$), or else of the form $\mc O_{\ms X,Q}$ for some codimension one point $Q$ of $\ms X$ 
that is not a generic point $\eta$ and whose closure meets~$U$ (if $t \not\in \mf p$).  In either case $g$ lies in $(\mc O_{\ms X,P})_{\mf p}$, in the latter case using that $\til f \in \mc O_{\ms X,P}$ and that $t$ is a unit in $(\mc O_{\ms X,P})_{\mf p}$.  
Since $\mc O_{\ms X,P}$ is a normal Noetherian domain, it follows from \cite[Corollary~11.4]{Eis} that $g \in \mc O_{\ms X,P}$ for each $P \in U$.  Hence $g \in \bigcap_{P \in U} \mc O_{\ms X,P} = \mc O_{\ms X,U}$.  Thus $\til f =t^ng \in t^n\mc O_{\ms X,U}$, and so $f=0$, yielding part~(a).

In part~(b), since $\ms U$ is an affine open subset of $\ms X$ such that $\ms U \cap X = U$, we have that
$\mc O_{\ms X}(\ms U)/(t^n) = \mc O_{X_n}(U_n)$; and taking inverse limits yields that the map $\wh{\mc O_{\ms X}(\ms U)} \to \mc O_{\mf X}(U)$ is an isomorphism. 
The inclusion $U_n \to \ms U$ induces a map $\mc O_{\ms X}(\ms U)/(t^n) \to \mc O_{\ms X,U}/(t^n) = \wh{\mc O}_{\ms X,U}/(t^n)$.
Since $\mc O_{\ms X}(\ms U)/(t^n) \to \mc O_{\mf X}(U)/ (t^n)$ factors through $\mc O_{\ms X}(\ms U)/(t^n) \to \mc O_{\ms X,U}/(t^n)$, 
by taking inverse limits we find that the isomorphism $\wh{\mc O_{\ms X}(U)} \to \mc O_{\mf X}(U)$ factors through $\wh{\mc O_{\ms X}(U)} \to \wh{\mc O}_{\ms X,U}$.  Hence the map $\wh{\mc O}_{\ms X,U} \to \mc O_{\mf X}(U)$ is surjective.  
So by part~(a) this map is an isomorphism, concluding the proof of part~(b).

To prove (c) we will show that an affine open subset $\ms U \subseteq \ms X$ as above exists 
when $\ms X$ is a normal projective $T$-curve.
First consider the case where $U$ is dense in $X$, so that 
its complement $S$ in $X$ is finite.  Since $U$ is affine, this complement meets each irreducible component of $X$.  By \cite[Proposition~3.3]{HHK:H1}, there is a finite morphism $\varphi:\ms X \to \mbb P^1_T$ such that $S$ is the inverse image of the point at infinity on the closed fiber $\mbb P^1_k$ (where $k$ is the residue field of $T$).  Thus $U$ is the inverse image of $\mbb A^1_k$.  We may then take $\ms U \subset \ms X$ to be the inverse image of $\mbb A^1_T$.  This is affine because the morphism $\varphi$ is finite and hence an affine morphism.

For the proof of (c) in the more general case where $U$ is not necessarily dense in $X$, let $J$ be the set of irreducible components of $X$ that do not meet $U$.
Since $U$ is non-empty, $J$ does not contain every irreducible component of $X$.  Thus by \cite[Section 6.7, Theorem~1, Corollary~3, Proposition~4]{BLR}, we may contract the components in $J$.  That is, there is a proper birational morphism $\pi:\ms X \to \ms Y$, where $\ms Y$ is a projective normal $T$-curve, such that the components of $J$ each map to a point, and $\pi$ is an isomorphism elsewhere.  Thus $U$ maps isomorphically onto its image $V$, which is dense in the reduced closed fiber $Y$ of $\ms Y$.  So by the above special case, there is an affine open subset $\ms V \subseteq \ms Y$ such that $\ms V \cap Y = V$.  The inverse image $\ms U = \pi^{-1}(\ms V)$ is isomorphic to $\ms V$, and so it is an affine open subset of $\ms X$.  Moreover its intersection with $X$ is $U$.  So $\ms U$ is as asserted. 
\end{proof}

\begin{rem}
\begin{enumerate}[(a)]
\item
The injective map $\wh{\mc O}_{\ms X,U} \to \mc O_{\mf X}(U)$ in Proposition~\ref{RU versions}(a) need not be surjective in general, and so the hypothesis in Proposition~\ref{RU versions}(b) cannot be dropped.
See Remark~\ref{different RUs} for an explicit example.
\item
In Proposition~\ref{RU versions}(c), once we reduce as above to the case that $U$ is dense, we can explicitly construct $\ms U$ as follows (following the proof of the result \cite[Proposition~3.3]{HHK:H1} that was cited above):
At each closed point $P \in S = X \smallsetminus U$, take an element $r_P$ in the maximal ideal of the
local ring $\mc O_{\ms X,P}$ such that $r_P$ does not vanish along any component of the closed fiber passing through $P$.  This defines an effective Cartier divisor on $\Spec(\mc O_{\ms X,P})$ whose support passes through $P$, and which is the restriction of an effective Cartier divisor $\ms D_P$ on $\ms X$ whose support meets $X$ precisely at $P$.  Here $\ms D := \sum_{P \in S} \ms D_P$ is an effective Cartier divisor on $\ms X$ whose support meets $X$ precisely at $S$, and so in particular meets each irreducible component of $X$.  Hence the restriction $D$ of $\ms D$ is $X$ is ample (by
\cite[Chapter~7, Proposition~5.5]{Liu}), and thus so is $\ms D$ (by \cite[Chapter 5, Corollary~3.24]{Liu}).  Hence some multiple of $\ms D$ is very ample, and so the complement of its support in $\ms X$ is affine.  We may then take $\ms U$ to be that complement.
\item
In the case where $\ms X$ has dimension greater than one over $T$, even if a given affine open set $U$ is not of the form $\ms U \cap X$, one can still cover $U$ by affine open subsets $V$ of that form, since every point of $U$ has such a neighborhood, by definition of the subspace topology.  Here a finite set of such subsets $V$ suffices, by quasi-compactness.  
\end{enumerate}
\end{rem}

We next study the behavior of the rings $\wh R_U = \mc O_{\mf X}(U)$.

\begin{lem} \label{flat patches}
Let $T$ be a complete discrete valuation ring with uniformizer $t$, let $\ms X$ be a normal integral $T$-scheme of finite type, and let
$\mf X$ be the formal completion of $\ms X$ along its reduced closed fiber $X$.  Let $U$ be an affine open subset of $X$. 
\begin{enumerate} [(a)]
\item 
The natural map $U \to \Spec(\wh R_U)$ is a bijection on closed points.
\item
If $V \subseteq U$ is an affine open subset, then $\wh R_V$ is flat over $\wh R_U$.
\item
If $V_1,\dots,V_n \subseteq U$ are affine open subsets such that $\bigcup_{i=1}^n V_i = U$, then $\prod_{i=1}^n \wh R_{V_i}$ is faithfully flat over $\wh R_U$.
\end{enumerate}
\end{lem}

\begin{proof}
Let $I$ be the radical of the ideal $t\mc O_{\mf X}(U)$.
Since $\wh R_U/I = \mc O_X(U)$, the natural map $U \to \Spec(\wh R_U)$ induces a bijection between the maximal ideals of $\mc O_X(U)$ and the maximal ideals of $\wh R_U$ that contain $I$.  But since $\wh R_U$ is $I$-adically complete, the ideal $I$ is contained in the Jacobson radical of $\wh R_U$ (see \cite[Proposition~23.G]{Mats}), and hence in every maximal ideal of $\wh R_U$.  So part~(a) follows.

For~(b), let $X_n$ be the reduction of $X$ modulo $t^n$, and let $U_n,V_n$ be the homeomorphic images of $U,V$ under $X \to X_n$. 
Then $V_n \subseteq U_n$ is an inclusion of affine open subsets of $X_n$ by
\cite[Proposition~5.1.9]{EGA1}, and so $\mc O_{X_n}(V_n)$ is flat over $\mc O_{X_n}(U_n)$.  Here $\mc O_{\mf X}(U)/(t^n) = \mc O_{X_n}(U_n)$ and similarly for $V$ and $V_n$.  
By \cite[Lemma~0912]{stacks}, 
$\wh R_V = \mc O_{\mf X}(V)$ is flat over $\wh R_U = \mc O_{\mf X}(U)$.  So part~(b) holds.

By part~(a), every maximal ideal of $\wh R_U$ is of the form $\frak m_{U,P}$ for some closed point $P$ of $U = \bigcup_{i=1}^n V_i$.  Here $P$ lies on some $V_i$, and so $\frak m_{U,P}$ is the contraction of the maximal ideal $\frak m_{V_i,P}$ of $\wh R_{V_i}$.  Thus every maximal ideal of $\wh R_U$ is the contraction of a maximal ideal of $\prod_{i=1}^n \wh R_{V_i}$.
Also, $\prod_{i=1}^n \wh R_{V_i}$ is flat over $\wh R_U$ because each $\wh R_{V_i}$ is, by part~(b).
Thus by \cite[Proposition~I.3.5.9]{Bo:CA}, $\prod_{i=1}^n \wh R_{V_i}$ is faithfully flat over $\wh R_U$; i.e.,
part~(c) holds.
\end{proof}

\begin{lem}  \label{RU props}
Let $T$ be a complete discrete valuation ring with uniformizer $t$, let $\ms X$ be a normal integral $T$-scheme of finite type, and let
$\mf X$ be the formal completion of $\ms X$ along its reduced closed fiber $X$.  Let $U$ be an affine open subset of $X$.  
\begin{enumerate} [(a)]
\item 
The ring $\wh R_U$ is quasi-excellent and normal (and in particular, Noetherian).
\item 
If $U = \ms U \cap X$ for some 
affine open subset $\ms U \subseteq \ms X$, then $\wh R_U$ is an excellent normal ring.
\item
The ring $\wh R_U$ is a domain if and only if $U$ is connected.
\item
If $U$ is a disjoint union of affine open subsets $U_i$, then the natural map $\wh R_U \to \prod_i \wh R_{U_i}$ is an isomorphism.
\end{enumerate}
\end{lem}

\begin{proof}
Let $I$ be the radical of $t\wh R_U$, and let $U_n \subseteq X_n$ be as in the discussion before Proposition~\ref{RU versions}.  As noted there,
$\displaystyle \wh R_U = \mc O_{\mf X}(U) = \lim_{\leftarrow} \mc O_{X_n}(U_n)$, with $\mc O_{X_n}(U_n) = \wh R_U/(t^n)$ and $\mc O_X(U) = \wh R_U/I$.  Since $\mc O_X(U)$ is of finite type over the residue field $k$ of $T$, it is excellent, and in particular quasi-excellent.  Hence $\wh R_U$ is quasi-excellent by a theorem of Gabber (see \cite[Theorem~5.1]{KuSh}).  This proves the first part of~(a), that $\wh R_U$ is quasi-excellent (and hence Noetherian).  Note also that since $\mc O_X(U) = \wh R_U/I$, we can identify $U$ with the closed subset of $\Spec(\wh R_U)$ defined by the ideal~$I$. 

Under the hypothesis of part~(b),  
$\wh R_U = \wh{\mc O_{\ms X}(\ms U)}$, by Proposition~\ref{RU versions}(b).  
Write $\ms U = \Spec(A) \subseteq \ms X$.
The inclusion $\iota:U \hookrightarrow \ms U$ corresponds to a morphism $A \to \mc O_X(U)$ that factors through the $t$-adic completion $\wh A = \wh{\mc O_{\ms X}(\ms U)}$ of $A$.  That is, $\iota$ factors through $\Spec(\wh R_U)$, corresponding to the natural embedding $U \to \Spec(\wh R_U)$.
By \cite[34.B]{Mats}, $T$ is excellent; hence so is $\mc O_{\ms X}(\ms U)$, being a finitely generated $T$-algebra.  So the $t$-adic completion $\wh R_U$ of $\mc O_{\ms X}(\ms U)$ is also excellent, by  \cite[Main~Theorem~2]{KuSh}.  By \cite[33.I, 34.A]{Mats}, $\wh R_U$ is a normal ring, 
since it is the completion of the excellent normal ring
$\mc O_{\ms X}(\ms U)$.   This proves part~(b).

For the last part of~(a), concerning normality, recall that any affine open subset $U$ of $X$ is the union of finitely many open subsets $V_i$ of the form $\ms V_i \cap X$, with $\ms V_i$ an affine open subset of $\ms X$.  By part~(b), each $\wh R_{V_i}$ is normal; hence so is $\prod_i \wh R_{V_i}$.  Also, by Lemma~\ref{flat patches}(c), $\prod_i \wh R_{V_i}$  is faithfully flat over $\wh R_U$.  So by \cite[Lemma~030C]{stacks}, 
$\wh R_U$ is normal, completing the proof of part~(a).
 
Since $\Spec(\wh R_U)$ is normal, it is in particular reduced.  So $\Spec(\wh R_U)$ is integral if and only if it is connected.  But since every connected component of $\Spec(\wh R_U)$ contains a closed point, it follows from 
Lemma~\ref{flat patches}(a)
that $\Spec(\wh R_U)$ is connected if and only if $U$ is connected.  Thus part~(c) follows.

Part~(d) is immediate from the definition of $\wh R_U$ as $\mc O_{\mf X}(U)$ together with the fact that $\mc O_{\mf X}$ is a sheaf.
\end{proof}

The next result further relates $\wh R_U$ to $\wh R_V$, where $V \subseteq U$ are affine open subsets of $X$.

\begin{lem} \label{inclusion lemma}  
Let $T$ be a complete discrete valuation ring,
and let $\ms X$ be a normal integral $T$-scheme of finite type, with reduced closed fiber $X$.  Let $V \subseteq U$ be an inclusion of non-empty affine open subsets of $X$.
\begin{enumerate} [(a)]
\item The contraction of each minimal prime ideal of $\wh R_V$ is a minimal prime ideal of $\wh R_U$.
\item Every regular element of $\wh R_U$ has the property that its image is regular in $\wh R_V$.
\item The natural map $\wh R_U \to \wh R_V$ is injective if and only if $V$ meets each connected component of $U$.  In particular, it is injective if $V$ is dense in $U$, or if $U$ is connected.
\item If the map $\wh R_U \to \wh R_V$ is injective, it
induces a well-defined injection between the total rings of fractions of $\wh R_U$ and $\wh R_V$.
\end{enumerate}
\end{lem}

\begin{proof}
Since $\wh R_V$ is flat over $\wh R_U$ by Lemma~\ref{flat patches}(b),
the going down theorem holds for this ring extension by \cite[Theorem~5.D]{Mats}. 
Hence the contraction of every minimal prime ideal of $\wh R_V$ is a minimal prime ideal of $\wh R_U$, proving part~(a).

To prove part (b), we show that an element of $\wh R_U$ that becomes a zero-divisor in $\wh R_V$ is already a zero-divisor in $\wh R_U$.  By Lemma~\ref{RU props}(a), the rings $\wh R_V$ and $\wh R_U$ are normal and in particular reduced.
Hence by \cite[Lemmas~0EMA, 05C3]{stacks},   
the set of zero-divisors in $\wh R_V$ (resp.\ $\wh R_U$) is the union of the minimal primes of that ring.  
So if the image $r' \in \wh R_V$ of some $r \in \wh R_U$ is a zero-divisor in $\wh R_V$, then $r'$ lies in a minimal prime of $\wh R_V$.
By part~(a), $r$ lies in a minimal prime of $\wh R_U$, and so is a zero-divisor in $\wh R_U$, as needed.   

In part~(c), the second assertion is immediate from the first.  For the forward direction of the first assertion, in the special case that $U$ is connected, $\wh R_U$ is a domain by Lemma~\ref{RU props}(c), and hence every non-zero element $r \in \wh R_U$ is regular.  Thus by part~(b) above, the image of $r$ in $\wh R_V$ is regular and hence non-zero.  Thus the map is injective.  For the more general case, let $U_1,\dots,U_n$ be the connected components of $U$, and let $V_i = U_i \cap V$.  Thus each $U_i$ and $V_i$ is an affine open set, with $\wh R_U \cong \prod_i \wh R_{U_i}$ and $\wh R_V \cong \prod_i \wh R_{V_i}$ by Lemma~\ref{RU props}(d).  By the above special case, each $\wh R_{U_i} \to \wh R_{V_i}$ is injective.  Hence so is $\wh R_U \to \wh R_V$, showing the forward direction.  For the reverse direction, if $V$ does not meet some connected component $U_j$ of $U$, let $r \in \wh R_U \cong \prod_i \wh R_{U_i}$ be the element given by $1$ in $\wh R_{U_j}$ and by $0$ in every other $\wh R_{U_i}$.  Then the image of $r$ in $\wh R_V$ is $0$, and so the map $\wh R_U \to \wh R_V$ is not injective.

For~(d), let $S_U, S_V$ be the sets of regular elements in $\wh R_U, \wh R_V$.  The total rings of fractions of these rings are $S_U^{-1}\wh R_U$ and $S_V^{-1}\wh R_V$.  By 
part~(b), the injection $\wh R_U \to \wh R_V$
restricts to an injection $S_U \to S_V$.  Thus $\wh R_U \to \wh R_V$ induces a map $S_U^{-1}\wh R_U \to S_V^{-1}\wh R_V$, which factors through $S_U^{-1}\wh R_V$.  Here $S_U^{-1}\wh R_U \to S_U^{-1}\wh R_V$ is injective because localization is exact; and $S_U^{-1}\wh R_V \to S_V^{-1}\wh R_V$ is injective because the elements of $S_V$ are regular in $\wh R_V$.  This proves~(d).
\end{proof}

The next lemma controls the behavior of the principal ideal $(t)$ in the rings corresponding to different patches.

\begin{lem}\label{associated}  
Let $T$ be a complete discrete valuation ring with uniformizer $t$, and let $\ms X$ be a normal integral  $T$-scheme of finite type.
Let $U$ be an affine open subset of the reduced closed fiber $X$ of $\ms X$, and let $U' \subseteq U$ be an affine dense open subset.
Write $\wh R$ and $\wh R'$ for $\wh R_U$ and $\wh R_{U'}$, respectively. For $i\geq 1$, let $R_i, R_i'$ denote the quotients of $\wh R, \wh R'$ by the ideals generated by $t^i$ in the respective rings.
Let $\{P_1,\ldots, P_s\}$ be the set of minimal primes over $t\wh R$,
and write $P_j\wh R', P_jR_i, P_jR'_i$ for the extension of $P_j$ to $\wh R', R_i, R'_i$, respectively.
Then 
\begin{enumerate}
\item \label{associated-patches} The minimal primes over $t\wh R'$ are the ideals $P_j\wh R'$ ($j=1,\ldots, s$). 
\item \label{associated-extension} $P_jR_i$ is the contraction of $P_jR_i'$ to $R_i$.
\item \label{associated-complement} The ideal $J_i \subset R_i$ defining the complement of $\operatorname{Spec}(R_i')$ in $\operatorname{Spec}(R_i)$ has the property that $J_iR_i'$ is the unit ideal.  Moreover, it
is generated by (finitely many) elements that are not in $\bigcup\limits_{j=1}^s P_jR_i$.
\end{enumerate}
\end{lem}

\begin{proof}
The natural map $\wh R \to \wh R'$ is an inclusion, by Lemma~\ref{inclusion lemma}(c).
The irreducible components of the reduced closed fiber of $\Spec(\wh R)$ are the integral schemes $Y_j := \Spec(\wh R/P_j)$ for $j=1,\dots,s$.  The irreducible components of the reduced closed fiber of $\Spec(\wh R ')$  are the intersections $Y_j' = U' \cap Y_j \subseteq U'$, each of which is non-empty because $U'$ is dense in $U$.  Here $Y_j'$ is the closed subset of $\Spec(\wh R')$ defined by the ideal $P_j\wh R'$, for $j=1,\dots,s$.  So these are the minimal primes of $\wh R'$ over $t\wh R'$, showing (a).

Fix $j$.  Since $Y_j'$ is a dense open subset of the integral scheme $Y_j$, the natural map $\mc O_X(Y_j) \to \mc O_X(Y_j')$ is an inclusion of subrings of the function field of $Y_j$ (or equivalently of $Y_j'$).  But $\mc O_X(Y_j) = \wh R/P_j = (\wh R/t^i\wh R)/(P_j/t^i\wh R) = R_i/P_jR_i$, and similarly $\mc O_X(Y_j') = R_i'/P_jR_i'$.  So the map 
$R_i/P_jR_i \to R_i'/P_jR_i'$ is an inclusion, for all $i$.  Hence $\ker(R_i \to R_i'/P_jR_i')$, which is
the contraction of $P_jR_i'$ to $R_i$, is equal to
$\ker(R_i \to R_i/P_jR_i) = P_jR_i$;
showing (b).

The first part of~(\ref{associated-complement}) is immediate because $J_iR_i'$ defines the empty subscheme of $\Spec(R_i')$.
To prove the second part of~(\ref{associated-complement}), first choose any finite set of generators $\{a_1,\dots,a_d\} \subset R_i$ of~$J_i$.  We will modify the generating set so that none of its elements lie in $\bigcup\limits_{j=1}^s P_jR_i$.

Since the ideals $P_j$ are minimal over $t\wh R$, no $P_j$ contains any $P_k$ for $k \ne j$.  By \cite[Proposition~1.11(ii)]{AM}, $P_j$ does not contain $\bigcap_{k \ne j} P_k$; i.e., there exists $\rho_j \in \wh R$ such that 
$\rho_j$ is not contained in $P_j$ but is contained in every other $P_{k}$.  Hence its image $\bar \rho_j \in R_i$ is not contained in $P_jR_i$ (using that $t \in P_j$) but is contained in $P_{k}R_i$ for every other $k$. 

Since $U'$ is dense in $U$, the ideal $J_i$ is not contained in any of the ideals $P_jR_i$ (each of which defines an irreducible component of $U_i := \Spec(R_i)$ and hence of $U = U_i^{\rm red}$).  By prime avoidance (see \cite[Proposition~1.11(i)]{AM}),  
$J_i$ is not contained in $\bigcup_j P_jR_i$; i.e., there exists $r_0 \in J_i$ that is not in any $P_jR_i$.  For $h=1,\dots,d$, let $A_h = \{j\,|\,a_h \in P_jR_i\}$, and let $r_h = a_h + \sum_{j\in A_h} r_0\bar \rho_j$.  Then $r_h$ does not lie in any $P_jR_i$, and the ideal $J_i$ is generated by the $d+1$ elements $r_0,r_1,\dots,r_d$.  This proves the second part of~(\ref{associated-complement}). 
\end{proof}

\section{Modules on patches} \label{patches}

The previous section concerned rings of formal functions on patches.  
In this section, we study modules over these rings.  We build on the previous results to obtain Proposition~\ref{patch 2}, a key step in the proof of our theorem on the finiteness of pushforwards. 

\begin{lem} \label{module inj}   
Let $T$ be a complete local domain,
and let $\ms X$ be a normal integral $T$-scheme of finite type, with reduced closed fiber $X$.  Let $V \subseteq U$ be an inclusion of non-empty affine open subsets of $X$, with $V$ dense in $U$.
Then for every finitely generated torsion-free $\wh R_U$-module~$M$,   
the natural map $\iota_V:M \to M \otimes_{\wh R_U} \wh R_V$ is injective.  
\end{lem}

\begin{proof}
Let $U_1,\dots,U_s$ be the connected components of $U$. So $\wh R_U = \prod_i \wh R_{U_i}$ by
Lemma~\ref{RU props}(d); and each $\wh R_{U_i}$ is a Noetherian normal domain by Lemma~\ref{RU props}(a,c).
Since $M$ is a finitely generated torsion-free $\wh R_U$-module, it follows that $M = \prod_i M_i$, where $M_i$ is a finitely generated torsion-free $\wh R_{U_i}$-module for each $i$.
For every $i$, the intersection $V_i := V \cap U_i$ is an affine open dense subset of $U_i$; and $V$ is their disjoint union.  Thus $\wh R_V = \prod_i \wh R_{V_i}$, and $\iota_V$ decomposes as a product of maps $\iota_{V,i}:M_i \to M_i \otimes_{\wh R_{U_i}} \wh R_{V_i}$.  
So by considering each pair $\wh R_{U_i}$, $\wh R_{V_i}$, we are reduced to the case where $U$ is connected and $\wh R_U$ is a Noetherian domain.

Let $K_U$ be the fraction field of the domain $\wh R_U$, and let $K_V$ be the total ring of fractions of~$\wh R_V$.
Since $V$ is dense in $U$, we have a natural map $K_U \hookrightarrow K_V$ by Lemma~\ref{inclusion lemma}(c,d); so~$K_U$ and~$K_V$ are $\wh R_U$-modules, and $K_V$ is a $K_U$-vector space.
The composition $M \to M \otimes_{\wh R_U} \wh R_V \to M \otimes_{\wh R_U} \wh R_V \otimes_{\wh R_V} K_V = M \otimes_{\wh R_U} K_V$ also factors as 
$M \to M \otimes_{\wh R_U} K_U \to M \otimes_{\wh R_U} K_U \otimes_{K_U} K_V = M \otimes_{\wh R_U} K_V$.  Here the map $M \to M \otimes_{\wh R_U} K_U$ is injective because $M$ is torsion-free over $\wh R_U$; and the map $M \otimes_{\wh R_U} K_U \to M \otimes_{\wh R_U} K_U \otimes_{K_U} K_V$ is injective since $M \otimes_{\wh R_U} K_U$ is flat over the field $K_U$.  
So the composition of these maps is injective.  But the above two compositions are equal, hence the map $M \to M \otimes_{\wh R_U} \wh R_V$ is injective.
\end{proof}

\begin{lem}\label{diagram}  
Let $K$ be a complete discretely valued field with valuation ring~$T$ and uniformizer~$t$. 
Let $\ms X$ be a normal integral $T$-scheme of finite type,
and let $U$ be 
a non-empty affine open subset of the reduced closed fiber $X$ of $\ms X$. Consider 
an affine open subset $U'\subseteq U$ 
that is dense in $U$. Write $\wh R$ and $\wh R'$ for $\wh R_U$ and $\wh R_{U'}$, respectively.
Let $M$ be a finitely generated $\wh R$-module, let $M'=M\otimes_{\wh R}\wh R'$, and let $R_i, M_i, Q_i, N_i$ (resp., $R_i', M_i', Q_i', N_i'$) be the rings and modules given by Lemma~\ref{mnqlemma} for these two modules, with respect to the ideal~$t\wh R$ (resp., $t\wh R'$).
Then the natural map $R_i\rightarrow R_i'$ induces a commutative diagram
\[\begin{tikzcd}
0\arrow{r}&Q_i\otimes_{R_i}R_i'\arrow{r}\arrow{d}{\cong}
&M_i\otimes_{R_i}R_i'\arrow{d}{\cong}\arrow{r}
&N_i\otimes_{R_i}R_i'
\arrow{r}\arrow{d}{\cong}
&0\\
0\arrow{r}&Q_i'\arrow{r}
&M_i'\arrow{r}
&N_i'
\arrow{r}
&0
\end{tikzcd}
\]
with exact rows.
\end{lem}

\begin{proof}
First note that $\wh R$ and $\wh R'$ are quasi-excellent normal rings
by Lemma~\ref{RU props}(a), and in particular they are G-rings.  So Lemma~\ref{mnqlemma} does in fact provide us with the data $R_i, M_i, Q_i, N_i$ and $R_i', M_i', Q_i', N_i'$ as in the above assertion.
Moreover, $\wh R \hookrightarrow \wh R'$ by Lemma~\ref{inclusion lemma}(c).

For each $i$, $R_i = \wh R/t^i \wh R = \mc O(U_i)$
and $R_i' = \wh R'/t^i \wh R' = \mc O(U_i')$,
where $U_i, U_i'$ are the 
open subsets of the mod $t^i$ reduction $X_i$ of $\ms X$ that correspond to $U,U' \subseteq X$, respectively.
Since $U_i' = \Spec(R_i')$ is an open subset of $U_i = \Spec(R_i)$, the ring $R_i'$ is flat over $R_i$.
Thus the exact sequence 
\[0\rightarrow Q_i\rightarrow M_i\rightarrow N_i\rightarrow 0\] 
from Lemma~\ref{mnqlemma} yields an exact sequence
\[0\rightarrow Q_i\otimes_{R_i}R_i'\rightarrow M_i\otimes_{R_i}R_i'\rightarrow N_i\otimes_{R_i}R_i'\rightarrow 0\] 
as in the top row in the diagram above.  Similarly, the bottom row is exact by 
Lemma~\ref{mnqlemma}.
Using the definition of $M'$, we have isomorphisms 
\[M_i \otimes_{R_i} R_i' \iso M \otimes_{\wh R} R_i \otimes_{R_i} R_i' \iso M \otimes_{\wh R} R_i' \iso M \otimes_{\wh R} \wh R' \otimes_{\wh R'} R_i' \iso M' \otimes_{\wh R'} R_i' \iso M_i'.\]

Let $P_1,\ldots, P_s$ denote the minimal primes over the ideal~$(t)$ in $\wh R$.  Thus the minimal primes over $t \wh R'$ are the ideals $P_j \wh R'$, by Lemma~\ref{associated}(\ref{associated-patches}).  For any element $m \in M_i$ that lies in $Q_i$, the annihilator of $m$ in $R_i$ is not contained in $P_jR_i$ for any $j$, by definition of $Q_i$.  Thus by Lemma~\ref{associated}(\ref{associated-extension}), this annihilator is also not contained in $P_jR_i'$ for any $j$.
Hence the image of the inclusion $Q_i\otimes_{R_i}R_i'\rightarrow M_i\otimes_{R_i}R_i'$ is contained in $Q_i'$. This gives the left hand vertical arrow $Q_i\otimes_{R_i}R_i' \to Q_i'$, which is then injective; and it also gives the right hand vertical arrow $N_i\otimes_{R_i}R_i' \to N_i'$ such that the diagram commutes.  
We claim that the map $Q_i\otimes_{R_i}R_i'\rightarrow Q_i'$ is surjective, and hence an isomorphism. 
Since the middle vertical arrow is an isomorphism as observed above, this claim will imply that 
the right hand vertical arrow is an isomorphism,
and thus will finish the proof. 

We begin with the case in which $U'$ is a basic open subset of $U$; i.e., it is the complement of the zero set of some element $\bar f \in \mc O_X(U)$.  We may lift $\bar f$ to some $\til f \in \wh R_U = \mc O_{\mf X}(U)$.  Since $U'$ is dense in $U$, the element $\til f$ (and similarly, $\bar f$) does not vanish at the generic point of any irreducible component of $U$.  Fixing $i$, we write $f$ for the image of $\til f$ in $R_i$.  Thus $R_i' = R_i[f^{-1}]$, and 
for $j=1,\dots,s$ the element $f$ does not lie in the ideal $P_jR_i$.  

Let $m'\in Q_i'\subseteq M_i'=M_i\otimes_{R_i}R_i'$; we wish to show that $m'$ is in the image of the map $Q_i\otimes_{R_i}R_i'\rightarrow Q_i'$.  Since $m' \in Q_i'$, there exists 
$r'\in R_i'$ with $r'm'=0 \in M_i'$, such that $r'$ does not lie in any of the primes $P_jR_i'$.  For $j=1,\dots,s$, we have $f \not\in P_jR_i$;
and the image $f' \in R_i'$ of $f$ is a unit in $R_i'$, say with inverse denoted by $g$.

The homomorphism $R_i \to R_i' = R_i[f^{-1}] = S^{-1}R_i$ induces the homomorphism $M_i \to M_i' = M_i \otimes_{R_i} R_i' = S^{-1}M_i$, where $S \subset R_i$ is the multiplicative set generated by $f$.  
We may write $r' = r/f^a \in S^{-1}R_i$ and $m' = m/f^b \in S^{-1}M_i$, for some $r \in R_i$, some $m \in M_i$, and some $a,b\ge 0$.
Thus the image of $m \in M_i$ in $M_i'$ is $(f')^b m'$.
Since $f' \in R_i'$ is a unit and since $r'$ lies in no $P_jR_i'$,
the element $(f')^a r' = r/1 \in R_i'$ also lies in no $P_jR_i'$.
The element $r \in R_i$ maps to $r/1 = (f')^a r' \in R_i'$, hence it lies in no $P_jR_i \subset R_i$.

Now $r m/f^{a+b} = r'm'=0 \in M_i'$, and so by definition of localization we have $f^c r m=0 \in M_i$ for some $c\ge 0$.    
But $f^c r \in R_i$ lies in no $P_jR_i$, since this is true for the elements $f,r \in R_i$ and since $P_jR_i$ is prime.
Thus $m \in Q_i$,  
and $m \otimes g^b \in Q_i \otimes_{R_i} R_i' \subseteq M_i \otimes_{R_i} R_i'$,
where $g$ is as above.
The image of $m \otimes g^b$ in $Q_i' \subseteq M_i'$ is $(f')^bg^bm' = m'$, proving the claim in this case.  

For the general case, let $J_i \subset R_i$ be the ideal defining the complement of $\Spec(R_i')$ in $\Spec(R_i)$
as in Lemma~\ref{associated}(\ref{associated-complement}); and let $f_1,\ldots, f_d$ be generators of $J_i$ given by that lemma, with no $f_h$ lying in $P_jR_i$ for any $j$.  For $h=1,\dots,d$, the element $f_h$ vanishes along the complement of $\Spec(R_i')$ in $\Spec(R_i)$.
Thus $R_i \subseteq R_i' \subseteq R_{h,i} := R_i[f_h^{-1}]$, and so $R_{h,i} = R_i'[f_h^{-1}]$.  The ring $R_{h,i}$
is flat over $R_i'$, being a localization; hence the product ring $\prod_h R_{h,i}$ is also flat over $R_i'$.
Moreover $U_{h,i} :=\Spec(R_{h,i})$
is a basic open subset of $U_i = \Spec(R_i)$ that is contained in $U_i'=\Spec(R_i')$, and such that $\bigcup_h U_{h,i} = U_i'$.   Thus 
$\Spec(\prod_h R_{h,i})$, which is the disjoint union of the open sets $U_{h,i}$, maps surjectively to $U_i'=\Spec(R_i')$.  Hence 
$\prod_h R_{h,i}$ is faithfully flat over $R_i'$, by \cite[4D, Theorem~3]{Mats}.  

Let $M_{h,i} = M_i \otimes_{R_i} R_{h,i} = M_i' \otimes_{R_i'} R_{h,i}$, and let
$Q_{h,i}$ be the submodule of $M_{h,i}$ given as in Lemma~\ref{mnqlemma}.
By the above special case, the maps $Q_i \otimes_{R_i} R_{h,i} \to Q_{h,i}$ 
and $Q_i'\otimes_{R_i'} R_{h,i} \to Q_{h,i}$
are isomorphisms.  Let $Q_i''$ be the image of the injective map $Q_i \otimes_{R_i} R_i' \to Q_i'$.  Then 
$Q_i \otimes_{R_i} R_{h,i} = (Q_i \otimes_{R_i} R_i') \otimes_{R_i'} R_{h,i}$, and so the image of 
$Q_i'' \otimes_{R_i'} R_{h,i} \to Q_{h,i}$ is $Q_{h,i}$.  Thus for each $h$, the quotient $R_i'$-module $Q_i'/Q_i''$ becomes trivial upon tensoring with $R_{h,i}$.  So $Q_i'/Q_i''$ also becomes trivial upon tensoring with the faithfully flat $R_i'$-module $\prod_h R_{h,i}$.  Hence $Q_i'/Q_i''$ is already trivial; i.e., $Q_i'' = Q_i'$ and so the map $Q_i \otimes_{R_i} R_i' \to Q_i'$ is indeed surjective, as claimed.
\end{proof}

Recall that if $t$ is a regular element in a ring $R$, then an $R$-module $M$ is $t$-{\it torsion-free} if $rm \ne 0$ for all non-zero elements $m \in M$.

\begin{lem} \label{inj lem}
Let $T$ be a complete discrete valuation ring with uniformizer $t$, and let $\ms X$ be a normal integral $T$-scheme of finite type.  Let $U_0\subseteq U_1,U_2$ be affine open subsets of the reduced closed fiber $X$ of $\ms X$.
Let $M_0$ be a $t$-torsion-free $\wh R_{U_0}$-module, and for $e=1,2$ let $M_e \subseteq M_0$ be an $\wh R_{U_e}$-submodule.  Let $M = M_1 \cap M_2 \subseteq M_0$; and for every $i \ge 1$ write $M_{e,i} := M_e/t^iM_e$ for $e=0,1,2$ and  
$M_i' := M_{1,i} \times_{M_{0,i}} M_{2,i}$.  Then the natural map $M/t^iM \to  M_i'$ is injective. 
\end{lem}

\begin{proof}
Let $m \in M/t^iM$ lie in the kernel of $M/t^iM \to  M_i'$,
and pick a representative $\til m \in M$ for $m$.  
The image of $m$ in $M_{e,i} = M_e/t^iM_e$ is trivial, and this is the same as the image of $\til m$.  Hence for $e=1,2$ there exist $m'_e \in M_e$ such that $\til m = t^im'_e \in M_e$.  
Since $t^im'_1 = \til m = t^im'_2$ in $M_0$, the
element $m'_1 - m'_2 \in M_0$ is $t^i$-torsion.  But $M_0$ is $t$-torsion-free, and so $m'_1=m'_2 \in M_0$.  That is, $m'_1 \in M_1$ and $m'_2 \in M_2$
define a common element $m' \in M$, and $\til m = t^im' \in M$.  
So $\til m \in t^iM$, and $m \in M/t^iM$ is trivial.
\end{proof}

We now come to the main result of this section, which will provide the key step in proving finiteness of formal pushforwards.

\begin{prop} \label{patch 2} 
Let $T$ be a complete discrete valuation ring with uniformizer $t$, and let~$\ms X$ be a normal integral $T$-scheme of finite type.
Let $U_0,U_1,U_2,U$ be 
affine open subsets of the reduced closed fiber $X$ of $\ms X$, with $U_1,U_2 \subseteq U$ dense, and with $U_0 = U_1 \cap U_2$, such that the complement of $W := U_1 \cup U_2$ in $U$ has codimension at least two.  
Let $M_e$ be a finitely generated torsion-free $\wh R_{U_e}$-module for $e=0,1,2$.  For $e=1,2$, 
consider the natural map $\iota_e:M_e \to M_e \otimes_{\wh R_{U_e}} \wh R_{U_0}$, and
let $\alpha_e:M_e \otimes_{\wh R_{U_e}} \wh R_{U_0} \to M_0$ be an isomorphism.  Then 
$\alpha_e\iota_e$ is injective for $e=1,2$, and the 
intersection $M := \alpha_1\iota_1(M_1) \cap \alpha_2\iota_2(M_2) \subseteq M_0$ is a finitely generated torsion-free $\wh R_U$-module.
\end{prop}

\begin{proof}
For short, write $\wh R_e = \wh R_{U_e}$ 
for $e=0,1,2$.  Since $U_1,U_2$ are each dense in $U$, the intersection $U_0 = U_1 \cap U_2$ is dense in $U_1,U_2$.  So we may apply
Lemma~\ref{module inj} and obtain that each $\iota_e$ is injective. Since $\alpha_e$ is an isomorphism, the composition 
$\alpha_e\iota_e$ is injective.  Because of this injectivity, we may
identify $M_e$ with its image under $\alpha_e \iota_e:M_e \to M_0$ for $e=1,2$, and thus regard $M_e$ as contained in $M_0$.  
Here $\wh R_U, M$ are respectively contained in $\wh R_e, M_e$, and every regular element of $\wh R_U$ is regular over $\wh R_e$ by Lemma~\ref{inclusion lemma}(b).
Thus $M$ is torsion-free over $\wh R_U$, since $M_e$ is torsion-free over $\wh R_e$.

With respect to the above identifications, the goal of the proof 
is then to show that $M:=M_1 \cap M_2$ is finitely generated over $\wh R_U$.
By \cite[Lemma~087W]{stacks}, 
it suffices to show that $M/t^iM$ is a finitely generated $\wh R_U/(t^i)$-module for all $i$. 

For $e=0,1,2$ and $i \ge 1$, write $R_{e,i} = \wh R_e/t^i \wh R_e$.
The irreducible components of the reduced closed fiber of $\Spec(\wh R_U)$ are $\Spec(\wh R_U/P_j)$ for $j=1,\dots,s$, where $P_1,\dots,P_s$ are the minimal primes over $t\wh R_U$.  For $e=0,1,2$, 
the minimal primes over $t\wh R_e$ are the ideals $P_j\wh R_e$ for $j=1,\dots,s$, by Lemma~\ref{associated}(a).
By Lemma~\ref{RU props}(a), 
each $\wh R_e$ is a quasi-excellent $t$-adically complete normal ring, and hence a G-ring.  So
Lemma~\ref{mnqlemma} applies, with $\wh R_e, M_e, P_j\wh R_e$ playing the roles of $R,M,P_j$ there.  Let $M_{e,i}, Q_{e,i}, N_{e,i}$ be the modules given in Lemma~\ref{mnqlemma} in that situation.  Thus $M_{e,i}$ and its quotient $N_{e,i}$ are finitely generated modules over $R_{e,i}$ and over $\wh R_e$, and 
$\displaystyle \lim_\leftarrow M_{e,i} = \lim_\leftarrow N_{e,i} = M_e$, for $e=1,2$.  
Also by that lemma, for $e=0,1,2$ and $i \ge 1$, the 
associated primes of $N_{e,i}$ are among
$P_1R_{e,i},\dots,P_sR_{e,i}$.  Here the support of $P_jR_{e,i}$ is dense in the corresponding irreducible component of $\Spec(\wh R_U/(t^i))$.

Since $U_0$ is dense in $U_e$ for $e=1,2$, 
we may apply Lemma~\ref{diagram} to $U_0\subseteq U_e$, and obtain 
isomorphisms of finite modules $N_{e,i} \otimes_{R_{e,i}} R_{0,i} = N_{0,i}$.
By \cite[Lemma~00AM]{stacks}, 
these modules and isomorphisms 
define a coherent sheaf $\ms N_i$ of $\mc O_{W_i}$-modules on $W_i$, where we write $W_i := \Spec(R_{1,i}) \cup \Spec(R_{2,i}) \subseteq \Spec(\wh R_{U}/(t^i))$.  Since the complement of $W$ in $U$ has codimension at least two, the same holds for the complement of $W_i$ in $\Spec(\wh R_{U}/(t^i))$.  Thus each point $z$ of that latter complement has codimension at least two in each irreducible component of $\Spec(\wh R_{U}/(t^i))$ on which it lies, and in particular in the closed subset defined by any of the associated primes of $N_{e,i}$ (each of which is of the form $P_jR_{e,i}$, corresponding to one of these irreducible components).  Since $\wh R_U/(t^i)$ is of finite type over $T$, it is excellent.  So 
\cite[Lemma~0AWA]{stacks} 
(or equivalently, \cite[Partie 2, Corollaire~5.11.4(ii)]{EGA4}) 
applies and shows that $(f_i)_*\ms N_i$ is coherent over $\Spec(\wh R_U/(t^i))$, where $f_i:W_i \to \Spec(\wh R_U/(t^i))$ is the natural inclusion.  
Its module of global sections, which is $N'_i := N_{1,i} \times_{N_{0,i}} N_{2,i}$, is thus finite over $\wh R_U/(t^i)$.  

For every $i \ge 1$, let $M_i' = M_{1,i} \times_{M_{0,i}} M_{2,i}$.  
For every $i \ge 1$, the maps $M=M_1\cap M_2 \to M_e \to M_{e,i} = M_e/t^iM_e$ for $e=0,1,2$ together induce a map $M \to M_i'$ that descends to a map $M/t^iM \to  M_i'$.  
By Lemma~\ref{inj lem}, this latter map is injective.

Say $h\ge i\ge 1$ is an integer.  Then the mod $t^i$ reduction maps $M_{e,h} \to M_{e,i}$, for $e=0,1,2$, together define a map $M_{h}' \to M_i'$.  With respect to the injections $M/t^{h} M \to  M_{h}'$ and $M/t^iM \to  M_i'$, this restricts to the surjection $M/t^{h} M \to M/t^iM$ given by reduction modulo~$t^i$.  Hence the image of $M_h' \to M_i'$ contains $M/t^iM$, viewed as a submodule of $M_i'$.

For every $i \ge 1$, write $Q_i' = Q_{1,i} \times_{Q_{0,i}} Q_{2,i}$.  For $e=0,1,2$, we have a short exact sequence $0 \to Q_{e,i} \to M_{e,i} \to N_{e,i} \to 0$.  Since taking fiber products is left exact, we obtain a left exact sequence $0 \to Q'_i \to M'_i \to N'_i$ for each $i$, where as above $N'_i = N_{1,i} \times_{N_{0,i}} N_{2,i}$.  Thus $N_i := M'_i/Q'_i$ is a submodule of the finitely generated $\wh R_U/(t^i)$-module $N_i'$; 
and so $N_i$ is also finitely generated over $\wh R_U/(t^i)$, since $\wh R_U/(t^i)$ is Noetherian.

To conclude the proof, we use the above to show that $M/t^iM$ is a finitely generated $\wh R_U/(t^i)$-module for all $i$. 
For $e=0,1,2$, let  $n_e$ be the integer given in Lemma~\ref{mnqlemma}(e) for the modules $\{Q_{e,i}\}$.  Let $n = \max(n_0,n_1,n_2)$.  
Thus $Q_{e,i-1+n} \to Q_{e,i}$ is trivial for $e=0,1,2$, and so the map $M_{e,i-1+n} \to M_{e,i}$ restricts to the trivial map on $Q_{e,i-1+n}$.  Hence the restriction of $M'_{i-1+n} \to M'_i$ to $Q'_{i-1+n}$ is also trivial.  Thus the map $M'_{i-1+n} \to M'_i$ induces a map $N_{i-1+n} \to M'_i$ that has the same image.  This image is finitely generated because $N_{i-1+n}$ is.  But as noted above (taking $h=i-1+n$), the image of $M'_{i-1+n} \to M'_i$ contains $M/t^iM$.  Thus 
$M/t^iM$ is indeed finitely generated over $\wh R_U/(t^i)$, completing the proof.
\end{proof}

\section{Formal pushforwards} \label{pushforwards}

Recall that if $(Z,\mc O_Z)$ is any ringed space, and $M$ is a module over $R:=\Gamma(Z,\mc O_Z)$, then there is a functorially associated quasi-coherent sheaf $\mc F_M$ on $Z$ whose presentation is induced by that of $M$.  That is, if we choose a presentation $\bigoplus_J R \to \bigoplus_I R \to M \to 0$, then $\mc F_M$ is given by the corresponding presentation $\bigoplus_J \mc O_Z \to \bigoplus_I \mc O_Z \to \mc F_M \to 0$;
see \cite[Lemma~01BH, Definition~01BI]{stacks}.  
In the case of a Noetherian affine formal scheme $\mf X = \Spf(A)$ and a finite $A$-module $M$, the sheaf $\mc F_M$ on $\fX$ is the formal sheaf $M^\Delta$ associated to the coherent sheaf of modules $\til M$ on the scheme $\Spec(A)$; see \cite[Section~10.10.1]{EGA1}.
This sheaf $M^\Delta$ is coherent as an $\mc O_\fX$-module and it satisfies $\Gamma(\fX,M^\Delta)=M$,
by \cite[Propositions~10.10.5, 10.10.2(i)]{EGA1}.  
Moreover, every coherent $\mc O_\fX$-module is uniquely of the form $M^\Delta$, by \cite[Proposition~10.10.5]{EGA1}.    

Consider a normal integral scheme $\ms X$ of finite type over a complete discrete valuation ring~$T$, with reduced closed fiber $X$, and let $V \subseteq U$ be open subsets of $X$.  Since $X \subseteq \ms X$ has the subspace topology, there exist (not necessarily affine) open subsets $\ms V \subseteq \ms U$ of $\ms X$ meeting $X$ at $V,U$.  The inclusion map $g: \ms V \hookrightarrow \ms U$ restricts to the inclusion $g: V \hookrightarrow U$; and it also pulls back to compatible inclusions $g_n: V_n \hookrightarrow U_n$ on the reductions of $\ms V, \ms U$ modulo $(t^n)$ for all $n \ge 1$.  As in \cite[10.9.1]{EGA1}, the morphisms $g_n$ together yield a morphism $\wh g: \mf V \to \mf U$ between the induced formal schemes $\mf V = \ms V_{/V}$ and $\mf U = \ms U_{/U}$.  Note that $g_n$ and hence $\wh g$ are independent of the choice of $\ms V$ and $\ms U$, and depend just on the inclusion $g: V \hookrightarrow U$ (and on the $T$-scheme~$\ms X$).

\begin{prop} \label{pushforward torsion-free}
Let $T$ be a complete discrete valuation ring with uniformizer $t$, and let $\ms X$ be a normal integral $T$-scheme of finite type,
with reduced closed fiber $X$ and formal completion $\mf X$.  Let $g:V \to U$ be an inclusion of open subsets of $X$, with inclusion $\wh g:\mf V \to \mf U$ of the associated formal open subschemes of $\mf X$.  
\begin{enumerate} [(a)]
\item 
If $\ms G$ is a coherent (resp.\ torsion-free) $\mc O_{\mf U}$-module, 
then $\wh g^*\ms G$ has the same property on $\mf V$.
\item 
If $U$ and $V$ are affine, 
and if\, $\ms G$ is a coherent $\mc O_{\mf U}$-module,
then there is a natural isomorphism $\ms G(V) \iso \ms G(U) \otimes_{\wh R_U} \wh R_V$. 
\item
If $\ms F$ is a torsion-free $\mc O_{\mf V}$-module, then $\wh g_*\ms F$ is a torsion-free $\mc O_{\mf U}$-module.
\end{enumerate}
\end{prop}

\begin{proof}
Part (a) is immediate from the fact that the properties of being coherent and torsion-free are each local.

For part~(b), write $\mf U = \Spf(R)$ and $\mf V = \Spf(S)$.
Then $\ms G(\mf U)$ is the unique finitely generated $R$-module $M$
such that $\ms G = M^\Delta$.  By \cite[Proposition~10.10.8]{EGA1}, there is a canonical isomorphism $\wh g^*\ms G = \wh g^*(M^\Delta) \iso (M \otimes_R S)^\Delta$.  Hence the composition $\ms G(\mf V) = \wh g^*\ms G(\mf V) \to (M \otimes_R S)^\Delta(\mf V) = M \otimes_R S = \ms G(\mf U) \otimes_{\mc O_{\mf X}(\mf U)} \mc O_{\mf X}(\mf V)$ defines an isomorphism $\ms G(V) \iso \ms G(U) \otimes_{\wh R_U} \wh R_V$, where as before we identify the underlying spaces of $\mf U, \mf V$ with those of $U,V$.

For part (c), we prove the contrapositive.  If $\wh g_*\ms F$ is not torsion-free, then there is an affine open subset $U' \subseteq U$ such that $\wh g_*\ms F(U')$ has torsion as a module over $\mc O_{\mf U}(U') = \wh R_{U'}$.  That is, there exists a non-zero element $m \in \wh g_*\ms F(U') = \ms F(U' \cap V)$ and a regular element $r \in \wh R_{U'}$ such that $rm=0 \in \ms F(U' \cap V)$.  Since $m\ne 0$, there is an affine open subset $V' \subseteq U' \cap V$ such that the image $m' \in \ms F(V')$ of $m$ is non-zero.  Let $r' \in \wh R_{V'}$ be the image of $r \in \wh R_{U'}$.
Thus $r'm'=0 \in \ms F(V')$ since $rm=0$; and $r'$ is regular in $\wh R_{V'}$ by Lemma~\ref{inclusion lemma}(b) applied to the inclusion $V' \subseteq U'$ of affine open sets.  Hence $\ms F(V')$ is not a torsion-free module over $\wh R_{V'}=\mc O_{\mf V}(V')$, and so 
$\ms F$ is not a torsion-free $\mc O_{\mf V}$-module.
\end{proof}

If $\mf U$ is a formal scheme, then a sheaf $\ms M$ of $\mc O_{\mf U}$-modules will be called {\it finite} if $\ms M(\mf V)$ is a finite $\mc O_{\mf U}(\mf V)$-module for every affine open subset $\mf V \subseteq \mf U$.  We make the same definition for sheaves of $\mc O_X$-modules on a scheme $X$.

Although Proposition~\ref{pushforward torsion-free}(a) holds both for the properties of being torsion-free and coherent, Proposition~\ref{pushforward torsion-free}(c) does not carry over in general to the coherent property, nor to the weaker property of being finite (e.g., if $U$ is affine and $V \subset U$ is the complement of a principal divisor).
But as we show in Theorem~\ref{genl fin ext} below, finiteness
is preserved under pushforward if the complement of $V$ in $U$ has codimension at least two.  First, we obtain the following special case of Theorem~\ref{genl fin ext}.

\begin{lem} \label{2 open coherent}  
Let $\ms X,U_0,U_1,U_2,U,W$ be as in Proposition~\ref{patch 2}, 
let $\mf X$ be the formal completion of $\ms X$,
and let $\wh g:\mf W \to \mf U$ be the inclusion of the formal open subschemes of $\mf X$ that are associated to $W, U$.  Let $\ms N$ be a torsion-free coherent sheaf on the formal scheme $\mf W$ that is associated to $W$.  Then $\wh g_*\ms N$ is a finite torsion-free 
sheaf on $\mf U$.
\end{lem}

\begin{proof}
Let $V$ be an affine open subset of $U$.
For $e=0,1,2$, let $V_e = V \cap U_e$.  Then 
$\ms N(V_e) \subseteq \ms N(V_0)$ for $e=1,2$ by Lemma~\ref{module inj} and Proposition~\ref{pushforward torsion-free}(b); 
and we have $\wh g_*\ms N(V) = 
\ms N(V \cap W) = \ms N(V_1) \cap \ms N(V_2)$.  
Since $\ms N$ is coherent and since $V_e$ is affine, $\ms N(V_e)$ is a finite module over $\mc O_\fU(V_e) = \wh R_{V_e}$.
Now $V_1,V_2$ are dense in $V$, since $U_1,U_2$ are assumed dense in $U$.  So we may apply Proposition~\ref{patch 2} to $V_0,V_1,V_2,V$, and conclude that $\wh g_*\ms N(V)$ is a finitely generated torsion-free module over 
$\wh R_V = \mc O_{\mf V}(V)$.  So the assertion follows.
\end{proof}

In order to prove our more general finiteness result about pushforwards (Theorem~\ref{genl fin ext}), we will need to find two affine open subsets $U_1,U_2$ such that the union of the complement has codimension at least two.  For this we state the following lemma, which we will apply to the closed fiber of a formal scheme.

\begin{lem} \label{exist 2 conn opens}
Let $V$ be a quasi-projective variety over a field.  
\begin{enumerate} [(a)]
\item There are affine dense open subsets $U_0,U_1,U_2 \subseteq V$ with $U_0 = U_1 \cap U_2$, such that the complement of $U_1 \cup U_2$ in $V$ has codimension at least two.
\item If $V$ is connected, then we may choose $U_0,U_1,U_2$ in (a) such that for every connected open subset $O \subseteq V$ the intersection $O \cap U_e$ is connected for $e=0,1,2$.  In particular, $U_0,U_1,U_2$ are connected in this case.
\end{enumerate}
\end{lem}

\begin{proof}
It suffices to prove the lemma under the hypothesis that $V$ is connected 
(i.e., proving part~(b)), since part~(a) 
then follows by considering the connected components of $V$.

Let $V_1,\dots,V_n$ be the irreducible components of $V$, with generic points $\eta_1,\dots,\eta_n$.  For $i\ne j$, 
consider the irreducible components $V_{i,j,\ell}$
of $V_i \cap V_j$, and write $\eta_{i,j,\ell}$ for the generic point of $V_{i,j,\ell}$.

We first give a criterion for a non-empty open subset $O \subseteq V$ to be connected.  Given~$O$, 
let $S_O \subseteq \{1,\dots,n\}$ be the set of indices $i$ such that $\eta_i \in O$ (or equivalently, $O \cap V_i$ is non-empty).
Thus the closure of $O$ is the union of the irreducible components $V_i$ for $i \in S_O$.
So $O$ is connected if and only if for every pair $i,j \in S_O$, there exists a chain of indices $i_0,\dots,i_r \in S_O$ with $i_0=i$ and $i_r=j$, such that for every $h=0,\dots,r-1$ the set 
$O$ contains $\eta_{i_h,i_{h+1},\ell}$ for some $\ell$.  

In particular, if an open subset $O \subseteq V$ contains each $\eta_i$ (for $i=1,\dots,n$) and each of the points $\eta_{i,j,\ell}$ (for all $i,j,\ell$), then $O$ is connected and dense in $V$.

We now construct the open sets $U_e$ asserted in the lemma.
Since $V$ is a quasi-projective variety over a field, by \cite[Proposition~3.3.36(b)]{Liu} there exists an affine open subset $U_1 \subseteq V$ that contains each $\eta_i$ and each $\eta_{i,j,k}$.  
Thus $U_1$ is a connected affine dense open subset of $V$, by the above criterion; and so the complement $Z$ of $U_1$ in $V$
has codimension at least one in $V$. 
Similarly, there exists an affine open subset $U_2 \subseteq V$ that contains each $\eta_i$, each $\eta_{i,j,k}$, and the generic points of each irreducible component of $Z$.  Thus $U_2$ is also a connected dense affine open subset of $V$, and so is $U_0 := U_1 \cap U_2$, which contains each of the points $\eta_i$ and $\eta_{i,j,k}$.
The intersection $U_2 \cap Z$ is dense in $Z$, since $U_2$ contains the generic points of $Z$.  
Thus the complement $Y$ of $U_2 \cap Z$ in $Z$ has codimension at least one in $Z$.  
Hence $Y$, which is also the complement of $U_1 \cup U_2$ in $V$, has codimension at least two in $V$, as asserted.

Finally, let $O \subseteq V$ be an arbitrary (non-empty) connected open subset; let $O_e=O\cap U_e$ for $e=0,1,2$; 
and let the set $S_O$ be as in the third paragraph of this proof.  Thus for every pair $i,j \in S_O$, there is a chain of indices in $S_O$ connecting $i$ to $j$ as above.  Since $U_e$ contains all the points $\eta_i$ and $\eta_{i,j,\ell}$, it follows that $S_{O_e}=S_O$ for $e=0,1,2$.  Thus $O_e$ also satisfies the above chain criterion, and hence it is connected.
\end{proof}

\begin{lem} \label{pull push}
Let $T$ be a complete discrete valuation ring with residue field $k$,
and let $\ms X$ be a quasi-projective normal integral $T$-scheme with 
reduced closed fiber~$X$ and formal completion~$\mf X$.  Let 
$f:V \to U$ be an inclusion of open subsets of $X$
such that $V$ is dense in $U$,
and write $\mf V, \mf U$ for the formal open subschemes of $\mf X$ associated to $V,U$.  
Let $\ms G$ be a torsion-free coherent sheaf on $\mf U$.  
Then $\ms G$ is a subsheaf of $\wh f_* \wh f^*\ms G$ via the natural morphism $\ms G \to \wh f_* \wh f^*\ms G$, and $\wh f_* \wh f^*\ms G$ is torsion-free .
\end{lem}

\begin{proof}
For every open subset $U' \subseteq U$, we have a restriction map $\ms G(U') \to \ms G(U' \cap V) = \wh f_* \wh f^*\ms G(U')$, and these are compatible as $U'$ varies.  These maps define a morphism $\ms G \to \wh f_* \wh f^*\ms G$.
Since injectivity of sheaves is local, in order to show that $\ms G$ is a subsheaf of $\wh f_* \wh f^*\ms G$ via this morphism, it suffices to show that if $U'$ is affine then $\ms G(U') \to \ms G(U' \cap V)$ is injective.  Let $V'$ be an affine dense open subset of $V$.  
By Proposition~\ref{pushforward torsion-free}(b),
the natural map 
$\ms G(U' \cap V') \to \ms G(U') \otimes_{\wh R_{U'}} \wh R_{U' \cap V'}$ is an isomorphism.  Since $U' \cap V'$ is dense in $U'$, 
Lemma~\ref{module inj} then yields that the map $\ms G(U') \to \ms G(U' \cap V')$ is injective.  But this map factors through $\ms G(U') \to \ms G(U' \cap V)$; and so that map is injective as well.  Thus the torsion-free coherent sheaf $\ms G$ is a subsheaf of $\wh f_* \wh f^*\ms G$, as asserted.  Here $\wh f_* \wh f^*\ms G$ is necessarily torsion-free, by Lemma~\ref{pushforward torsion-free}.
\end{proof}

We now come to our main finiteness result, 
which generalizes Lemma~\ref{2 open coherent}, and provides a
weak analog for formal schemes of the assertion in Theorem~\ref{scheme pushforward}.

\begin{thm} \label{genl fin ext}  
Let $T$ be a complete discrete valuation ring 
and let $\ms X$ be a quasi-projective normal integral 
$T$-scheme 
with 
reduced closed fiber $X$ and formal completion $\mf X$.  Let $U$ be a non-empty open subset of $X$, and let $V$ be an open subset of $U$ whose complement in $U$ has codimension at least two.
Let $f:V\to U$ be the inclusion map, and write $\frak U, \frak V$ for the formal open subschemes of $\mf X$ associated to $U,V$. Let $\ms F$ be a torsion-free coherent sheaf on~$\frak V$.  Then $\wh f_*\ms F$ is a finite torsion-free 
sheaf on $\mf U$. 
\end{thm}

\begin{proof}
Let $k$ be the residue field of $T$.
Since $\ms X$ is quasi-projective over $T$, the $k$-scheme $V$ is quasi-projective over $k$.
By Lemma~\ref{exist 2 conn opens}(a), 
we may choose affine open dense subsets  $U_1, U_2 \subseteq V$ such that the complement of  $W := U_1 \cup U_2$ in $V$ has codimension at least two.  
Hence the complement of $W$ in $U$ also has codimension at least two.
Let $g:W \to V$ be the inclusion map.
By Proposition~\ref{pushforward torsion-free}(a),
pullbacks with respect to open inclusions preserve the property of being a torsion-free coherent sheaf; so 
$\wh g^*\ms F$ is a torsion-free coherent sheaf on the formal scheme $\mf W$ associated to $W$.  
By Lemma~\ref{2 open coherent} applied to $\wh g^*\ms F$ and the inclusion $fg:W \to U$, we have that
$(\wh{fg})_* \wh g^*\ms F$ is a finite torsion-free 
sheaf on $\mf U$. 

By Lemma~\ref{pull push}, $\ms F$ is a subsheaf of $\wh g_*\wh g^*\ms F$ on $\mf V$.  
Thus $\wh f_*(\ms F)$ is a subsheaf of $\wh f_* \wh g_*\wh g^* \ms F = (\wh{fg})_*\wh g^* \ms F$,
and hence is finite.
It is also torsion-free, being a subsheaf of the torsion-free sheaf $(\wh{fg})_* \wh g^*\ms F$ (or by Proposition~\ref{pushforward torsion-free}(c)).
\end{proof}

\begin{rem} \label{rems on fin pushforward}
\begin{enumerate} [(a)]
\item \label{weaker hyp remark}
In Theorem~\ref{genl fin ext}, it would suffice to assume that $\ms X$ is a normal integral $T$-scheme and that $V$ is quasi-projective over $k$, rather than requiring $\ms X$ to be quasi-projective over $T$, because the proofs use only the weaker assumption.
\item \label{rem reductions not tor free}
The proof of Theorem~\ref{genl fin ext} relies in particular on Lemma~\ref{2 open coherent}, whose proof uses the technical results in Section~\ref{patches} and therefore also builds on those in Section~\ref{2 lemmas}.
As mentioned at the end of Section~\ref{background},
it would be tempting to try to prove Theorem~\ref{genl fin ext} more directly by using 
Theorem~\ref{scheme pushforward} or the ingredients used in its proof; viz., by applying such assertions about schemes
to the reductions of the given sheaf modulo powers of the uniformizer of $T$.
The difficulty with that approach is that these reductions need not be torsion-free.  For example, take $T = k[[t]]$ for some field $k$ and take $\ms X = \mbb A^1_T$.  Let $U$ be the closed fiber $\mbb A^1_k$, so that $\wh R_U = k[x][[t]]$.  Let $M$ be the torsion-free $\wh R_U$-module with two generators $m_1,m_2$ and the single relation $xm_1-tm_2=0$.  In the reduction $M_i$ of $M$ modulo $t^i$, the element $t^{i-1}m_1$ is $x$-torsion, and $x$ is regular.  So $M_i$ is not torsion-free, and $(x,t^i)$ is a non-minimal associated prime of $M_i$ (being the annihilator of $t^{i-1}m_1$), with support of codimension one in $X_i = \mbb A^1_{T/(t^i)}$. 
Hence one cannot apply \cite[Lemma~0AWA]{stacks} 
(or \cite[Partie 2, Corollaire~5.11.4(ii)]{EGA4}) to $M_i$.   
Note also 
that as in Theorem~\ref{scheme pushforward}, the torsion-free hypothesis cannot simply be dropped; see the discussion after that assertion.
\end{enumerate}
\end{rem}

\section{The reflexive case} \label{reflexive case}

The phenomena discussed in Section~\ref{pushforwards} are better behaved in the situation of reflexive sheaves, which includes the case of locally free sheaves.

\subsection*{Preliminaries}

Proposition~\ref{patch 2} says in particular that $\wh R_{U_1} \cap \wh R_{U_2}$ is a finitely generated torsion-free module over $\wh R_U$, where $U_1,U_2\subseteq U$ are affine dense open subsets such that the complement of $U_1 \cup U_2$ has codimension at least two.  In fact, more is true:

\begin{prop} \label{intersection}  
Let $T$ be a complete discrete valuation ring, and let $\ms X$ be a normal integral 
$T$-scheme of finite type.
Let $U_0,U_1,U_2,U$ be connected affine open subsets of the reduced closed fiber $X$ of $\ms X$, 
with $U_1,U_2 \subseteq U$ dense, and with $U_0 = U_1 \cap U_2$,
such that the complement of $W := U_1 \cup U_2$ in $U$ has codimension at least two.  Then $\wh R_{U_1} \cap \wh R_{U_2} = \wh R_U$, where the intersection takes place in $\wh R_{U_0}$.
\end{prop}

\begin{proof}
Observe that $\wh R_{U_e} \to \wh R_{U_0}$ is injective for $e=1,2$ by Lemma~\ref{inclusion lemma}(c),
because $U_0$ is dense in $U_e$.  Viewing $\wh R_{U_e}$ as a subring of $\wh R_{U_0}$, we let $A = \wh R_{U_1} \cap \wh R_{U_2} \subseteq \wh R_{U_0}$.  
By Lemma~\ref{RU props}, 
$\wh R_U, \wh R_{U_1}, \wh R_{U_2}, \wh R_{U_0}$ are normal domains; hence so is $A$.
Since $\wh R_U \subseteq \wh R_{U_e}$ for each $e$, we have $\wh R_U \subseteq A$.  
We wish to show that this containment is an equality.

Let $t$ be a uniformizer of $T$.  For $n \ge 1$, and for $e=0,1,2$, the scheme $\Spec(\wh R_{U_e}/(t^n))$ has the same underlying topological space as $U_e$.  So
$\Spec(\wh R_{U_0}/(t^n))$ is a Zariski dense open subset of $\Spec(\wh R_{U_i}/(t^n))$ for $i=1,2$; and the map $\wh R_{U_i}/(t^n) \to \wh R_{U_0}/(t^n)$ is injective.  We may thus form the intersection $A_n := \wh R_{U_1}/(t^n) \cap \wh R_{U_2}/(t^n)$ in $\wh R_{U_0}/(t^n)$.  We claim that the natural map $\alpha_n:A/t^nA \to A_n$ is injective.  This is trivial if the integral scheme $\ms X$ lies over the closed point $(t)$ of $\Spec(T)$, since then $\alpha_n$ is the identity.  Otherwise, $\ms X$ dominates $\Spec(T)$ and 
$t$ is a regular element of $\wh R_{U_0}$; i.e., $\wh R_{U_0}$ is $t$-torsion-free.  So in this case the claim follows from Lemma~\ref{inj lem}, taking $M_e = \wh R_{U_e}$.

By Proposition~\ref{patch 2}, 
$A$ is finite over $\wh R_U$ as an extension of normal domains, say of generic degree $d \ge 1$.  
So tensoring the ring extension $\wh R_U \subseteq A$ with the fraction field $K_U$ of $\wh R_U$, we obtain a finite field extension $K_U \subseteq A \otimes_{\wh R_U} K_U = \Frac(A)$ of degree $d$.
It remains to show that $d=1$, since then $\wh R_U \subseteq A$ is a finite extension of normal domains having the same fraction field, and this inclusion is then an equality as desired.

Let $A':= A \otimes_{\wh R_U}\wh R_{U_1}$.  Now $A$ is finite over $\wh R_U$, so $A'$ is finite over $\wh R_{U_1}$.  But $\wh R_{U_1}$ is complete; i.e., is equal to its own $t$-adic completion.  Hence the same holds for $A'$, since its completion is $A' \otimes_{\wh R_{U_1}} \wh R_{U_1}$ by \cite[Lemma~00MA, (3)]{stacks}. 
That is, $\displaystyle A' = \lim_\leftarrow A'/t^nA'$.
Now 
$A'/t^nA' 
 = (A \otimes_{\wh R_U} \wh R_{U_1}) \otimes_{\wh R_U} \wh R_U/(t^n)
 = \bigl(A \otimes_{\wh R_U} \wh R_U/(t^n)\bigr) \otimes_{\wh R_U/(t^n)} \bigl(\wh R_{U_1} \otimes_{\wh R_U} \wh R_U/(t^n)\bigr)
 = A/t^nA \otimes_{\wh R_U/(t^n)} \wh R_{U_1}/(t^n)$.
But $A/t^nA \subseteq A_n$ via the injection $\alpha_n$.  
Also, $\Spec(\wh R_{U_1}/(t^n))$ is an affine open subset of $\Spec(\wh R_U/(t^n))$ by \cite[Proposition~5.1.9]{EGA1},
and so $\wh R_{U_1}/(t^n)$ is flat over $\wh R_U/(t^n)$.  We thus obtain 
an inclusion 
\[A'/t^n A' = A/t^nA \otimes_{\wh R_U/(t^n)} \wh R_{U_1}/(t^n) \subseteq A_n \otimes_{\wh R_U/(t^n)} \wh R_{U_1}/(t^n).\]

Let $W_n := \Spec(\wh R_{U_1}/(t^n)) \cup \Spec(\wh R_{U_2}/(t^n)) \subseteq \Spec(\wh R_U/(t^n))$ and write
$f_n:W_n \to \Spec(\wh R_U/(t^n))$ for the natural inclusion.  
Thus $(f_n)_*(\mc O_{W_n})$ is a quasi-coherent sheaf on $\Spec(\wh R_U/(t^n))$, and its module of global sections is 
\[\mc O_{W_n}(\Spec(\wh R_{U_1}/(t^n)) \cap \mc O_{W_n}(\Spec(\wh R_{U_2}/(t^n))
= \wh R_{U_1}/(t^n) \cap \wh R_{U_2}/(t^n) = A_n.\]  
Hence $(f_n)_*(\mc O_{W_n}) = \til A_n$, the sheaf on $\Spec(\wh R_U/(t^n))$ associated to the $\wh R_U/(t^n)$-module $A_n$.
So for any affine open subset $\Spec(S) \subseteq \Spec(\wh R_U/(t^n))$, we have
$\Gamma(\Spec(S),(f_n)_*(\mc O_{W_n})) = A_n \otimes_{\wh R_U/(t^n)} S$.  In particular, 
\[A_n \otimes_{\wh R_U/(t^n)} \wh R_{U_1}/(t^n) = \Gamma(\Spec(\wh R_{U_1}/(t^n)),(f_n)_*(\mc O_{W_n})) = \Gamma(\Spec(\wh R_{U_1}/(t^n)),\mc O_{W_n}) = \wh R_{U_1}/(t^n).\]  Thus we have an inclusion $A'/t^nA' \subseteq \wh R_{U_1}/(t^n)$.  
Since taking inverse limits is left exact, it follows that 
$A' \subseteq  \wh R_{U_1}$.
Tensoring this inclusion over $\wh R_{U_1}$ with the fraction field $K_{U_1}$ of $\wh R_{U_1}$ (which is flat over $\wh R_{U_1}$), we find that $K_{U_1}$ contains $A' \otimes_{\wh R_{U_1}} K_{U_1} = A \otimes_{\wh R_U}\wh R_{U_1} \otimes_{\wh R_{U_1}} K_{U_1}
= A \otimes_{\wh R_U} K_{U_1}$.
Hence $K_{U_1}^d = K_U^d \otimes_{K_U} K_{U_1} = (A \otimes_{\wh R_U} K_U) \otimes_{K_U} K_{U_1} = A \otimes_{\wh R_U} K_{U_1} \subseteq K_{U_1}$, as $K_{U_1}$-vector spaces.  Thus $d=1$, completing the proof.
\end{proof}

\begin{rem} \label{intersection rk}
\begin{enumerate} [(a)]
\item
If the normality assumption is dropped from the hypotheses of the proposition, then the conclusion need not hold.  For example, suppose that $\ms X$ is an affine integral $T$-variety with closed fiber $X=U$, with $U_0,U_1,U_2$ as before, such that $\ms X$ is normal at the points of $W = U_1 \cup U_2$ but not at all the points of $U$.  Then $\wh R_{U_1}, \wh R_{U_2}$ are normal, and hence so is their intersection.  But $\wh R_U$ is not normal, and hence is strictly smaller than $\wh R_{U_1} \cap \wh R_{U_2}$.
\item
Similarly, the connectedness assumption cannot be dropped.  For example, suppose that $\ms X$ is a normal integral projective $T$-variety
such that the reduced closed fiber $X$ is a union of two copies of $\mbb P_k^2$ meeting at a single point $P$.  Let $U \subseteq X$ be the union of two copies of $\mbb A^2_k$ meeting at $P$, and let $U_1$ (resp., $U_2$) be the union of the complements of the $x$-axes (resp., the $y$-axes) in the two copies of $\mbb A^2_k$.  Then $\wh R_{U_1}= k[x,y,y^{-1}][[t]]^{\oplus 2}$ and $\wh R_{U_2}= k[x,x^{-1},y][[t]]^{\oplus 2}$.  So
$\wh R_{U_1} \cap \wh R_{U_2} = k[x,y][[t]]^{\oplus 2}$, which is strictly larger than $\wh R_U$, the difference being that the spectrum of the former consists of two disjoint copies of a thickened $\mbb A^2_k$.  
\end{enumerate}
\end{rem}

\begin{prop} \label{cod 2 struc shf}
Let $T$ be a complete discrete valuation ring with residue field $k$,
and let~$\ms X$ be a quasi-projective normal integral $T$-scheme with  
reduced closed fiber~$X$ and formal completion $\mf X$.  Let 
$f:V \hookrightarrow U$ be an inclusion of connected open subsets of $X$
such that the complement of $V$ in $U$ has codimension at least two,
and write $\mf V, \mf U$ for the formal open subschemes of $\mf X$ associated to $V,U$.  
Then $\wh f_* \mc O_{\mf V} = \mc O_{\mf U}$.
\end{prop}

\begin{proof}
We wish to show that $\mc O_{\mf U}(U') = \wh f_* \mc O_{\mf V}(U')$ for every connected affine open subset $U' \subseteq U$.
By Lemma~\ref{exist 2 conn opens}, there exist affine dense open subsets $U_0,U_1,U_2 \subseteq U'$
with $U_0 = U_1 \cap U_2$, such that the complement of $U_1 \cup U_2$ in $U'$ has codimension at least two, and such that $V_e := V \cap U_e$ is connected for $e=0,1,2$.  
Note that $V_e$ is dense in $V \cap U'$ and hence in $U'$.
Also, the complement of $W := V_1 \cup V_2$ in $U'$ has codimension at least two.  
Let $g:W \hookrightarrow V$ be the natural inclusion.   Thus $fg$ is the natural inclusion $W \hookrightarrow U$, and let $\mf W$ be the formal open subscheme of $\mf X$ associated to $W$.
By Proposition~\ref{intersection}, $\wh R_{U'} = \wh R_{V_1} \cap \wh R_{V_2}  \subseteq \wh R_{V_0}$.  That is, $\mc O_{\mf U}(U') = \mc O_{\mf U}(W) = \mc O_{\mf W}(W) = 
\wh {(fg)}_* \mc O_{\mf W}(U')$.  By Lemma~\ref{pull push}, 
$\mc O_{\mf U}$ is a subsheaf of $\wh f_* \wh f^* \mc O_{\mf U}$, and $\wh f^* \mc O_{\mf U}$ is a subsheaf of $\wh g_* \wh g^* \wh f^* \mc O_{\mf U}$; hence also $\wh f_* \wh f^* \mc O_{\mf U}$ is a subsheaf of $\wh f_* \wh g_* \wh g^* \wh f^* \mc O_{\mf U} =  \wh {(fg)}_*  \wh {(fg)}^* \mc O_{\mf U}$.  Thus $\mc O_{\mf U}(U') \subseteq \wh f_* \wh f^* \mc O_{\mf U}(U') \subseteq  \wh {(fg)}_*  \wh {(fg)}^* \mc O_{\mf U}(U') = \wh {(fg)}_* \mc O_{\mf W}(U') = \mc O_{\mf U}(U')$.  Hence these containments are equalities, and we have $\mc O_{\mf U}(U')  = \wh f_* \wh f^* \mc O_{\mf U}(U') = \wh f_* \mc O_{\mf V}(U')$, as asserted.
\end{proof}

In the situation of the above proposition, if we take $\ms G = \mc O_{\mf U}$, then $\wh f^*\ms G = \mc O_{\mf V}$ and so $\wh f_* \wh f^*\ms G=\ms G$ by the proposition.
But as the following example shows, for a general torsion-free coherent sheaf $\ms G$ on $\mf U$, the containment $\ms G \subseteq \wh f_* \wh f^*(\ms G)$ in Lemma~\ref{pull push} need not be an equality (even in the above situation 
where the complement of $V$ in $U$ has codimension at least two).

\begin{ex}  \label{strict containment ex}
Let $k$ be a field, let $T=k[[t]]$, and let $\ms X = \mbb P^2_T$, the projective $x,y$-plane over $T$, with closed fiber $X = \mbb P^2_k$.  Let $U = \mbb A^2_k \subset X$ and let $V \subset U$ be the complement of the origin, with inclusion morphism $f:V \to U$.  Thus $V = U_1 \cup U_2$, where $U_1, U_2 \subset U$ are the complements in $U$ of the $x$- and $y$-axes, respectively.  Let $\mf U, \mf V, \mf U_i$ be the formal schemes associated to $U,V,U_i$.  Thus $\mf U = \Spf(k[x,y][[t]])$.  Let $I$ be the ideal $(x,y) \subset k[x,y][[t]]$, and let $\ms G = I^\Delta$ be the coherent formal $\mc O_{\mf U}$-module associated to $I$ (see the beginning of Section~\ref{pushforwards}).  Note that $\ms G$ is torsion-free, but not flat (since it is not locally free).
The pullback $\wh f^*(\ms G)$ to~$\mf V$ is the structure sheaf $\mc O_{\mf V}$ (and so is free), and the pushforward $\wh f_* \wh f^*(\ms G)$ is the structure sheaf on $\mf U$.  Thus $\ms G \subseteq \wh f_* \wh f^*(\ms G)$ is a strict containment.
\end{ex} 

As a consequence, a torsion-free coherent sheaf $\ms F$ on $\mf V$ can be the pullback of more than one torsion-free coherent sheaf $\ms G$ on $\mf U$.  But if $\mf U$ is affine, and if $\ms F$ is a pullback of a torsion-free coherent sheaf $\ms G$ on $\mf U$, then there is a maximal such $\ms G$ in the sense that all others are subsheaves of it.  We prove this in Proposition~\ref{max coh} below.  First we state two lemmas.

Given a domain $R$ that is complete with respect to an ideal $I$, and given $R$-modules $M,N$, write $\wh M$ and $M \,\wh \otimes_R\, N$ for the $I$-adic completions of $M$ and $M \otimes_R N$, respectively.

\begin{lem}  \label{completed tensor la}
Let $T$ be a complete discrete valuation ring with uniformizer $t$, let $\ms X$ be a normal integral $T$-scheme of finite type, and let
$\mf X$ be the formal completion of $\ms X$ along its reduced closed fiber $X$.  
Let $W \subseteq U,V \subseteq Z$ be affine open subsets of $X$, with $U,V$ dense subsets of $Z$ and $W = U \cap V$, and let $\mf W, \mf U, \mf V, \mf Z$ be the associated formal subschemes of $\mf X$.  
\begin{enumerate} [(a)]
\item 
If $\ms F$ is a coherent sheaf on $\mf Z$, then the canonical map $\ms F(U) \,\wh \otimes_{\wh R_Z} \wh R_V \to \ms F(W)$ is an isomorphism.
\item
In particular, 
$\wh R_U \,\wh \otimes_{\wh R_Z} \, \wh R_V = \wh R_W$.
\end{enumerate}
\end{lem}

\begin{proof}
It suffices to prove part~(a), since part~(b) is the special case $\ms F = \mc O_{\mf Z}$.

For every $n \ge 1$, write $W_n, U_n, V_n, Z_n$ for the reductions of $\mf W, \mf U, \mf V, \mf Z$ modulo $t^n$, and let $F_n$ be the reduction of $\ms F$ modulo $t^n$.  By \cite[Th\'eor\`eme~10.11.3]{EGA1}, $F_n$ is a 
coherent sheaf on the scheme $Z_n$ and $\displaystyle\lim_\leftarrow F_n(U_n) = \ms F(U)$.  Similarly, $\displaystyle \lim_\leftarrow \mc O(V_n)=  \mc O_{\mf X}(V) = \wh R_V$ and $\displaystyle\lim_\leftarrow \mc O(Z_n) = \wh R_Z$.
Since 
$F_n$ is coherent, $F_n(U_n) = F_n(Z_n) \otimes_{\mc O(Z_n)} \mc O(U_n)$.  
Also $\mc O(W_n) = \mc O(U_n) \otimes_{\mc O(Z_n)} \mc O(V_n)$ since $U_n,V_n$ are affine open subsets of $Z_n$ with intersection~$W_n$ (see \cite[Proposition~5.1.9]{EGA1}).
So 
$F_n(U_n) \otimes_{\mc O(Z_n)} \mc O(V_n) = F_n(Z_n) \otimes_{\mc O(Z_n)} \mc O(U_n) \otimes_{\mc O(Z_n)} \mc O(V_n) = F_n(Z_n) \otimes_{\mc O(Z_n)} \mc O(W_n) = F_n(W_n)$ by the coherence of $F_n$.
Taking inverse limits yields the assertion of part~(a).
\end{proof}

\begin{lem} \label{complete tensor}
Let $R$ be a Noetherian domain that is complete with respect to a principal ideal $(t)$.  Let $P \to Q$ be a homomorphism of $R$-modules, with $Q$ flat over $R$.  Let $M$ be an  
$R$-module that is $t$-torsion-free, and write $N = \ker(M \otimes_R P \to M \otimes_R Q)$.
Then $\wh N \subseteq \ker(M \,\wh \otimes_R\, P \to M \,\wh \otimes_R\, Q)$.
\end{lem}  

\begin{proof}
Since $M$ is $t$-torsion-free, multiplication by $t$ is injective on $M$.  But $Q$ is flat over $R$.  Hence multiplication by $t$ is also injective on $M \otimes_R Q$; i.e., $M \otimes_R Q$ is also $t$-torsion-free.

We claim that the $t$-adic topology of $N$ coincides with the topology induced by the $t$-adic topology on $M \otimes_R P$.  To see this, take $i \ge 1$.  If $n \in N \cap t^i(M \otimes_R P)$, then we may write $n = t^im$ for some $m \in M \otimes_R P$.  Let $m'$ be the image of $m$ in $M \otimes_R Q$.  Then $t^im'$ is the image of $t^im=n$ in $M \otimes_R Q$, and hence $t^im'=0$ by definition of $N$.
But $M \otimes_R Q$ is $t$-torsion-free; so $m'=0$.  
Thus $m \in N$ and so $n = t^im \in t^iN$.
Hence $N \cap t^i(M \otimes_R P) = t^iN$, proving the claim.

Let $C$ be the image of $M \otimes_R P \to M \otimes_R Q$; this is the 
cokernel of $N \to M \otimes_R P$.
Thus the sequence 
$0 \to N \to M \otimes_R P \to C \to 0$ is exact, with $C \subseteq M \otimes_R Q$.
By \cite[Theorem~8.1]{Mats89}, together with the above claim, the sequence $0 \to \wh N \to M \,\wh \otimes_R\, P \to \wh C \to 0$ is exact.  Also, the inclusion $C \to M \otimes_R Q$ induces a map $\wh C \to M \,\wh \otimes_R\, Q$.
Hence we have $\wh N = \ker(M \,\wh \otimes_R\, P \to \wh C) \subseteq \ker(M \,\wh \otimes_R\, P \to M \,\wh \otimes_R\, Q)$.
\end{proof}

\begin{prop} \label{max coh} 
Let $T$ be a complete discrete valuation ring with residue field $k$,
and let~$\ms X$ be a quasi-projective normal integral $T$-scheme with  
reduced closed fiber~$X$ and formal completion $\mf X$.  Let 
$f:V \hookrightarrow U$ be an inclusion of connected open subsets of $X$ with $U$ affine,
such that the complement of $V$ in $U$ has codimension at least two,
and write $\mf V, \mf U$ for the formal open subschemes of $\mf X$ associated to $V,U$.
Let $\ms F$ be a torsion-free coherent sheaf on~$\mf V$ which is the pullback of a torsion-free coherent sheaf on $\mf U$.  Then $\ms F(V)$ is a finite $\wh R_U$-module, and so we can consider the coherent sheaf $\ms G = \ms F(V)^\Delta$ on $\mf U = \Spf(\wh R_U)$.  Then~$\ms G$ is torsion-free, $\wh f^*\ms G = \ms F$, and every torsion-free coherent sheaf on $\mf U$ that pulls back to $\ms F$ is a subsheaf of~$\ms G$.  
\end{prop}

\begin{proof}
Since $\ms F(V) = \wh f_*\ms F(U)$, it follows by Theorem~\ref{genl fin ext} that $\ms F(V)$ is a finite torsion-free $\wh R_U$-module, and hence $\ms G$ is defined, coherent, and torsion-free on  $\mf U$. 

By hypothesis, there exists a finite torsion-free coherent $\wh R_U$-module that pulls back to $\ms F$ on $\mf V$.  For any such $\wh R_U$-module $\ms H$, we have 
$\ms H(U) \subseteq \wh f_* \wh f^* \ms H(U) = \wh f_* \ms F(U) = \ms F(V) = \ms G(U)$
by Lemma~\ref{pull push} together with the condition that $\wh f^* \ms H = \ms F$. So for every affine open subset $W \subseteq U$, we have 
$\ms H(W) = \ms H(U) \otimes_{\wh R_U} \wh R_W \subseteq \ms G(U)  \otimes_{\wh R_U} \wh R_W = \ms G(W)$, by Proposition~\ref{pushforward torsion-free}(b) and the flatness of $\wh R_W$ over $\wh R_U$ (see Lemma~\ref{flat patches}(b)).  This proves the assertion that any such $\ms H$ is a subsheaf of $\ms G$.

To prove the proposition, it remains to show that $\wh f^*\ms G  = \ms F$.  Pick any $\ms H$ as above.  Since $\wh f^* \ms H = \ms F$, it suffices to show that $\ms G(W) = \ms H(W)$ for every affine open subset $W \subseteq V$.  We have already shown that $\ms H(W) \subseteq \ms G(W)$ for every such $W$.  So it suffices to show that $\ms G(W) \subseteq \ms H(W)$ for every affine open subset $W \subseteq V$.

Fix such a $W$.  Since $X$ is Noetherian, $V$ is the union of finitely many affine open subsets $U_i \subseteq U$.  Thus $W \subseteq V$ is the union of the affine open subsets $W_i := W \cap U_i$. For $i<j$ write $U_{ij} := U_i \cap U_j$ and $W_{ij} = W \cap U_{ij} = W_i \cap W_j$.  Thus
\[ \ms G(U) = \ms F(V)
= \wh f^*\ms H(V) = \ms H(V) = \ker\bigl(\bigoplus_i \ms H(U_i) \to \bigoplus_{i,j} \ms H(U_{ij})\bigr),\]
where the arrow is given by taking differences.

Since $\wh R_W$ is a flat $\wh R_U$-algebra, 
by tensoring the left exact sequence
\[0 \to \ms G(U) \to  \bigoplus_i \ms H(U_i) \to \bigoplus_{i,j} \ms H(U_{ij})\]
over $\wh R_U$ with $\wh R_W$ we obtain
\[0 \to  \wh R_W\otimes_{\wh R_U} \ms G(U)  \to \wh R_W\otimes_{\wh R_U} \bigoplus_i \ms H(U_i) \to  \wh R_W\otimes_{\wh R_U} \bigoplus_{i,j} \ms H(U_{ij}).\]
That is, 
\[ \wh R_W\otimes_{\wh R_U} \ms G(U) = \ker\bigl(\wh R_W \otimes_{\wh R_U} \ms H(U) \otimes_{\wh R_U} \bigoplus_i \wh R_{U_i} \to \wh R_W \otimes_{\wh R_U} \ms H(U) \otimes_{\wh R_U} \bigoplus_{i,j} \wh R_{U_{ij}}\bigr).\]
We now apply Lemma~\ref{complete tensor}, taking $R=\wh R_U$, 
$P = \bigoplus_i \wh R_{U_i}$, $Q =\bigoplus_{i,j} \wh R_{U_{ij}}$, $M  = \ms H(W) = \wh R_W \otimes_{\wh R_U} \ms H(U)$, and $N = \ms G(W)  = \wh R_W\otimes_{\wh R_U} \ms G(U)$.
We obtain 
\[ \ms G(W)
\subseteq  \ker\bigl( \ms H(W) \,\wh\otimes_{\wh R_U}\,\bigoplus_i \wh R_{U_i}  \to  
 \ms H(W) \,\wh\otimes_{\wh R_U}\, \bigoplus_{i,j}\wh R_{U_{ij}}\bigr).\]
 By Lemma~\ref{completed tensor la}(a), we have that
 $\ms H(W) \,\wh\otimes_{\wh R_U}\, \wh R_{U_i} = \ms H(W_i)$ and 
$\ms H(W) \,\wh\otimes_{\wh R_U}\, \wh R_{U_{ij}} = \ms H(W_{ij})$.
 So $\ms G(W) \subseteq \ker(\bigoplus_i \ms H(W_i) \to \bigoplus_{i,j} \ms H(W_{ij})\bigl) = \ms H(W)$, as needed. 
\end{proof}

\subsection*{Pushforwards and pullbacks of reflexive modules}

Recall that a sheaf $\ms G$ on $\mf U$ is {\it reflexive} if it is equal to its double dual $\ms G^{\vee\vee}$, where the dual $\ms G^{\vee}$ is defined to be $\ms H\!om_{\mc O_{\mf U}}(\ms G,\mc O_{\mf U})$.  Reflexivity is a local property of sheaves; reflexive sheaves are torsion-free (by \cite[Lemma~0AV0]{stacks}); 
and $\ms G^{\vee}$ is reflexive for all coherent $\mc O_{\mf U}$-modules $\ms G$ (by \cite[Lemma~0AV3]{stacks}).  
In particular, the {\it reflexive hull} $\ms G^{\vee\vee}$ of a coherent $\mc O_{\mf U}$-module $\ms G$ is necessarily reflexive; and there is a natural morphism $\ms G \to \ms G^{\vee\vee}$.  

If $\ms M, \ms N$ are coherent sheaves on $\mf U$, then so is $\ms H\!om_{\mc O_{\mf U}}(\ms M, \ms N)$ by \cite[10.10.2.2]{EGA1}, since coherence is a local condition and since coherent sheaves are locally of the form $M^\Delta$.  Hence if $\ms G$ is coherent on $\mf U$, then so are $\ms G^\vee$ and $\ms G^{\vee\vee}$.  Moreover if $U$ (or equivalently, $\mf U$) is affine, then $\ms G^\vee(U) = \ms G(U)^\vee$, again by \cite[10.10.2.2]{EGA1}.

\begin{lemma} \label{pull push dual}
Let $T$ be a complete discrete valuation ring,
let $\ms X$ be a $T$-scheme with reduced closed fiber~$X$, and let $f:V \hookrightarrow U$ be an inclusion of open subsets of $X$.  Let $\mf X$ be the formal completion of $\ms X$, 
and write $\mf V, \mf U$ for the formal open subschemes of $\mf X$ associated to $V,U$, with induced map $\wh f:\mf V \to \mf U$.  
If $\ms G$ is a coherent sheaf on $\mf U$, then $\wh f^*(\ms G^\vee) = (\wh f^*\ms G)^\vee$.
\end{lemma}

\begin{proof}
Let $O$ be an affine open subset of $V$.  Then 
$\wh f^*(\ms G^\vee)(O) = \ms G^\vee(O) = \ms G(O)^\vee 
= (\wh f^*\ms G(O))^\vee = (\wh f^*\ms G)^\vee(O).$  So $\wh f^*(\ms G^\vee) = (\wh f^*\ms G)^\vee$.
\end{proof}

\begin{prop} \label{pushforward reflex}
Let $T$ be a complete discrete valuation ring with residue field $k$, and let $\ms X$ be a quasi-projective normal integral $T$-scheme with  
reduced closed fiber~$X$ and formal completion $\mf X$.  Let 
$f:V \hookrightarrow U$ be an inclusion of connected open subsets of $X$
such that the complement of $V$ in $U$ has codimension at least two,
and write $\mf V, \mf U$ for the formal open subschemes of $\mf X$ associated to $V,U$. 
\begin{enumerate} [(a)]
\item \label{pushforward coh reflex}
Let $\ms F$ be a coherent sheaf on $\mf V$ such that $\wh f_* \ms F$ is coherent.  Then $\ms F$ is reflexive if and only if $\wh f_* \ms F$ is reflexive.
\item \label{pullback coh reflex}
Let $\ms G$ be a reflexive coherent sheaf on $\mf U$.   Then $\wh f^* \ms G$ is reflexive on $\mf V$.
\end{enumerate}
\end{prop}

\begin{proof}
For the forward direction of part~(\ref{pushforward coh reflex}), suppose that $\ms F$ is reflexive on $\mf V$.  Since $\wh f_*\ms F$ is coherent, $\ms F = \wh f^* \wh f_*\ms F$ is the pullback of a coherent sheaf on $\mf U$.  Hence Proposition~\ref{max coh} applies.  Thus $\ms G = \ms F(V)^\Delta$ is the maximum torsion-free coherent sheaf on $\mf U$ that pulls back to $\ms F$ on $\mf V$.  But by Lemma~\ref{pull push dual} together with the fact that $\ms F$ is reflexive, $\ms G^{\vee\vee}$ is also a torsion-free coherent sheaf on $\mf U$ that pulls back to $\ms F$.  Hence $\ms G^{\vee\vee} = \ms G$; i.e., $\ms G$ is reflexive. 
But $\ms G$ is the unique coherent sheaf on $\mf U$ such that $\ms G(U) = \ms F(V) = \wh f_* \ms F(U)$.  Since $\wh f_* \ms F$ is coherent, it follows that $\ms G= \wh f_* \ms F$.  So indeed $\wh f_* \ms F$ is reflexive.  

For the reverse direction of part~(\ref{pushforward coh reflex}), 
if $\wh f_*\ms F$ is reflexive, then $\ms F^{\vee\vee} = (\wh f^*\wh f_* \ms F)^{\vee\vee} = \wh f^*((\wh f_* \ms F)^{\vee\vee}) = \wh f^*\wh f_* \ms F = \ms F$ by Lemma~\ref{pull push dual}, and so $\ms F$ is reflexive.

For part~(\ref{pullback coh reflex}), 
since $\ms G$ is reflexive, it follows that $(\wh f^* \ms G)^{\vee\vee} = \wh f^* (\ms G^{\vee\vee}) = \wh f^* \ms G$ by Lemma~\ref{pull push dual}.  Thus $\wh f^* \ms G$ is reflexive.
\end{proof}

\begin{lem} \label{open nbd lem}
Let $\wh R$ be a Noetherian domain that is complete with respect to an ideal~$I$, let $U = \Spec(\wh R/I)$, let $\wh U = \Spec(\wh R)$, and let $\mf U = \Spf(\wh R)$.  Consider a Zariski open subset $f:V \hookrightarrow U$, with associated formal completion $\wh f:\mf V \to \mf U$.  Let $\phi:M \to N$ be a morphism of finite $\wh R$-modules, corresponding to a morphism $\til \phi:\til M \to \til N$ of coherent $\mc O_{\wh U}$-modules and inducing a morphism $\wh \phi:M^\Delta \to N^\Delta$ of coherent $\mc O_{\mf U}$-modules.  Suppose that $\wh f^*\wh \phi:\wh f^* M^\Delta \to \wh f^* N^\Delta$ is an isomorphism.  Then there exists a Zariski open subset $j:O \hookrightarrow \wh U$ with $V \subseteq O \cap U$ such that $j^*\til\phi:j^* \til M \to j^* \til N$ is an isomorphism.  Moreover the codimension of the complement of $O$ in $\wh U$ is greater than or equal to the codimension of the complement of $V$ in $U$.
\end{lem}

\begin{proof}
Let $P$ be a closed point of $V$.  The isomorphism $\wh f^*\wh \phi:\wh f^* M^\Delta \to \wh f^* N^\Delta$ induces  an isomorphism on the completed stalks at $P$; i.e., over $\wh{\mc O}_{\mf V,P} = \wh{\mc O}_{\mf U,P} =  \wh{\mc O}_{\wh U,P}$.  Hence so does $\til \phi:\til M \to \til N$.
Since $\mc O_{\wh U,P}$ is a Noetherian local ring, 
$\wh{\mc O}_{\wh U,P}$ is faithfully flat over $\mc O_{\wh U,P}$ by \cite[Proposition~III.3.5.9, Definition~III.3.3.2]{Bo:CA}.
It then follows 
by \cite[Proposition~I.3.1.2]{Bo:CA}
that $\til \phi:\til M \to \til N$ induces an isomorphism over the local ring $\mc O_{\wh U,P}$ of $\wh U$ at $P$.  
Since $M$ is finitely presented, $\til \phi:\til M \to \til N$ also induces an isomorphism over some Zariski open neighborhood $O_P$ of $P$ in $\wh U$.  

Let $O \subseteq \wh U$ be the union of the open subsets $O_P \subseteq \wh U$, as $P$ ranges over the closed points of $V$, and let $j:O \hookrightarrow \wh U$ be the natural inclusion.  Thus $V \subseteq O \cap U$, and $\til \phi:\til M \to \til N$ induces an isomorphism 
$j^*\til M \to j^*\til N$ over $O$.  This proves the first assertion.

For the second assertion, let $d$ be the codimension of the complement of $V$ in $U$.  Thus the complement of $O \cap U$ in $U$ has codimension at least $d$.  Since every non-empty closed subset of $\wh U$ having codimension less than $d$ meets $U$ in a closed subset of codimension less than $d$, the assertion follows.
\end{proof}

\begin{thm} \label{reflexive pull push}
Let $T$ be a complete discrete valuation ring with residue field $k$,
and let~$\ms X$ be a quasi-projective normal integral $T$-scheme with  
reduced closed fiber~$X$ and formal completion $\mf X$.  Let 
$f:V \hookrightarrow U$ be an inclusion of connected open subsets of $X$
such that the complement of $V$ in $U$ has codimension at least two,
and write $\mf V, \mf U$ for the formal open subschemes of $\mf X$ associated to $V,U$.  
If $\ms G$ is a coherent sheaf on $\mf U$ such that $\wh f^* \ms G$ is reflexive, then $\wh f_* \wh f^* \ms G$ is equal to the reflexive hull of $\ms G$ and thus
is a coherent sheaf on $\mf U$.
\end{thm}

\begin{proof}
Since the assertion is local, we may assume that $U = \ms U \cap X$ for some affine open subset $\ms U = \Spec(R) \subseteq \ms X$.  Let $\wh R$ be the completion of $R$ with respect to the valuation induced by that on $T$.
Thus $\mf U = \Spf(\wh R)$ and $\ms G = M^\Delta$, where $M = \ms G(U)$, a finite $\wh R$-module.  Moreover $\ms G$ is the formal sheaf induced by the coherent sheaf $\til M$ on the scheme $\wh U := \Spec(\wh R)$.  By the reflexivity hypothesis and by Lemma~\ref{pull push dual}, the natural morphism $\wh f^* M^\Delta \to (\wh f^* M^\Delta)^{\vee\vee} = \wh f^*((M^\Delta)^{\vee\vee}) = \wh f^*((M^{\vee\vee})^\Delta)$ is an isomorphism.  So by Lemma~\ref{open nbd lem}, there is a Zariski open subset $\til f:\wh V \hookrightarrow \wh U$ with $V \subseteq \wh V \cap U$, where the complement of $\wh V$ in $\wh U$ has codimension at least two, such that $\til f^*\phi:\til f^* \til M \to \til f^* \til M^{\vee\vee}$ is an isomorphism, where $\phi:\til M \to \til M^{\vee\vee}$ is the natural map.  That is, the coherent sheaf $\til f^* \til M$  on $\wh V$ is reflexive and in particular torsion-free.

Observe that by \cite[Proposition~1.6]{Hts80}, it follows that $\til f_* \til f^* \til M$ is reflexive on $\wh U$.  Namely, that result says that on a normal integral scheme, a torsion-free coherent sheaf $\ms M$ if reflexive if and only if the restriction map $\ms M(O) \to \ms M(O')$ is bijective for every choice of open sets $O' \subseteq O$ such that the complement of $O'$ in $O$ has codimension at least two.  If we take such $O' \subseteq O \subseteq U$, then $\til f_* \til f^* \til M(O) = \til f^* \til M(O \cap V)$, and $\til f_* \til f^* \til M(O') = \til f^* \til M(O' \cap V)$.  Since the torsion-free coherent sheaf $\til f^* \til M$ is reflexive and since the complement of $O' \cap V$ in $O \cap V$ has codimension at least two, it follows that 
$\til f_* \til f^* \til M(O) = \til f_* \til f^* \til M(O')$.  Since $\til f_* \til f^* \til M$ is a torsion-free coherent sheaf by Theorem~\ref{scheme pushforward}, it is indeed reflexive by \cite[Proposition~1.6]{Hts80}.

Thus, by taking double duals, the natural morphism $\til M \to \til f_* \til f^* \til M$ induces a morphism $\til M^{\vee\vee} \to (\til f_* \til f^* \til M)^{\vee\vee} = \til f_* \til f^* \til M$.
Via formal completion, this morphism induces a morphism $\ms G^{\vee\vee} \to \wh f_* \wh f^* \ms G$; and we wish to show that this is an isomorphism.

Observe that the morphism $\ms G^{\vee\vee} \to \wh f_* \wh f^* \ms G$ becomes an isomorphism upon restriction to~$\mf V$.  Namely, $\wh f^*\wh f_*$ is the identity since $V$ is open in $U$, and so $\wh f^*(\wh f_* \wh f^* \ms G) = \wh f^* \ms G$.  Meanwhile, $\wh f^*( \ms G^{\vee\vee}) = (\wh f^*\ms G)^{\vee\vee} = \wh f^*\ms G$, by Lemma~\ref{pull push dual} and the assumption that $\wh f^* \ms G$ is reflexive.  So indeed the restriction
$\wh f^*\ms G^{\vee\vee} \to \wh  f^*\wh f_* \wh f^* \ms G = \wh f^* \ms G$
of $\ms G^{\vee\vee} \to \wh f_* \wh f^* \ms G$ to $\mf V$ is an isomorphism.
Since $\ms G = M^\Delta$, 
by Lemma~\ref{open nbd lem} there is a Zariski open subset $j:\wh O \hookrightarrow \wh U$ with $V \subseteq \wh O \cap U$, where the codimension of the complement $\wh Z$ of $\wh O$ in $\wh U$ is at least two, such that $\til M^{\vee\vee} \to \til f_* \til f^* \til M$ induces an isomorphism 
$j^*(\til M^{\vee\vee}) \to j^*\til f_* \til f^* \til M$
over $\wh O$.

Now $\til M^{\vee\vee}$ is a reflexive coherent sheaf by \cite[Remark~0EBH]{stacks}, and it was observed above that $\til f_* \til f^* \til M$ is also a reflexive coherent sheaf.  Since the complement $\wh Z$ of $\wh O$ has codimension at least two in $\wh U$, \cite[Proposition~1.6]{Hts80} implies that  
$\til M^{\vee\vee} = j_*j^*(\til M^{\vee\vee})$ and 
$\til f_* \til f^* \til M = j_*j^*\til f_* \til f^* \til M$.
Hence by applying $j_*$, we see that the above isomorphism $j^*(\til M^{\vee\vee}) \to j^*\til f_* \til f^* \til M$ induces an isomorphism $\til M^{\vee\vee} \to \til f_* \til f^* \til M$.  
Passing to the formal completions, this in turn induces an isomorphism $(M^\Delta)^{\vee\vee} \to \wh f_*\wh f^*(M^\Delta)$.  That is, $\ms G^{\vee\vee} \to \wh f_* \wh f^* \ms G$ is an isomorphism; i.e., $f_* \wh f^* \ms G$ is the reflexive hull of $\ms G$.
Since the reflexive hull of a coherent sheaf on $\mf U$ is coherent (as noted before Lemma~\ref{pull push dual}), $\wh f_* \wh f^* \ms G$ is coherent.
\end{proof}

Thus, in Theorem~\ref{reflexive pull push}, if $\ms G$ is not reflexive, then the containment $\ms G \subseteq \wh f_* \wh f^*(\ms G)$ in Lemma~\ref{pull push} is strict, since $\wh f_* \wh f^*(\ms G)$ is reflexive.  This is illustrated in 
Example~\ref{strict containment ex} above.  In contrast, if the coherent formal sheaf $\ms G$ in Lemma~\ref{pull push} is assumed to be reflexive (e.g., locally free), and not just torsion-free, and if $V$ is connected, then the containment $\ms G \subseteq \wh f_* \wh f^*\ms G$ is an equality, by the following corollary.

\begin{cor} \label{reflexive pp cor}
Let $T$ be a complete discrete valuation ring with residue field $k$,
and let $\ms X$ be a quasi-projective normal integral $T$-scheme with  
reduced closed fiber~$X$ and formal completion $\mf X$.  Let 
$f:V \hookrightarrow U$ be an inclusion of connected open subsets of $X$
such that the complement of $V$ in $U$ has codimension at least two,
and write $\mf V, \mf U$ for the formal open subschemes of $\mf X$ associated to $V,U$.  
If $\ms G$ is a reflexive coherent sheaf on $\mf U$, then 
$\wh f_* \wh f^* \ms G = \ms G$.
\end{cor}

\begin{proof}
Since $\ms G$ is reflexive, so is $\wh f^* \ms G$, by Proposition~\ref{pushforward reflex}(\ref{pullback coh reflex}).
Hence $\wh f_* \wh f^* \ms G$ is the reflexive hull of $\ms G$ by Theorem~\ref{reflexive pull push}; and so the 
assertion follows from the assumption that $\ms G$ is reflexive.
\end{proof}

\begin{rem}
Concerning the necessity of the connectivity hypothesis in Corollary~\ref{reflexive pp cor}, 
consider the situation in Remark~\ref{intersection rk}(b), and let $V = U_1 \cup U_2$ there,
with inclusion morphism $f:V \to U$.  Then the complement of $V$ in~$U$ (viz., the point $P$) has codimension at least two in $U$, and $V$ is disconnected.  Let $\mf U, \mf V$ be the formal completions of $U, V$, with inclusion map $\wh f: \mf V \to \mf U$.  Then $\wh f_* \wh f^*\mc O_{\mf U}$ is strictly bigger than $\mc O_{\mf U}$, since the former sheaf is ``doubled'' at the point $P$ (corresponding to the strict containment $\wh R_U \subset \wh R_{U_1} \cap \wh R_{U_2}$ in Remark~\ref{intersection rk}(b)).  
\end{rem}

As in the above discussion, let $f:V \hookrightarrow U$ be an inclusion of connected open subsets of the reduced closed fiber $X$ of a normal integral $T$-scheme $\ms X$, with associated formal schemes~$\mf V, \mf U$.  Let $\ms F$ be a torsion-free coherent sheaf on $\mf V$ such that $\wh f_*\ms F$ is a coherent sheaf on $\mf U$.  Then $\wh f_*\ms F$ is also torsion free by Proposition~\ref{pushforward torsion-free}; and if $\ms F$ is moreover reflexive, then $\wh f_*\ms F$ is also reflexive by Proposition~\ref{pushforward reflex}(\ref{pushforward coh reflex}).  It is then natural to ask whether the coherent sheaf $\wh f_*\ms F$ is necessarily locally free (or equivalently, flat) provided that $\ms F$ is.  The analogous assertion for schemes holds in the two-dimensional regular case, since a reflexive sheaf on a regular scheme is locally free except in codimension at least three (see \cite[Lemma~1.4]{Hts80}).  
But as Example~\ref{nonflat pushforward ex} below shows, the assertion for regular schemes fails in dimension three, and it fails for formal schemes even in relative dimension two (where the associated $T$-scheme $\ms X$ has dimension three).

\begin{ex} \label{nonflat pushforward ex}
\begin{enumerate} [(a)]
\item
We first consider the scheme case.
In \cite[Example~1.1.13]{OSS}, the authors give an example of a coherent sheaf $G$ on $\mbb P^3_{\mbb C}$ that is reflexive 
but is not locally free, though it is locally free away from a certain closed point $x_0$ (say $(0:0:0:1)$).  Let $W$ be the complement of $x_0$ in $\mbb P^3_{\mbb C}$, let $f:W \to \mbb P^3_{\mbb C}$ be the inclusion map.  Thus $F = f^* G$ is locally free.  Since $G$ is reflexive, and since the complement of $W$ has codimension at least two (in fact, three), it follows from \cite[Proposition~1.6]{Hts80} that $G \cong f_* F$.  Thus $F$ is a locally free coherent sheaf on $W$, but $f_* F$ is not locally free.

\item
We use the above example to produce an example in the case of a formal scheme of relative dimension two.  Let $\mf X$ be the formal completion of 
$\ms X = \mbb A^2_{\mbb C[[t]]}$ along its closed fiber $X = \mbb A^2_{\mbb C}$.  Let $\pi$ be the composition
$\mf X \to \ms X \to \mbb A^3_{\mbb C} \to \mbb P^3_{\mbb C}$, where $\mbb A^3_{\mbb C}$ is the complement of the hyperplane at infinity in $\mbb P^3_{\mbb C}$ and where $x,y,t$ are the coordinates on $\mbb A^3_{\mbb C}$.  With notation as in part~(a), 
the pullback $\ms G = \pi^*G$ of $G$ to $\mf X$ is a coherent sheaf that is reflexive but not locally free, though it is locally free away from the point $P$ where $x=y=t=0$.  That is, it is locally free over the formal subscheme $\mf V \subseteq \mf X$ associated to the complement $V$ of $x=y=0$ in the closed fiber $X$ of $\mf X$.  Let $\wh f:\mf V \to \mf X$ be the natural inclusion, and consider the coherent sheaf $\ms F = \wh f^* \ms G$ on $\mf V$.  By Corollary~\ref{reflexive pp cor}, $\wh f_* \ms F = \wh f_* \wh f^* \ms G = \ms G$.  So $\ms F$ is a locally free coherent sheaf on $\mf V$ whose pushforward to $\mf X$ is reflexive and coherent but not locally free.
\end{enumerate}
\end{ex}

On the other hand, there is the following positive assertion, which was essentially pointed out to us by V.~Srinivas.

\begin{prop} \label{Weil Cartier}
Let $T$ be a complete discrete valuation ring with residue field $k$,
and let~$\ms X$ be a quasi-projective regular $T$-scheme with  
reduced closed fiber~$X$ and formal completion $\mf X$.  Let 
$f:V \hookrightarrow U$ be an inclusion of connected open subsets of $X$
such that the complement of $V$ in $U$ has codimension at least two,
and write $\mf V, \mf U$ for the formal open subschemes of $\mf X$ associated to $V,U$.  
If $\ms F$ is a locally free coherent sheaf of rank one on $\mf V$ such that $\wh f_*\ms F$ is coherent, then $\wh f_*\ms F$ is locally free of rank one on $\mf U$.
\end{prop}

\begin{proof}
Since the asserted property is local, we may assume that $U$ is an affine open subset of $X$ and that it is of the form $U = \ms U \cap X$ for some affine open subset $\ms U \subseteq \ms X$.  Thus $\mf U = \Spf(\wh R_U)$.
The coherent sheaf $\wh f_*\ms F$ on $\mf U$ is then of the form $M^\Delta$, where $M = \wh f_*\ms F(U) = \ms F(V)$.  Here $\wh f_*\ms F$ is reflexive, by Proposition~\ref{pushforward reflex}(\ref{pushforward coh reflex}).  Thus $M$ is a finite reflexive $\wh R_U$-module, and $\til M$ is a reflexive coherent sheaf on $\wh U := \Spec(\wh R_U)$.  
Now for every closed point $P \in \wh U$ (each of which is in $U$), the natural map
$\wh {\mc O}_{\wh U,P} \to \wh {\mc O}_{\ms X,P}$ is an isomorphism;
and for every closed point $P \in V$, so is $\wh {\mc O}_{\mf V,P} \to \wh {\mc O}_{\wh U,P}$.
So since $\ms X$ is regular, so is $\wh U$; 
and since $\ms F$ has rank one on $\mf V$, 
the reflexive sheaf $\til M$ has rank one on $\wh U$.
Thus by \cite[Proposition~1.9]{Hts80}, $\til M$ is locally free of rank one on $\wh U$.  Hence $M$ is locally free of rank one over $\wh R_U$, and $M^\Delta = \wh f_*\ms F$ is locally free of rank one on $\mf U$.
\end{proof}

\section{Finite vs.\ coherent pushforward} \label{fin vs coh}

In Theorem~\ref{genl fin ext} , we showed that in the context of an inclusion of formal sheaves $f:\mf V \to \mf U$ with complement of codimension at least two, 
the pushforward of a torsion-free coherent sheaf is finite.  And above,  we considered the situation of a torsion-free coherent sheaf $\ms F$ on $\mf V$ whose pushforward $\wh f_*\ms F$ to $\mf U$ is assumed to be coherent, rather than merely finite.  We next consider the question of whether the pushforward $\wh f_*\ms F$ will in fact be coherent.  

For context, consider first the analogous situation for an open inclusion $f:V \hookrightarrow U$ of Noetherian {\em schemes} (as opposed to formal schemes).  There, if a sheaf of modules is finite and quasi-coherent, then it is coherent because quasi-coherent sheaves have the form $\til M$ on every affine open subset, with each $M$ a finite module if and only if the sheaf is coherent.
Hence for $\ms F$ a coherent sheaf on $V$, if $f_*\ms F$ is finite then it is coherent, because the pushforward of a quasi-coherent sheaf under an open inclusion of schemes is always quasi-coherent (see \cite[Proposition~II.5.8(c)]{Hts} or \cite[Lemma~01LC]{stacks}). 
The proof of this last property 
reduces to the case of an inclusion of affine schemes $f:V \hookrightarrow U$, say with $U = \Spec(R)$ and $V = \Spec(S)$.  
There the assertion relies on the above fact that quasi-coherent sheaves on an affine scheme~$Y$ are of the form $\til M$ for some $\mc O(Y)$-module $M$, together with the fact that 
$S \otimes_R S = S$ if $\Spec(S) \hookrightarrow \Spec(R)$ is an open inclusion of affine schemes (i.e., $\mc O(V) \otimes_{\mc O(U)} \mc O(V) = \mc O(V)$ for $V \hookrightarrow U$ affine open).

In the situation of Noetherian formal schemes, by \cite[Lemma~01BK]{stacks}, 
every quasi-coherent sheaf is locally of the form $\mc F_M$ (notation as in Section~\ref{pushforwards}).  Moreover, it remains true that a sheaf of modules that is finite and quasi-coherent is coherent.  Namely, we have the following result.

\begin{prop} \label{coh qcoh fin}
If $\mf U$ is a Noetherian formal scheme, then a sheaf $\ms M$ of $\mc O_{\mf U}$-modules is coherent if and only if it is finite and quasi-coherent.
\end{prop}

\begin{proof}
The forward direction is clear, and so we prove the converse direction.  Since $\ms M$ is quasi-coherent, every point of $\mf U$ has a neighborhood $\mf V$ such that $\ms M|_{\mf V} = \mc F_M$ for some $\mc O_{\mf U}(\mf V)$-module $M$, by \cite[Lemma~01BK]{stacks}. 
After shrinking $\mf V$, we may assume that it is affine.  
Write $M$ as the direct limit of finitely generated $\mc O_{\mf U}(\mf V)$-modules $M_\alpha$.  Then $\mc F_M$ is the direct limit of the $\mc O_{\mf V}$-modules $\mc F_{M_\alpha} = M_\alpha^\Delta$, 
and by \cite[Proposition~10.10.2(i)]{EGA1} we have a canonical isomorphism $\mc F_{M_\alpha}(\mf V) = M_\alpha^\Delta(\mf V) \iso M_\alpha$, functorial in $M_\alpha$.  Taking direct limits yields an isomorphism $\mc F_M(\mf V) \iso M$.  Thus $M = \mc F_M(\mf V) = \ms M(\mf V)$ is a finite $\mc O_{\mf U}(\mf V)$-module, by the finiteness assumption on $\ms M$.   So $\ms M|_{\mf V} = \mc F_M = M^\Delta$ is coherent, by \cite[Th\'eor\`eme~10.10.5]{EGA1}.  Since coherence is a local property, $\ms M$ is coherent, as asserted.
\end{proof}

But in the formal scheme situation, the identity $\mc O(V) \otimes_{\mc O(U)} \mc O(V) = \mc O(V)$ 
no longer holds, and pushforward under open inclusions of formal schemes no longer necessarily preserves quasi-coherence, even in the case of the structure sheaf, as the next proposition shows.

\begin{prop} \label{quasicoh}
Let $R$ be a finitely generated algebra over a field $k$, let $r \in R$ be a non-zero non-unit, and let $S = R[1/r]$.  Let $A = R[[t]]$ and $B = S[[t]]$, and write $\mf U = \Spf(A)$ and $\mf V = \Spf(B)$, with inclusion $\wh f:\mf V \to \mf U$ induced by $f:\Spec(S) \to \Spec(R)$.  Then $\wh f_*\mc O_{\mf V}$ is not quasi-coherent on $\mf U$.
\end{prop}

\begin{proof}
Suppose that $\wh f_*\mc O_{\mf V}$ is quasi-coherent on $\mf U$.  Let $P$ be a point of $\mf U$ corresponding to a maximal ideal of $A$ containing $r$.  Thus as noted above, there is an affine neighborhood 
$\Spec(R')$ of $P$ in $U = \Spec(R)$ such that the restriction of
$\wh f_*\mc O_{\mf V}$ to $\Spf(A')$ is of the form $\mc F_M$ for some $A'$-module $M$, where $A' = R'[[t]]$.  After replacing $R,A$ by $R',A'$, and similarly replacing $S,B$, we may assume that $\wh f_*\mc O_{\mf V} = \mc F_M$ for some $A$-module $M$.  Since $\wh f_*\mc O_{\mf V}(\mf U) = \mc O(\mf V) = B$ by definition of pushforward, 
and since $\mc F_M(\mf U) = M$, we have that 
$M = B$, viewed as a module over $A = \mc O(\mf U)$. 
By \cite[Lemma~01BJ]{stacks}, 
$\wh f^*\mc F_B = \mc F_{B \otimes_A B}$,  
and so $\wh f^*\wh f_*\mc O_{\mf V}(\mf V) = \wh f^*\mc F_B(\mf V) = \mc F_{B \otimes_A B}(\mf V) = B \otimes_A B$.  
Meanwhile, $\wh f^*\wh f_*\mc O_{\mf V} = \mc O_{\mf V}$ since $V$ is open in $U$, and so $\wh f^*\wh f_*\mc O(\mf V) = \mc O(\mf V) = B$.  
But $\mc O(\mf V) \otimes_{\mc O(\mf U)} \mc O(\mf V) = B \otimes_A B$ is unequal to $\mc O(\mf V) = B$;
in fact, it is much larger (even though $S \otimes_R S = S$).  This is a contradiction, showing that $\wh f_*\mc O_{\mf V}$ is not quasi-coherent on $\mf U$.
\end{proof}

In light of the above discussion, the failure of pushforward to preserve quasi-coherence for formal schemes leads to the question of when the pushforward of a torsion-free coherent sheaf under $\wh f:\mf V \to \mf U$ is coherent, rather than merely finite, in the situation in which the complement has codimension two.
As shown in Examples~\ref{non-factor ex}-\ref{Srinivas ex} below, coherence does not always hold in this situation, and hence neither does quasi-coherence, in light of Proposition~\ref{coh qcoh fin} and Theorem~\ref{genl fin ext}.  

\begin{ex} \label{non-factor ex}
Let $k$ be a field of characteristic unequal to $2$, and let $\mf U = \Spf(k[x,y][[t]])$, with closed fiber $U = \mbb A^2_k$.  Let  $U_1$, $U_2$, $U_0$, and $V \subset U$ be the complements in $U$ of $x=0$, $y=0$, $xy=0$, and $x=y=0$, respectively; and let $\mf U_1$, $\mf U_2$, $\mf U_0$, and $\mf V$ be the associated formal subschemes of $\mf U$.  Write $f:V \to U$ and $\wh f:\mf V \to \mf U$ for the natural inclusions. 
Let $\ms L_1$ and $\ms L_2$ be free modules of rank one on $\mf U_1 = \Spf(k[x,x^{-1},y][[t]])$ and $\mf U_2 = \Spf(k[x,y,y^{-1}][[t]])$ with generators $m_1$ and $m_2$, respectively.  Fix some positive integer $n$, and consider the isomorphism of $\ms L_1$ and $\ms L_2$ over $\mf U_0 = \Spf(k[x,x^{-1},y,y^{-1}][[t]])$ given by $m_2=zm_1$,
where $z = (1+t^n/xy)^{1/2} \in \wh R_{U_0}^\times$.
These define a locally free coherent sheaf $\ms F_n$ of rank one on $\mf V$.  Then $\wh f_*\ms F_n$ is finite and torsion-free by Theorem~\ref{genl fin ext}; but we claim 
$\wh f_*\ms F_n$ is not coherent on $\mf U$ and moreover that $\wh f_*\ms F_n(U)=0$.  Note that this last equality implies that $\wh f_*\ms F_n$ is not coherent, since otherwise $\ms L_e(U_e) = \ms F_n(U_e) = \wh f_* \ms F_n(U_e) =  \wh f_* \ms F_n(U) \otimes_{\wh R_U} \wh R_{U_e} = 0$ for $e=1,2$ by Proposition~\ref{pushforward torsion-free}(b).  Thus it suffices to show that $\wh f_*\ms F_n(U)=0$.

So suppose that $s \in \wh f_*\ms F_n(U) = \ms F_n(V)$ is nonzero.  Then $s|_{U_1} = fm_1$ and $s|_{U_2} = gm_2$ for some nonzero $f \in \wh R_{U_1}$ and $g \in \wh R_{U_2}$.  On $U_0$, we have $fm_1 = s|_{U_0} = gm_2 = gzm_1$.  So $f=gz$, and $xyf^2=(xy+t^n)g^2 \in \wh R_{U_1} \cap \wh R_{U_2} = \wh R_U$ by Proposition~\ref{intersection}.
Now $xy+t^n$ is irreducible in $\wh R_{U_1}=k[x,x^{-1},y][[t]]$ and in $\wh R_{U_2}=k[x,y,y^{-1}][[t]]$ since as a power series in $t$ its constant term is a unit.  It is also irreducible in $\wh R_U = k[x,y][[t]]$, since otherwise there would be a nontrivial factorization  $xy+t^n = (\sum_0^\infty a_it^i)(\sum_0^\infty b_it^i)$ with $a_i,b_i \in k[x,y]$ and $a_0=x$, $b_0=y$.  But then, by equating the coefficients of $t^i$ on both sides, for $i<n$, induction on $i$ would yield that $x|a_i$ and $y|b_i$ for all $i<n$; and equating the coefficients of $t^n$ would give $1 \in (x,y) \subset k[x,y]$, a contradiction.  So indeed, $xy+t^n$ is irreducible in each of $\wh R_U, \wh R_{U_1}, \wh R_{U_2}$, say with associated discrete valuations $v,v_1,v_2$.  Here $v_e|_{\wh R_U} = v$ for $e=1,2$.  So $v(xyf^2) = v_1(xyf^2)$ is even, whereas $v((xy+t^n)g^2) = v_2((xy+t^n)g^2)$ is odd.  This contradiction proves the claim. 
\end{ex} 

The above example shows that there exists a sequence of torsion-free (and even locally free) coherent sheaves $\ms F_n$ on $\mf V$ such that $\wh f_*\ms F_n$ is not coherent on $\mf U$ (and moreover $\wh f_*\ms F_n(U)=0$) but such that $\ms F_n \to \mc O_{\mf V}$, in the sense that $\ms F_n$ is congruent to the structure sheaf $\mc O_{\mf V}$ modulo $t^n$.  Thus the set of torsion-free coherent sheaves $\ms F$ on $\mf V$ with $\wh f_*\ms F$ coherent is not open in the $t$-adic topology induced by congruence of sheaves modulo powers of $t$.

Note that if in the above example we replace $z = (1+t^n/xy)^{1/2}$ with $w = 1+t^n/xy$ then the pushforward would have nonzero sections; e.g., the section given by $(xy+t^n)m_1$ on $U_1$ and by $xym_2$ on $U_2$, since $xy+t^n = xyw$.  But the pushforward would still fail to be coherent.  Namely, if it were coherent, then by Proposition~\ref{Weil Cartier} it would be locally free of rank one, and hence free of rank one since $k[x,y][[t]]$ is a UFD 
by \cite[Theorem~20.8]{Mats89} (because $k[x,y]$ is a regular UFD).  One would then have a factorization $f=gw$ for some {\em units} $f \in \wh R_{U_1}^\times$ and $g \in \wh R_{U_2}^\times$; but this is impossible, as can be seen by examining the lowest degree terms of $f$ and $g$ that have denominators (powers of $x$ and $y$, respectively).  
Note also that this example generalizes to examples with rank one sheaves over higher dimensional formal schemes such as $\Spf(k[x,y,s_1,\dots,s_d][[t]])$.  

\begin{ex} \label{cone ex}
Define $\ms X \subset \mbb A^3_{k[[t]]}$ by $z^2=xy-t$, with closed fiber $X \subset \mbb A^3_k$ being the affine cone $z^2=xy$, and take the formal completion $\mf X$ along $X$.  Note that $\ms X$ is regular.  Consider the affine open subsets $U_1 \subset X$ where $x\ne 0$ and $U_2 \subset X$ where $y\ne 0$, and let $\mf U_1, \mf U_2$ be the formal completions of $\ms X$ along $U_1, U_2$. Let $V = U_1 \cup U_2$, having formal completion $\mf V$, with inclusion $f: V \hookrightarrow X$.  Let $\ms D$ be the (irreducible) divisor on $\ms X$ given by $z=0$, and let~$\mf D$ be its formal completion.  The reduction of $\mf D$ (or $\ms D$) modulo $t$ consists of two irreducible components; viz., the $x$-axis and $y$-axis in $\mbb A^3_k$.  
Thus $\mf D$ restricts to the trivial divisor on the formal completion $\mf U_0$ along $U_0 = U_1 \cap U_2$.  Let $\mf D_e$ be the restriction of $\mf D$ to $\mf U_e$, for $e=0,1,2$.  Let $\ms F$ be the invertible subsheaf of $\mc O_{\mf V}$ on $\mf V$ whose restriction to $\mf U_1$ is $\mc O_{\mf U_1}$ and whose restriction to $\mf U_2$ is $\mc O_{\mf U_2}(-\mf D_2)$.  (These two restrictions agree over $\mf U_0$ as subsheaves of the structure sheaf, since $\mf D_0$ is the trivial divisor.)  

To see that $\wh f_*\ms F$ is not coherent, first note that 
\[\wh f_*\ms F(\mf X) = \ms F(\mf U_1) \cap \ms F(\mf U_2) = \mc O_{\mf X}(\mf U_1) \cap z\mc O_{\mf X}(\mf U_2) \subseteq \mc O_{\mf X}(\mf U_1) \cap \mc O_{\mf X}(\mf U_2) = \mc O_{\mf X}(\mf X),\]
where the intersection takes place in $\mc O_{\mf X}(\mf U_0)$ and where the last equality is by Proposition~\ref{intersection}.  Since the contraction of the ideal $(z) \subset \mc O_{\mf X}(\mf U_2)$ to $\mc O_{\mf X}(\mf X)$ is the ideal $(z)$ in that ring, we have that 
$M:=\wh f_*\ms F(\mf X) =\mc O_{\mf X}(\mf U_1) \cap z\mc O_{\mf X}(\mf U_2) = z\mc O_{\mf X}(\mf X) = \mc O_{\mf X}(-\mf D)(\mf X)$.
So $\mc O_{\mf X}(-\mf D) =  M^\Delta$, the
unique formal coherent sheaf on $\mf X$ whose module of global sections is $M$. 
But $M^\Delta(\mf U_1) = \mc O_{\mf X}(-\mf D)(U_1)$, which is strictly smaller than $\ms F(\mf U_1)$ (viewing these as submodules of $\mc O_{\mf X}(\mf U_1)$).  So $\wh f_*\ms F$ is not the same subsheaf of $\mc O_{\mf X}$ as $M^\Delta$, and thus is it not coherent.
\end{ex}

The following example was shown to us by V.~Srinivas (prior to our finding the examples above).

\begin{ex} \label{Srinivas ex}
Let $k$ be an algebraically closed field, and $\ms X \to \Spec(k[[t]])$ a projective flat morphism of relative dimension two, with $X$ its special fiber. Assume $\ms X$ is regular, and that $X$ is
integral and normal. Let $U$ be the regular locus of $X$, with inclusion $f:U \hookrightarrow X$; this is a regular connected $k$-variety, smooth and quasi-projective over $k$.  Suppose further that $\Pic(U)$ has a non-zero torsion element of order
relatively prime to $\cha(k)$, not contained in the image
of the restriction map
$\Pic(X) \to \Pic(U)$, and let $\ms L$ be the corresponding invertible sheaf. (As a concrete example, we can take $\ms X \subset \mbb P^3_{k[[t]]}$ to be the cubic hypersurface $x^3+y^3+z^3+tw^3=0$, where $x,y,z,w$ are the homogeneous coordinates, and where $\cha(k) \ne 3$. Here $\Pic(X) \cong \mbb Z$, generated by the class of a hyperplane section; whereas $\Pic(U)$ is isomorphic to the Picard group of the Fermat curve, and so has many torsion elements.)

Let $\mf X$ and $\mf U$ be the formal completions of $\ms X$ along $X$ and $U$, respectively, and let $U_n$ be the reduction of $\mf U$ modulo $t^n$.  Thus the map  $\displaystyle \Pic(\mf U) \to \lim_{\leftarrow} \Pic(U_n)$ is an isomorphism, by \cite[Chapter~II, Exercise~9.6]{Hts}.  Moreover the maps $\Pic(U_n) \to \Pic(U_{n-1})$ restrict to isomorphisms on torsion subgroups of exponent prime to $\cha(k)$.  So we obtain a formal invertible sheaf $\wh {\ms L}$ on $\mf U$ whose restriction to $U$ is $\ms L$, and $\wh {\ms L}$ is a torsion element of $\Pic(\mf U)$ of the same order as $\ms L$ in $\Pic(U)$.  

If $\wh f_*\wh {\ms L}$ is coherent on $\mf X$, then by Proposition~\ref{Weil Cartier}, it is an invertible sheaf on $\mf X$ that restricts to $\wh {\ms L}$ on $\wh{\mf U}$.  Thus its reduction modulo $t$ is an invertible sheaf on $X$ that restricts to $\ms L$ on $U$.  That is, the class of $\ms L$ is in the image of $\Pic(X) \to \Pic(U)$.  This contradicts the choice of $\ms L$, and shows that in fact $\wh f_*\wh {\ms L}$ is not coherent on $\mf U$.  (Alternatively, as in Srinivas's original argument, one could invoke Grothendieck's Existence Theorem (\cite[Corollaire~5.1.6]{EGA3}) to obtain a coherent sheaf $\ms M$ on $\ms X$ that induces $\wh j_*\wh{\ms L}$; then show that $\ms M$ is necessarily reflexive of rank one and hence locally free by the regularity of $\ms X$; and then again obtain the contradiction that $\ms L$ is in the image of $\Pic(X) \to \Pic(U)$.)
\end{ex}

\begin{rem} \label{different RUs}
A variant of Example~\ref{Srinivas ex} shows that in Proposition~\ref{RU versions}, the injective map $\wh{\mc O}_{\ms X,U} \to \mc O_{\mf X}(U)$ need not be an isomorphism without the additional assumption in Proposition~\ref{RU versions}(b).
Namely, as in Example~\ref{Srinivas ex} take $\ms X \subset \mbb P^3_{k[[t]]}$ given by $x^3+y^3+z^3+tw^3=0$, and write $X$ for the closed fiber $(t=0)$.  Here~$X$ is the projective cone over the Fermat cubic curve $E$ given by $x^3+y^3+z^3=0$, which we can identify with the intersection of $X$ with the plane $w=0$ in $\mbb P^3_k$.  Let $P_0 \in E$ be the flex point $(-1:1:0)$; here $3P_0$ is cut out by the line $x+y=0$ in $\mbb P_k^2$.  With respect to the group law on $E$ in which $P_0$ is the identity element, pick a non-torsion point $P \in E$.  Let $L_0, L$ be the lines on $X$ that pass through the vertex and correspond to the points $P_0,P \in E$.  Thus $3L_0$ is the class of a hyperplane section on~$X$.
Since $P$ is a non-torsion point on $E$, no positive multiple of it is linearly equivalent on $E$ to a multiple of $P_0$; and so no positive multiple of $L$ is linearly equivalent to a multiple of $L_0$, or to a multiple of a hyperplane section.  Since $\Pic(X)$ is generated by the class of a hyperplane section, this implies that no positive multiple of the Weil divisor $L$ is a Cartier divisor on~$X$.  Now let $U_0 \subset X$ be the complement of $L$ in $X$; this is an affine open subset.  
If $f \in \mc O_{\ms X,U_0}$, then $f$ is a rational function on $\ms X$ whose pole divisor $D$ meets $X$ only along the irreducible Weil divisor $L$, and so $D|_X$ is a multiple of $L$.  Since $\ms X$ is regular, $D$ is a Cartier divisor; 
hence so is $D|_X$.
But no positive multiple of $L$ is a Cartier divisor, so $D$ is trivial.  Hence $f$ is a regular function on $\ms X$; i.e., $f \in T$.  This shows that 
$\mc O_{\ms X,U_0} = T$.  Hence the completion $\wh{\mc O}_{\ms X,U_0}$ is also $T$, whereas $\mc O_{\mf X}(U_0)$ is much larger.  Note also that there is no affine open subset $\ms U = \Spec(R) \subseteq \ms X$ such that $\ms U \cap X = U_0$, since otherwise $R \subseteq \wh{\mc O}_{\ms X,U_0} = T$, a contradiction.
\end{rem}

The next assertion provides conditions under which $\wh f_*\ms F$ is coherent, in terms of the reductions modulo powers of the uniformizer of $T$.  

\begin{thm} \label{coherent conditions}
Let $T$ be a complete discrete valuation ring with uniformizer $t$ and residue field $k$,
and let $\ms X$ be a quasi-projective normal integral $T$-scheme with  
reduced closed fiber~$X$ and formal completion $\mf X$.  Let 
$f:V \hookrightarrow U$ be an inclusion of connected open subsets of $X$
such that the complement of $V$ in $U$ has codimension at least two,
and write $\mf V, \mf U$ for the formal open subschemes of $\mf X$ associated to $V,U$.  
Let $f_n:V_n \to U_n$ be the reduction of $f$ modulo $t^n$, and let $\phi_n:U_n \to \mf U$ and $\omega_n:U_n \to U_{n+1}$ be the natural inclusion morphisms.
Let $\ms F$ be a coherent sheaf on $\mf V$, and write $F_n = \phi_n^* \ms F$.
\begin{enumerate} [(a)]
\item \label{surj coh}
Suppose that for every sufficiently large $n$, $(f_n)_*F_n$ is a coherent sheaf on $U_n$ and that $\omega_n^*(f_{n+1})_*F_{n+1} = (f_n)_*F_n$.   Then $\wh  f_* \ms F$ is a coherent sheaf on~$\mf U$.
\item \label{good reflexive}
Suppose that $\ms F$ is reflexive. 
Then $F_n$ is coherent and torsion-free on $V_n$, and 
$(f_n)_*F_n$ is a torsion-free coherent sheaf on $U_n$.  So $\wh f_* \ms F$ is coherent on $\mf U$ if $\omega_n^*(f_{n+1})_*F_{n+1} = (f_n)_*F_n$ for every sufficiently large $n$.  In this case $\wh f_* \ms F$ is also reflexive.
\item \label{loc free coh}
If $\ms F$ is locally free, then $\wh f_* \ms F$ is coherent and locally free on $\mf U$ if and only if $(f_n)_*F_n$ is a locally free module on $U_n$ for all sufficiently large $n$ (or equivalently, for all $n$).
\item \label{loc free rk 1}
If $\ms F$ is locally free of rank one, and $\ms X$ is regular, then $\wh f_* \ms F$ is coherent on $\mf U$ if and only if $(f_n)_*F_n$ is a locally free module on $U_n$ for all sufficiently large $n$ (or equivalently, for all $n$).  In this case, $\wh f_* \ms F$ is locally free on $\mf U$.
\end{enumerate}
\end{thm}

\begin{proof}
For part~(\ref{surj coh}), choose a positive integer $n_0$ such that $(f_n)_*F_n$ is a coherent sheaf and
$\omega_n^*(f_{n+1})_*F_{n+1} = (f_n)_*F_n$ for all $n \ge n_0$.  
Since $\ms F$ is coherent on $\mf V$, it is the inverse limit of the sheaves $F_n$, for $n \ge n_0$; and these sheaves are coherent on the reductions $V_n$ (see \cite[Th\'eor\`eme~10.11.3]{EGA1}).  So for every open subset $O \subseteq U$ we have $\displaystyle \wh f_*\ms F(O) = \ms F(O \cap V) = \lim_\leftarrow F_n(O \cap V) = \lim_\leftarrow (f_n)_* F_n(O)$; i.e., $\displaystyle \wh f_*\ms F = \lim_\leftarrow (f_n)_* F_n$, where the inverse limit ranges over $n \ge n_0$. Since  $\omega_n^*(f_{n+1})_*F_{n+1} = (f_n)_*F_n$ for $n \ge n_0$, $\displaystyle \wh f_*\ms F = \lim_\leftarrow (f_n)_* F_n$ is  a formal coherent sheaf (again by \cite[Th\'eor\`eme~10.11.3]{EGA1}).

Concerning part~(\ref{good reflexive}), this assertion is trivial if $\ms X$ lies over the closed point of $\Spec(T)$;  so we may assume that $\ms X$ dominates $\Spec(T)$.  Since $\ms X$ is integral, the uniformizer $t$ of $T$ is then not a zero divisor in the local rings of $\ms X$ at its closed points.  Since $\ms X$ is normal, the local rings $\mc O_{\ms X,P}$ satisfy Serre's condition $S_2$ (e.g., by \cite[17.I, Theorem~39]{Mats}).   That is, ${\rm depth}(\mc O_{\ms X,P}) \ge \min(2,\dim(\mc O_{\ms X,P}))$ for all $P \in X_n$, where $X_n$ is the reduction of $X$ modulo $t^n$.   Since $t$ is not a zero-divisor in the local rings, it follows that 
${\rm depth}(\mc O_{X_n,P}) = {\rm depth}(\mc O_{\ms X,P}) - 1 \ge \min(1,\dim(\mc O_{X_n,P})$.  So the local rings $\mc O_{X_n,P}$ satisfy Serre's condition~$S_1$; and in particular the local rings $\mc O_{U_n,P}$ satisfy $S_1$.  Thus $U_n$, or equivalently $\mc O_{U_n}$, has no embedded associated points, by \cite[Lemma~031Q]{stacks}. 

Meanwhile, since $\ms F$ is reflexive and since $\ms X$ satisfies condition $S_2$, so does $\ms F$, by \cite[Lemma~0EBA]{stacks}.  
So arguing as above, but for $\ms F$ rather than for the structure sheaf, we obtain that $F_n$ also has no embedded associated points.

As a consequence, the associated points of $U_n$ and $F_n$ are the same; viz., the generic points of the irreducible components of $U_n$.  So by \cite[Partie~4, Proposition~20.1.6]{EGA4}, $F_n$ is torsion-free.  It is also coherent, since $\ms F$ is.  So by Theorem~\ref{scheme pushforward} 
it follows that $(f_n)_*F_n$ is a torsion-free coherent sheaf on $U_n$.  
This proves the first assertion in part~(\ref{good reflexive}).
The second assertion in part~(\ref{good reflexive}) then follows from part~(\ref{surj coh}), and the last assertion follows from Proposition~\ref{pushforward reflex}(\ref{pushforward coh reflex}).

For part~(\ref{loc free coh}), we begin by proving the following claim: If $\ms M, \ms N$ are finite locally free coherent sheaves on $U_n$ for some $n$, with $\ms M$ a submodule of $\ms N$, and if the inclusion induces an isomorphism $f_n^*\ms M \to f_n^*\ms N$, then $\ms M = \ms N$.  To prove this, consider an affine open subset $O \subset U_n$ on which these modules are free, and choose bases for $\ms M(O)$ and $\ms N(O)$.  These are of the same size since $\ms M$ and $\ms N$ agree on $V_n$.
Let $A$ be the matrix whose columns give the coordinates of the basis elements of $\ms M(O)$ in terms of those of $\ms N(O)$.  Since $f_n^*\ms M \to f_n^*\ms N$ is an isomorphism, the determinant $\delta$ of this matrix is a unit in $\mc O_{U_n}(V_n \cap O)$.  Since the complement of $V_n \cap O$ has codimension at least two in $U_n \cap O$, and since a non-unit vanishes on a closed subset of codimension one, it follows that $\delta$ is a unit in $\mc O_{U_n}(U_n \cap O)$.  That is, $\ms M(O) = \ms N(O)$.  Since this is true for all sufficiently small affine open subsets $O \subseteq U_n$, the claim follows.

We now use this claim to prove part~(\ref{loc free coh}).  First suppose that $(f_n)_*F_n$ is a locally free coherent sheaf on $U_n$ for all sufficiently large $n$, say $n \ge n_0$.  For each such $n$, the pullback $\omega_n^*\bigl((f_{n+1})_*F_{n+1})\bigr)$ is a locally free sheaf on $U_n$, since 
$(f_{n+1})_*F_{n+1}$ is locally free.
Now $\omega_n f_n=f_{n+1}\nu_n$, where $\nu_n$ is the restriction of $\omega_n$ to $V_n$.  So
$f_n^*\omega_n^*(f_{n+1})_*F_{n+1} = \nu_n^* f_{n+1}^*(f_{n+1})_*F_{n+1}
= \nu_n^*F_{n+1}=F_n = f_n^*(f_n)_*F_n$.  Since $\omega_n^*(f_{n+1})_*F_{n+1}$ is a subsheaf of $(f_n)_*F_n$ and both are locally free, the claim implies that these sheaves are equal.  Hence $\wh f_*\ms F$ is a coherent sheaf on $\mf U$ by part~(\ref{surj coh}).  Moreover $\wh f_*\ms F$ is locally free (or equivalently, flat), because it is so on stalks, by \cite[Lemma~0523]{stacks}.

For the other direction, suppose that $\wh f_*\ms F$ is a locally free coherent sheaf on $\mf U$. 
The inclusion $\phi:U_n \to \mf U$ induces morphisms $\phi^*:\ms F \to F_n$ and $\phi^*:\wh f_*\ms F \to (f_n)_*F_n$.
For every~$n$, the pullback $\phi_n^*\wh f_*\ms F$ 
of $\wh f_*\ms F$
under $\phi:U_n \to \mf U$ is a locally free coherent sheaf on $U_n$ whose pullback under $f_n:V_n \to U_n$ is $f_n^* \phi_n^*\wh f_*\ms F = \phi_n^*\wh f^*\wh f_*\ms F = \phi_n^*\ms F = F_n$.  Thus $\phi_n^*\wh f_*\ms F$ is a locally free sheaf on $U_n$ that has the same restriction $F_n$ to $V_n$ as $(f_n)_*F_n$.  But 
$\phi_n^*\wh f_*\ms F$ is a subsheaf of $(f_n)_*F_n$.
Hence $\phi_n^*\wh f_*\ms F = (f_n)_*F_n$ by the above claim.  Thus $(f_n)_*F_n$ is locally free for all $n$ (and in particular, for all sufficiently large $n$).  This completes the proof of part~(\ref{loc free coh}).

Part~(\ref{loc free rk 1}) now follows from part~(\ref{loc free coh}) together with Proposition~\ref{Weil Cartier}.
\end{proof}

\begin{rem}
The above theorem provides another way of verifying that $\wh f_*\ms F$ is not coherent in Example~\ref{cone ex}.  Namely, consider the reduction $F_n$ of $\ms F$ modulo $t^n$, and its pushforward $(f_n)_*F_n$ to the reduction of $\ms X$ modulo $t^n$.  The global sections of this pushforward form a module that is generated by the elements $\{z,x^n,x^{n-1}t,\dots,xt^{n-1}\}$; and this module is not locally free for $n>1$.  So by Theorem~\ref{coherent conditions}(\ref{loc free rk 1}), $\wh f_*\ms F$ is not coherent.
\end{rem}

The hypotheses in Theorem~\ref{coherent conditions} each involve infinitely many conditions, on the reductions of the sheaf modulo $t^n$ for all sufficiently large $n$.  In the affine case of a reflexive sheaf, we can state hypotheses that involve just the formal sheaf itself.  This is done in the following two corollaries.

\begin{cor} \label{H1 tor free cor}
Under the hypotheses of Theorem~\ref{coherent conditions}, 
assume that $U$ is affine and that $\ms F$ is reflexive.  If $H^1(\mf V,\ms F)$ is $t$-torsion-free, then $\wh f_*\ms F$ is a reflexive coherent sheaf on $\mf U$.
\end{cor}

\begin{proof}
We preserve the notation of Theorem~\ref{coherent conditions}, with $\omega_n: U_n \to U_{n+1}$ and $F_n=\omega_n^*\ms F$ as before.  By Theorem~\ref{coherent conditions}(\ref{good reflexive}), to prove the corollary it suffices to show that $\omega_n^*(f_{n+1})_*F_{n+1} = (f_n)_*F_n$ for all $n \ge 1$.  This is immediate if $\ms X$ lies over the closed point of $\Spec(T)$.  So we may assume that the integral scheme $\ms X$ dominates $\Spec(T)$, and hence $t$ is a regular element of $\wh R_U$.  By Theorem~\ref{coherent conditions}(\ref{good reflexive}), $(f_n)_*F_n$ is coherent on the affine scheme $U_n$; hence $(f_n)_*F_n = \til M_n$, where $M_n = (f_n)_*F_n(U_n) = F_n(V_n)$, for all $n$.   

Since $\ms F$ is reflexive, it is torsion-free, and in particular $t^n$-torsion-free for all $n \ge 1$, since~$t$ is regular in $\wh R_U$.  
Hence we have short exact sequences
$0 \to \ms F \buildrel{\cdot t^n} \over {\rightarrow} \ms F \to \ms F/t^n\ms F \to 0$ of sheaves on $\mf V$.  These yield
long exact sequences that begin
\[0 \to \ms F(V) \buildrel{\cdot t^n} \over {\rightarrow} \ms F(V) \to (\ms F/t^n\ms F)(V) \to  H^1(\mf V,\ms F)  \buildrel{\cdot t^n} \over {\rightarrow} H^1(\mf V,\ms F).\]  
By hypothesis, multiplication by $t$ on $H^1(\mf V,\ms F)$ is injective, and hence so is multiplication by $t^n$. 
Thus for every $n \ge 1$ the above long exact sequence yields an isomorphism 
\[\ms F(V)/t^n\ms F(V) \iso (\ms F/t^n\ms F)(V) = \phi_n^* \ms F(V_n) = F_n(V_n) = M_n.\]
We then have  
$\omega_n^*(f_{n+1})_*F_{n+1} = \omega_n^* \til M_{n+1} = \til M_{n+1,n}$, where $M_{n+1,n} = 
M_{n+1}/t^nM_{n+1}$.  Here $M_{n+1,n} 
= (\ms F(V)/t^{n+1}\ms F(V))/(t^n\ms F(V)/t^{n+1}\ms F(V)) =  \ms F(V)/t^n\ms F(V) = M_n$.
Thus we have 
$\omega_n^*(f_{n+1})_*F_{n+1} = \til M_n = (f_n)_*F_n$, as needed.
\end{proof}

\begin{cor} \label{Cech H1 cor}
Under the hypotheses of Theorem~\ref{coherent conditions}, 
assume that $U$ is affine and that $V = U_1 \cup U_2$, where $U_1$ and $U_2$ are affine open subsets of $X$.  Let $U_0 = U_1 \cap U_2$, assume that $\ms F$ is reflexive, and define $\psi: \ms F(U_1) \times \ms F(U_2) \to \ms F(U_0)$ by $\psi(f,g) = f-g$.  If 
\[t\ms F(U_0) \cap \im(\psi) = t \im(\psi),\]
then $\wh f_*(\ms F)$ is a reflexive coherent sheaf.
\end{cor}

\begin{proof}
First note that since $\ms F$ is reflexive, it is torsion-free.  Hence the natural maps $\ms F(U_e) \to \ms F(U_0)$ are injective, for $e=1,2$, by Lemma~\ref{module inj} and Proposition~\ref{pushforward torsion-free}(b).  Thus we may regard $\ms F(U_e) \subseteq \ms F(U_0)$ for $e=1,2$; and the map $\psi$ is well defined.  

By Corollary~\ref{H1 tor free cor}, it suffices to show that $H^1(\mf V,\ms F)$ is $t$-torsion-free.  Since each $U_e$ is affine, and since $\ms F$ is coherent, it follows that $\ms F$ is acyclic on the formal completion $\mf U_e$ of~$\ms X$ along~$U_e$ (e.g., by \cite[Proposition~5.2.3]{EGA3}, taking $\mf X = \mf Y = \mf U_e$ there).  Hence $H^1(\mf V,\ms F)$ is equal to the \v{C}ech cohomology group $\check{H}^1(\ms U,\ms F)$, where $\ms U$ is the affine open covering $\{U_1,U_2,U_0\}$ of~$V$. So it suffices to show that $\check{H}^1(\ms U,\ms F)$ is $t$-torsion-free.

Let $\alpha \in \check{H}^1(\ms U,\ms F)$ and suppose that $t\alpha = 0 \in \check{H}^1(\ms U,\ms F)$.  Thus $\alpha$ is the cohomology class of some element $f \in \ms F(U_0)$.  Since $t\alpha = 0$, we have that $tf$ is a coboundary; i.e., there exist $g_e \in \ms F(U_e)$ for $e=1,2$ such that $tf=g_1-g_2$.  
Thus $tf \in t\ms F(U_0) \cap \im(\psi)  = t \im(\psi)$.  That is, there exists $f' \in \im(\psi)$ such that $tf = tf' \in \ms F(U_0)$.  Since $\ms F$ is torsion-free, $f=f' \in \im(\psi)$.  Thus $f = g_1' - g_2'$ where $g_e' \in \ms F(U_e)$ for $e=1,2$; i.e., $f$ is a coboundary.  Thus $\alpha = 0$, as needed.
\end{proof}

Although the conditions in parts~(\ref{loc free coh}) and~(\ref{loc free rk 1}) 
of Theorem~\ref{coherent conditions}, 
are both necessary and sufficient, the conditions in parts~(\ref{surj coh}) and~(\ref{good reflexive}) of the theorem are sufficient to guarantee coherence of $\wh f_*\ms F$ but are not necessary (and hence the same holds for the above two corollaries, since they prove coherence of $\wh f_*\ms F$ via Theorem~\ref{coherent conditions}(\ref{good reflexive})).  
Namely, as the following example shows, in the situation of Theorem~\ref{coherent conditions} there exist coherent sheaves $\ms F$ on $\mf V$ such that $\wh f_*\ms F$ is coherent although the morphisms $\omega_n^*(f_{n+1})_*F_{n+1} \to (f_n)_*F_n$ are not surjective; and $\ms F$ can even be chosen so as to be locally free (and in particular reflexive).

\begin{ex} \label{reflexive converse counterex}
Let $\mf V \subseteq \mf U$ be as in Example~\ref{non-factor ex},
and preserve the notation $f_n: V_n \to U_n$ and $\omega_n$ from Theorem~\ref{coherent conditions}.  Let $M$ be the finite $k[x,y][[t]]$-module having generators $m_1,m_2,m_3$ subject to the relation $xm_1+ym_2+tm_3=0$.  
Thus $\ms G := M^\Delta$ is coherent.  Moreover the restrictions of $\ms G$ to the open sets where $x \ne 0$ and where $y \ne 0$ are each free, since $M$ induces free modules over $k[x,y,x^{-1}][[t]]$ and over $k[x,y,y^{-1}][[t]]$.  Hence the pullback $\ms F = \wh f^* \ms G$ is coherent 
and locally free on $\mf V$, and thus reflexive and torsion-free.
By Theorem~\ref{reflexive pull push},
$\wh f_*\ms F = \wh f_*\wh f^* \ms G$ is coherent and reflexive on $\mf U$.  
Meanwhile, by Theorem~\ref{coherent conditions}(\ref{good reflexive}), 
$F_n$ and $(f_n)_*F_n$ are coherent and torsion-free on $V_n$ and $U_n$, respectively.
But $\omega_n^*(f_{n+1})_*F_{n+1} \to (f_n)_*F_n$ is not surjective for any $n$.  Namely, for every $n$ let $U_{1,n}, U_{2,n}, U_{0,n}$ be the complements of $x=0, y=0, xy=0$, respectively, in $U_n$.  Since $t^{n-1}xm_1+t^{n-1}ym_2+t^nm_3=0$, and since~$F_n$ is torsion-free, the elements $t^{n-1}x^{-1}m_2 \in F_n(U_{1,n})$ and $-t^{n-1}y^{-1}m_1 \in F_n(U_{2,n})$ induce the same element in $F_n(U_{0,n})$ and hence define an element $s_n \in F_n(V_n) = (f_n)_*F_n(U_n)$.  Here $s_n$ maps to $0$ in $(f_{n-1})_*F_{n-1}(U_{n-1})$.  Also,
since $1/xy$ is not a $k[x,y]$-linear combination of $1/x$ and $1/y$, $(-t^{n-1}y^{-1}m_1) - (t^{n-1}x^{-1}m_2) = x^{-1}y^{-1}t^nm_3 \in t^nF_{n+1}(U_{0,n+1})$ is not the sum of elements of $t^nF_{n+1}(U_{1,n+1})$ and $t^nF_{n+1}(U_{2,n+1})$; and hence
$s_n$ is not the image of an element of $(f_{n+1})_*F_{n+1}(U_{n+1})$. 
Thus $\omega_n^*(f_{n+1})_*F_{n+1}(U_{n+1}) \to (f_n)_*F_n(U_n)$ is not surjective, and so  
$\omega_n^*(f_{n+1})_*F_{n+1} \to (f_n)_*F_n$ is not an isomorphsm of sheaves.
This shows
that the converses to Theorem~\ref{coherent conditions}(\ref{surj coh},\ref{good reflexive}) do not hold.   
\end{ex}

Thus, as Example~\ref{reflexive converse counterex} shows, 
there are more coherent sheaves $\ms F$ on $\mf V$ such that $\wh f_*\ms F$ is coherent on $\mf U$ than are guaranteed by Theorem~\ref{coherent conditions}.  On the other hand, as 
Examples~\ref{non-factor ex}-\ref{Srinivas ex} above show,
there are also torsion-free coherent sheaves $\ms F$ on $\mf V$ such that $\wh f_*\ms F$ is not coherent on~$\mf U$.  In light of this, 
it would be desirable to find general necessary and sufficient conditions for the pushforward $\wh f_*(\ms F)$ to be coherent.

\section{Patching problems} \label{p problems}

Given a ring $R$ and overrings $R_0,R_1,R_2 \supseteq R$ with $R_1,R_2 \subseteq R_0$ and $R = R_1 \cap R_2 \subseteq R_0$, a {\it patching problem} for these rings consists of finitely generated $R_e$-modules $M_e$ for $e=0,1,2$, together with isomorphisms $\alpha_e:M_e \otimes_{\wh R_e} R_0 \to M_0$ for $e=1,2$.  A {\it solution} to the patching problem consists of a finitely generated $R$-module $M$ together with isomorphisms $\gamma_e:M \otimes_R R_e \to M_e$, for $e=1,2$, such that the diagram
    $$
	\xymatrix{
	& (M \otimes_R R_1) \otimes_{R_1} R_0 \ar^-{\gamma_1 \otimes {\rm id}}[r] & M_1 \otimes_{R_1} R_0 \ar_{\alpha_1}^-{\cong}[rd] & \\
	M \otimes_R R_0 \ar^-{\cong}[ru] \ar_-{\cong}[rd] & & &  M_0 \\
	& (M \otimes_R R_2) \otimes_{R_2} R_0 \ar^-{\gamma_2 \otimes {\rm id}}[r] & M_2 \otimes_{R_2} R_0 \ar^{\alpha_2}_-{\cong}[ru] &
	}
	$$
commutes.
That is, the $R$-module $M$ induces the modules $M_e$, for $e=1,2$, compatibly with the maps $\alpha_e$.

For example, let $U = \Spec(R)$ be an affine scheme with affine dense open subsets $U_e=\Spec(R_e)$ for $e=0,1,2$, such that $U_0 = U_1 \cap U_2$ and $U = U_1 \cup U_2$.  Thus $R = R_1 \cap R_2 \subseteq R_0$.  In this situation,
every patching problem for the rings $R,R_0,R_1,R_2$ has a solution, by Zariski patching (gluing) of coherent sheaves and the correspondence between coherent sheaves on an affine scheme $\Spec(A)$ and finitely presented $A$-modules (e.g., see \cite[Lemmas~00AN, 01I9(1), 01IA]{stacks}). 

An analogous statement holds for formal schemes.  Namely, let $\mf X$ be a formal scheme with reduced closed fiber $X$, and let $U,U_0,U_1,U_2$ be affine dense open subsets of $X$ such that $U_0 = U_1 \cap U_2$ and $U = U_1 \cup U_2$.  Let $\mf U, \mf U_e$ be the formal schemes associated to $U, U_e$ (i.e., the restrictions of $\mf X$ to those subsets).  As before, we write $\wh R_U = \mc O_{\mf X}(U)$ and $\wh R_{U_e} = \mc O_{\mf X}(U_e)$ for $e=0,1,2$.  Then 
$\wh R_U = \wh R_{U_1} \cap \wh R_{U_2} \subseteq \wh R_{U_0}$ since $\mc O_{\mf X}$ is a sheaf; and 
every patching problem for the rings $\wh R_U, \wh R_{U_0}, \wh R_{U_1}, \wh R_{U_2}$ has a solution.  To see this, recall that by 
\cite[Proposition~10.10.5]{EGA1}, every coherent sheaf on $\mf U$ is of the form $M^\Delta$ for some finitely generated $\wh R_U$-module, and similarly for each $\mf U_e$.  Since the underlying space $U$ of $\mf U$ is the union of the underlying spaces of $\mf U_1, \mf U_2$, with intersection being the underlying space of $\mf U_0$, a coherent sheaf on $\mf U$ is given by finitely generated modules over $\wh R_{U_1}, \wh R_{U_2}$ together with an agreement over $\wh R_{U_0}$.  Hence every patching problem for the rings $\wh R_U,\wh R_{U_0},\wh R_{U_1},\wh R_{U_2}$ has a solution.  Moreover, in these two situations (schemes and formal schemes), the solution is unique up to isomorphism, because there is an equivalence of categories between patching problems and finitely generated modules over the base ring ($R$ or $\wh R_U$, respectively).

The next result considers patching problems for torsion-free coherent formal sheaves, where $U_1 \cup U_2$ is strictly contained in $U$, with complement having codimension at least two.

Recall the situation of Proposition~\ref{patch 2}:  We have a complete discrete valuation ring $T$ 
with uniformizer $t$, and a normal integral $T$-scheme $\ms X$ of finite type.
We consider  
affine open subsets $U_0,U_1,U_2,U$ of the reduced closed fiber $X$ of $\ms X$, with $U_e$ a dense subset of $U$ for $e=1,2$, and with $U_0 = U_1 \cap U_2$, 
such that the complement of $W := U_1 \cup U_2$ in $U$ has codimension at least two.  
In this situation, we still have 
$\wh R_U = \wh R_{U_1} \cap \wh R_{U_2} \subseteq \wh R_{U_0}$, by Proposition~\ref{intersection}.
Let $M_e$ be a finitely generated torsion-free $\wh R_{U_e}$-module for $e=0,1,2$.  For $e=1,2$, 
we consider the natural map $\iota_e:M_e \to M_e \otimes_{\wh R_{U_e}} \wh R_{U_0}$, 
which is injective by Lemma~\ref{module inj}; and
we let $\alpha_e:M_e \otimes_{\wh R_{U_e}} \wh R_{U_0} \to M_0$ be an isomorphism. Then $\alpha_e\iota_e:M_e \to M_0$ is injective, mapping $M_e$ isomorphically onto its image in $M_0$.  As in Proposition~\ref{patch 2}, the intersection $M := \alpha_1\iota_1(M_1) \cap \alpha_2\iota_2(M_2) \subseteq M_0$ is a finitely generated torsion-free $\wh R_U$-module.

This situation can be reinterpreted in terms of formal sheaves.  As above, let $\fU$, $\fU_e$ be the formal schemes associated to $U$, $U_e$ (for $e=0,1,2$), with inclusions $\wh f_e:\fU_e \to \fU$.  Also let $\mf W$ be the formal scheme associated to $W = U_1 \cup U_2$, with inclusion $\wh f:\mf W \to \fU$ induced by the inclusion $f:W \to U$.  As before write $M_e^\Delta$ for the coherent formal sheaf on $\fU_e$ associated to $M_e$.  The isomorphisms $\alpha_e$ induce isomorphisms $M_e^\Delta \otimes_{\mc O_{\fU_e}}  \mc O_{\fU_0} \to M_0^\Delta$, and so by \cite[Lemma~00AM]{stacks} 
we may glue the sheaves $M_e^\Delta$ to obtain a torsion-free coherent sheaf of modules $\ms N$ on $\mf W$, which we call the sheaf {\it associated to} the given patching problem.

\begin{prop} \label{pp sheaf bij}
Given a patching problem $(M_0,M_1,M_2;\alpha_1,\alpha_2)$ as above, there is a bijection from the set of isomorphism classes of torsion-free solutions to the patching problem to the set of isomorphism classes of torsion-free coherent sheaves on $\mf U$ whose pullback to $\mf W$ is isomorphic to the associated sheaf $\ms N$, given by $S \mapsto S^\Delta$.   Under this bijection, reflexive (resp., locally free) solutions correspond to reflexive (resp., locally free) coherent sheaves.
\end{prop}

\begin{proof}
Suppose that $S$ is a finite $\wh R_U$-module that is a torsion-free solution to the given patching problem, with associated torsion-free coherent sheaf $S^\Delta$ on $\mf U$.  
Thus we have isomorphisms $\gamma_e:S \otimes_{\wh R_U} \wh R_{U_e} \to M_e$
for $e=1,2$ that are compatible with the isomorphisms $\alpha_e$ (in the sense discussed above).  By \cite[Proposition~10.10.8]{EGA1}, for every affine open subset $g:O \hookrightarrow U$ we have a functorial isomorphism $\wh g^*(S^\Delta) \iso (S \otimes_{\wh R_U} \wh R_O)^\Delta$.  Applying this to $O=U_e$ for $e=0,1,2$, we find that $S^\Delta$ compatibly pulls back to the sheaves $M_e^\Delta$ on the open sets $U_e \subseteq U$, and hence $\ms N = \wh f^*(S^\Delta)$.  Thus the  functor $S \mapsto S^\Delta$ defines a map between the two sets.  It is injective since $S^\Delta(U)=S$.  It remains to show that it is surjective.  

So let $\ms S$ be a coherent sheaf on $\mf U$ such that $\ms N = \wh f^*\ms S$.  Thus $\ms S = S^\Delta$, where $S = \ms S(U)$.  For $e=0,1,2$ we have isomorphisms $M_e = \ms N(U_e) \iso (S \otimes_{\wh R_U} \wh R_{U_e})^\Delta(U_e) =  S \otimes_{\wh R_U} \wh R_{U_e}$ that are compatible with the isomorphisms $\alpha_e$, again by \cite[Proposition~10.10.8]{EGA1}.  Hence $S$ is a solution to the given embedding problem.

The last assertion follows from the fact that an $\wh R_U$-module $S$ is reflexive (resp., locally free) if and only if the coherent sheaf $S^\Delta$ is.
\end{proof}

In the above situation, the solution to the given patching problem need not be unique if $U_1 \cup U_2$ is strictly contained in $U$; viz., the solution is not determined on the codimension two complement.  This is illustrated by the following, which draws on Example~\ref{strict containment ex}.

\begin{ex}
In the notation of Example~\ref{strict containment ex}, let $U_0 = U_1 \cap U_2$, and let $M_e=\wh R_{U_e}$ for $e=0,1,2$, with associated isomorphism $\alpha_e:M_e \otimes_{\wh R_{U_e}} \wh R_{U_0} \to M_0$ for $e=1,2$.  Then both $M = \wh R_U = k[x,y][[t]]$ and the ideal $I=(x,y) \subset \wh R_U$ are solutions to the patching problem, with $I$ strictly contained in $M$.  Here 
$I$ is torsion-free but not reflexive; whereas $M$, which is the module given in Proposition~\ref{patch 2}, is reflexive (in fact, free).
Note also that the $\wh R_U$-module $M = \alpha_1\iota_1(M_1) \cap \alpha_2\iota_2(M_2) \subseteq M_0$ 
is equal to $\wh f_*\ms N(U)$, where $\ms N$ is the coherent sheaf on $\mf W$ associated to the given patching problem, with $\mf W$ the formal scheme associated to $W = U_1 \cup U_2$ and with 
$f:W \hookrightarrow U$ the natural inclusion.
\end{ex}

More generally, we have the following result.

\begin{cor} \label{max soln pp}   
In the notation as above and in Proposition~\ref{patch 2},
suppose that $\ms X$ is quasi-projective over $T$, and consider the patching problem given by 
$(M_0,M_1,M_2;\alpha_1,\alpha_2)$.  
\begin{enumerate} [(a)]
\item \label{max tor free}
If the above patching problem has a torsion-free solution, then 
the $\wh R_U$-module $M := \alpha_1\iota_1(M_1) \cap \alpha_2\iota_2(M_2) \subseteq M_0$, together with the maps $\gamma_e$ for $e=1,2$, defines the maximum torsion-free solution to this patching problem.
\item \label{reflexive soln}
If the patching problem has a reflexive solution, then that solution is necessarily isomorphic to the above module $M$ (and so then there is a unique reflexive solution).
\end{enumerate}
\end{cor}

\begin{proof}
Let $\mf U$ and $\mf W$ be the formal completions along $U$ and $W = U_1 \cup U_2$.
Since the patching problem has a torsion-free solution, by Proposition~\ref{pp sheaf bij} there is a torsion-free coherent sheaf~$\ms H$ on $\mf U$ whose pullback to $\mf W$ is the torsion-free coherent sheaf $\ms N$ associated to the patching problem.  By Proposition~\ref{max coh}, the maximal such sheaf $\ms G$ on $\mf U$ is given by $\ms G = \ms N(V)^\Delta$.  Here $\ms N(V) = \alpha_1\iota_1(M_1) \cap \alpha_2\iota_2(M_2) = M$; i.e., $\ms G = M^\Delta$.  By Proposition~\ref{pp sheaf bij}, $M$ is the maximal solution to the patching problem.  This proves part~(\ref{max tor free}).

For part~(\ref{reflexive soln}), suppose that $M'$ is a reflexive solution to the given patching problem.  
Thus $M'^\Delta$ is a reflexive coherent sheaf on $\mf U$,
and $\wh f^* (M'^\Delta) = \ms N$.
By Corollary~\ref{reflexive pp cor},  $M'^\Delta = \wh f_*\wh f^* (M'^\Delta) = 
\wh f_*\ms N$,
so $M' = M'^\Delta(U)=\wh f_*\ms N(U) = \ms N(V) =M$, and $M$ is reflexive.
\end{proof}

Note that the maximality condition in Corollary~\ref{max soln pp}(\ref{max tor free})
can fail without the torsion-free hypothesis.
For example, any torsion $\wh R_U$-module that is supported on the complement of $W$ in $U \subset \Spec(\wh R_U)$ 
is a solution to the trivial patching problem (i.e., the one defined by the zero modules over the rings 
$\wh R_{U_e}$).  This observation is consistent with Corollary~\ref{max soln pp}(\ref{reflexive soln}), since every reflexive module is torsion-free.

In the situation discussed prior to Proposition~\ref{pp sheaf bij}, for each $n \ge 1$ we may reduce the given patching problem $(M_0,M_1,M_2;\alpha_1,\alpha_2)$ modulo $t^n$.  We thus obtain a patching problem $(M_{0,n},M_{1,n},M_{2,n};\alpha_{1,n},\alpha_{2,n})$ for the reductions 
modulo $t^n$; i.e., over the rings $R_n = \wh R_U/(t^n)$ and $R_{e,n} = \wh R_{U_e}/(t^n)$.
We say that a system of solutions $(N_n)_{n\ge 1}$ to these respective patching problems 
is {\it compatible} if for each $n$ 
the isomorphisms $M_{e,n+1}\otimes_{R_{n+1}} R_n \to M_{e,n}$, for $e=0,1,2$, restrict to a common isomorphism $N_{n+1} \otimes_{R_{n+1}} R_n \to N_n$.

We say that a patching problem $(M_0,M_1,M_2;\alpha_1,\alpha_2)$ is {\it reflexive} (resp.\ {\it locally free}) if each $M_e$ is a module with that property.

\begin{prop} \label{pp reductions}
Suppose we are given a reflexive patching problem in the situation discussed prior to Proposition~\ref{pp sheaf bij}, and assume that $\ms X$ satisfies Serre's condition $S_3$ (e.g., $\ms X$ is Cohen-Macaulay).  
\begin{enumerate} [(a)]
\item \label{reflex pp}
If the reductions of the patching problem mod $t^n$ have compatible reflexive solutions for all sufficiently large $n$, then the given patching problem has a reflexive solution.
\item \label{loc free pp}
If the given patching problem is locally free, then it has a locally free solution if and only if the reductions of the patching problem mod $t^n$ have locally free solutions for all sufficiently large $n$ (or equivalently, for all $n$).
\end{enumerate}
\end{prop}

\begin{proof}
With notation as in the above discussion, write $R_n = \wh R_U/(t^n)$, $R_{e,n} = \wh R_{U_e}/(t^n)$, and $M_{e,n} = M_e/t^nM_e$, with associated maps $\iota_{e,n}, \alpha_{e,n}$ induced by $\iota_e, \alpha_e$ modulo $t^n$.  Let $\ms F$ be the coherent sheaf on $\mf W$ given by $M_e^\Delta$ on $\mf U_e$ for $e=0,1,2$, together with the isomorphisms on $\mf U_0$ as in the patching problem.  Let $W_n$ and $U_n'$ be the reductions of $\mf W$ and $\mf U$ modulo $t^n$, with inclusion $f_n:W_n \to U_n'$.  Thus $U_n' = \Spec(R_n)$.  Write $F_n$ for the reduction of $\ms F$ modulo $t^n$; this is a coherent sheaf on $U_n'$, and it is the sheaf defined by the coherent sheaves $\til M_{e,n}$ on $R_{e,n}$ and the maps induced by $\alpha_{e,n}$.

The assertion of the proposition is trivial if $\ms X$ is supported over the closed point of $\Spec(T)$; so we may assume otherwise.  Since $\ms X$ is integral, it follows that the uniformizer~$t$ of $T$ is not a zero-divisor in the local rings of $\ms X$, and hence those of $\wh R_U$.  Also, the localizations $\wh R_{U,\frak p}$ of $\wh R_U$ satisfy  ${\rm depth}(\wh R_{U,\frak p}) \ge \min(3,\dim(\wh R_{U,\frak p}))$ since $\ms X$ and hence $\wh R_U$ is $S_3$.  Identifying a prime ideal $\frak p \in \Spec(\wh R_U)$ containing $t$ with its pullback to $\Spec(R_n)$, we have 
${\rm depth}(R_{n,\frak p}) = {\rm depth}(\wh R_{U,\frak p}) - 1 \ge \min(2,\dim(\wh R_{U,\frak p})-1) = \min(2,\dim(R_{n,\frak p}))$; i.e., $R_n$ (or equivalently, $U_n'$) satisfies the~$S_2$ condition.

If $H_n$ is any reflexive sheaf on $U_n'$, then $H_n$ satisfies the $S_2$ condition since $U_n'$ satisfies~$S_2$ (see \cite[Partie~2, Th\'eor\`eme~5.10.5]{EGA4}).  
That is, for each point $P$ of $U_n'$ we have ${\rm depth}((H_n)_P) \ge \min(2,\dim(\mc O_{U_n',P}))$.  Since the complement $Z_n$ of $W_n$ in $U_n'$ has codimension at least two, 
$\dim(\mc O_{U_n',P}) \ge 2$ for every $P \in Z_n$; and so ${\rm depth}((H_n)_P) \ge 2$ for each such point.  That is, ${\rm depth}_Z(H_n) := \inf_{P \in Z} {\rm depth}_{\mc O_{U_n',P}}((H_n)_P) \ge 2$.  By Th\'eor\`eme~5.10.5 in \cite[Partie~2]{EGA4}, this inequality implies that $H_n$ is $Z$-closed; i.e., the natural map $H_n \to (f_n)_*f_n^* H_n$ is an isomorphism (see \cite[Partie~2, Section~5.9]{EGA4}, especially D\'efinition~5.9.9, for a discussion of the $Z$-closed condition).  So if in addition $f_n^*H_n=F_n$, then we get an identification $H_n = (f_n)_* F_n$.  Thus, if some reflexive sheaf on $U_n'$ pulls back to $F_n$, then that sheaf is unique up to isomorphism; viz., it is the sheaf $(f_n)_* F_n$.

Now for part~(\ref{reflex pp}), let $M_n'$ be a reflexive solution to the reduction of the given patching problem modulo $t^n$, for $n\gg 0$. Thus $G_n := \til M_n'$ is a reflexive coherent sheaf on $U_n'$, and $f_n^* G_n = F_n$.  
So by the previous paragraph, $G_n = (f_n)_*F_n$.
Since the solutions $M_n'$ to the reduced patching problems are compatible, the sheaves $G_n$ are also compatible; i.e., 
$\omega_n^*(f_{n+1})_*F_{n+1} = (f_n)_*F_n$ for each $n$, where $\omega_n:U_n \to U_{n+1}$ is the natural inclusion.
Hence by Theorem~\ref{coherent conditions}(\ref{good reflexive}), it follows that 
$\wh f_*\ms F$ is coherent and reflexive on $\mf U$.  Since $\wh f^*\wh f_*\ms F = \ms F$, 
the assertion of part~(\ref{reflex pp}) now follows from Proposition~\ref{pp sheaf bij}.

For the reverse direction of part~(\ref{loc free pp}), assume that the given patching problem is locally free and that for $n \gg 0$ the reductions modulo $t^n$ have locally free solutions $N_n$ (which are not assumed to be compatible).  Since $N_n$ is a solution to the reduced patching problem, the locally free sheaf $H_n := \til N_n$ on $U_n'$ pulls back to $F_n$, the sheaf on $W_n$ defined by the reduced patching problem.  By the third paragraph of the proof (which applies since locally free sheaves are reflexive), we have that $H_n = (f_n)_*F_n$.  
Thus $(f_n)_*F_n$ is locally free.  By Theorem~\ref{coherent conditions}(\ref{loc free coh}), $\wh f_*\ms F$ is coherent and locally free on $\mf U$, and so the reverse direction of part~(\ref{loc free pp}) again follows by Proposition~\ref{pp sheaf bij}.

For the forward direction of part~(\ref{loc free pp}), suppose that the given patching problem has a locally free solution.  This solution corresponds to a locally free sheaf $\ms G$ on $\mf U$ such that $\wh f^*\ms G = \ms F$, in the above notation.  Thus $\wh f_*\ms F = \wh f_*\wh f^*\ms G = \ms G$ by Corollary~\ref{reflexive pp cor}.  Hence $\wh f_*\ms F$ is coherent and locally free on $\mf U$, and so $(f_n)_*F_n$ is locally free on $U_n'$ for all $n$, by Theorem~\ref{coherent conditions}(\ref{loc free coh}).  This sheaf is then of the form $\til N_n$ for some finite locally free $R_n$-module $N_n$, and 
$f_n^*\til N_n = f_n^*(f_n)_*F_n = F_n$.
Hence $N_n$ is a locally free solution to the 
mod $t^n$ reduction of the given patching problem for all $n$ (and in particular, for all sufficiently large $n$).
\end{proof}

Note that we do not need to require an assumption of compatibility in part~(\ref{loc free pp}) of the above result, because there is no such assumption in Theorem~\ref{coherent conditions}(\ref{loc free coh}).  

Meanwhile, by using Corollary~\ref{Cech H1 cor}, we obtain the following criterion in the reflexive case, with no compatibility assumption:

\begin{prop}
In the context of the discussion before Proposition~\ref{pp sheaf bij},
a reflexive patching problem $(M_0,M_1,M_2;\alpha_1,\alpha_2)$ has a reflexive solution provided that
\[tM_0 \cap (M_1 + M_2) = t(M_1 + M_2) \subseteq M_0,\]
where for $e=1,2$ we identify $M_e$ with its image in $M_0$ via $\alpha_e\iota_e$.
\end{prop}

\begin{proof}
With notation as in the discussion before Proposition~\ref{pp sheaf bij}, 
let $\ms N$ be the coherent sheaf on $\mf W$ that is associated to the given patching problem.  By Corollary~\ref{Cech H1 cor}, $\wh f_*\ms N$ is coherent and reflexive on $\mf U$.  Since $\wh f^*\wh f_*\ms N = \ms N$, it follows from Proposition~\ref{pp sheaf bij} that there is a reflexive solution $S$ to the patching problem.
\end{proof}

\noindent{\bf Author Information:}

\smallskip
 
\noindent David Harbater\\
Department of Mathematics, University of Pennsylvania, Philadelphia, PA 19104-6395, USA\\
email: harbater@math.upenn.edu

\smallskip

\noindent Julia Hartmann\\
Department of Mathematics, University of Pennsylvania, Philadelphia, PA 19104-6395, USA\\
email: hartmann@math.upenn.edu

\smallskip

\noindent Daniel Krashen\\
Department of Mathematics, University of Pennsylvania, Philadelphia, PA 19104-6395, USA\\
email: dkrashen@math.upenn.edu

\smallskip

\noindent The authors were supported on NSF grants DMS-1805439, DMS-2102987, and DMS-2402367 (DH and JH), and by NSF grants DMS-1902144, DMS-2049180, and DMS-2401018 (DK).


\begin{thebibliography}{KuSh21}

\bibitem[AM69]{AM}
Michael F.~Atiyah and Ian G.~Macdonald. 
\newblock {\it Introduction to commutative algebra}. 
\newblock Addison-Wesley Publishing Co., Reading, Mass.-London-Don Mills, Ont.~1969.

\bibitem[BLR90]{BLR}
Siegfried Bosch, Werner L\"utkebohmert, and Michel Raynaud.
\newblock N\'eron models.
\newblock Ergeb. Math. Grenzgeb. (3), 21, Springer-Verlag, Berlin, 1990. 

\bibitem[Bou72]{Bo:CA} 
Nicolas Bourbaki.
\newblock Commutative Algebra.
\newblock Hermann and Addison-Wesley, 1972.

\bibitem[Eis95]{Eis}
David Eisenbud.
\newblock {\it Commutative algebra.
With a view toward algebraic geometry.}
\newblock Grad.\ Texts in Math., vol.~150,
Springer-Verlag, New York, 1995. 

\bibitem[EGA1]{EGA1}
Alexander Grothendieck.
\newblock \'El\'ements de g\'eom\'etrie alg\'ebrique. I. Le langage des sch\'emas.
\newblock Publ.\ Math., Inst.\ Hautes \'Etudes Sci.\ (1960), no.~4.

\bibitem[EGA3]{EGA3}
Alexander Grothendieck.
\newblock \'El\'ements de g\'eom\'etrie alg\'ebrique. III. \'Etude cohomologique des faisceaux coh\'erents.
\newblock Publ.\ Math., Inst.\ Hautes \'Etudes Sci.  Partie 1 (1961), no.~11.

\bibitem[EGA4]{EGA4}
Alexander Grothendieck.
\newblock \'El\'ements de g\'eom\'etrie alg\'ebrique. IV. \'Etude locale des schemas et des morphismes de schemas.
\newblock Publ.\ Math., Inst.\ Hautes \'Etudes Sci.  Partie 1 (1964), no.~20.
Partie 2 (1965), no.~24. Partie 4 (1967), no.~32.

\bibitem[Ha94]{Ha94}
David Harbater.
\newblock Abhyankar's conjecture on Galois groups over curves. 
Invent.\ Math.\ {\bf 117} (1994), 1--25.

\bibitem[HH10]{HH:FP}
David Harbater and Julia Hartmann.
\newblock Patching over fields.
\newblock Israel J.\ Math.\ \textbf{176} (2010), 61--107.

\bibitem[HHK09]{HHK}
David Harbater, Julia Hartmann, and Daniel Krashen. 
\newblock Applications of patching to quadratic forms and central simple algebras. 
\newblock Invent.\ Math.\ {\bf 178} (2009), no.~2, 231--263.

\bibitem[HHK15]{HHK:H1}
David Harbater, Julia Hartmann, and Daniel Krashen.
\newblock  Local-global principles for torsors over arithmetic curves.
\newblock American Journal of Mathematics, {\bf 137}  (2015), 1559--1612.

\bibitem[Hts77]{Hts}
Robin Hartshorne.
\newblock {\em Algebraic Geometry}.
\newblock Graduate Texts in Mathematics.
\newblock Springer-Verlag, 1977.

\bibitem[Hts80]{Hts80}
Robin Hartshorne.
\newblock Stable reflexive sheaves.
\newblock Math.\ Ann.\ {\bf 254} (1980), no.~2, 121--176.

\bibitem[KuSh21]{KuSh}
Kazuhiko Kurano and Kazuma Shimomoto. 
\newblock Ideal-adic completion of quasi-excellent rings (after Gabber). 
\newblock Kyoto J.\ Math.\ {\bf 61} (2021), no.~3, 707--722.

\bibitem[Liu02]{Liu}
Qing Liu.  
\newblock {\it Algebraic geometry and arithmetic curves}. 
\newblock Translated from the French by Reinie Ern\'e. Oxford Graduate Texts in Mathematics, vol.~6. Oxford University Press, Oxford, 2002.

\bibitem[Mat80]{Mats}
Hideyuki Matsumura. 
\newblock {\it Commutative algebra}. Second edition. Mathematics Lecture Note Series, vol.~56. 
\newblock Benjamin/Cummings Publishing Co., Inc., Reading, Mass., 1980.

\bibitem[Mat89]{Mats89}
Hideyuki Matsumura. 
\newblock {\it Commutative ring theory}. Second edition.  Cambridge Stud. Adv. Math., vol.~8.
\newblock Cambridge University Press, Cambridge, 1989. 

\bibitem[OSS11]{OSS}
Christian Okonek, Michael Schneider, and Heinz Spindler.
\newblock {\it Vector bundles on complex projective spaces}.
Corrected reprint of the 1988 edition. With an appendix by S.I.~Gelfand.
\newblock Mod. Birkh\"auser Class.
\newblock Birkh\"auser/Springer Basel AG, Basel, 2011.

\bibitem[Sch85]{Schen}
Peter Schenzel.
\newblock Symbolic powers of prime ideals and their topology.
\newblock Proc.\ Amer.\ Math.\ Soc.\ {\bf 93} (1985), 15--20.

\bibitem[Sta26]{stacks}
The Stacks project authors.
\newblock The Stacks project.
\newblock \url{https://stacks.math.columbia.edu}, 2026.

\bibitem[Vas68]{vasc}
Wolmer V.~Vasconcelos. 
\newblock Reflexive modules over Gorenstein rings.
\newblock Proc.\ Amer.\ Math.\ Soc.\ {\bf 19} (1968), 1349--1355.

\end{thebibliography}
\end{document}